\documentclass[10pt,reqno]{amsart}

\usepackage[margin=1in]{geometry}   
\usepackage[utf8]{inputenc}  
\usepackage[T1]{fontenc}
\usepackage{upgreek}

\usepackage{amsmath,amsthm,amsfonts,amssymb,amscd}
\usepackage{mathtools}
\usepackage[mathscr]{euscript}
\usepackage{bm}
\usepackage{bbm}
\usepackage{stmaryrd}

\usepackage[dvipsnames]{xcolor}
\usepackage{soul}

\usepackage[square,numbers]{natbib}

\usepackage{etoolbox,xstring}
\makeatletter
\patchcmd{\NAT@test}{\else\NAT@nm}{\else\NAT@nmfmt{\NAT@nm}}{}{}
\let\NAT@up\scshape
\makeatother

\makeatletter
\renewcommand{\NAT@nmfmt}{\expandafter\aliNAT@nmfmt\expandafter}
\newcommand\aliNAT@nmfmt[1]{{%
  \noexpandarg
  \def~{}%
  \edef\temp#1\edef\temp{\detokenize\expandafter{\temp}}%
  \begingroup\edef\x{\endgroup
    \noexpand\StrSubstitute{\temp}{\detokenize{etal}}}\x
    {\textnormal{et\nobreakspace al}}[\tempa]%
  \textsc{\tempa}}}
\makeatother

\usepackage{graphicx}
\graphicspath{{graphics/}}

\usepackage{wrapfig}
\usepackage{subcaption}
\usepackage{tikz}
\usetikzlibrary{arrows.meta,calc}
\usetikzlibrary{decorations.pathreplacing,angles,quotes,patterns}
\usepackage{pgfplots}
\pgfplotsset{compat=1.18}
\usepgfplotslibrary{fillbetween}

\usepackage{array}
\newcolumntype{P}[1]{>{\centering\arraybackslash}p{#1}}		
\newcolumntype{M}[1]{>{\centering\arraybackslash}m{#1}}
\usepackage{multirow}
\usepackage{booktabs}

\usepackage{float}
\usepackage{enumitem}
\usepackage{algorithm}
\usepackage{verbatim}
\usepackage{cases}
\usepackage{needspace}

\usepackage{siunitx}
\usepackage{hyperref}
\hypersetup{colorlinks=True, linkcolor=blue, urlcolor=blue, citecolor=red}
\usepackage{cleveref}

\numberwithin{equation}{section}

\newtheorem{theorem}{Theorem}[section]
\newtheorem{lemma}[theorem]{Lemma}
\newtheorem{proposition}[theorem]{Proposition}

\newtheorem{definition}[theorem]{Definition}

\newtheorem{remark}[theorem]{Remark}

\def\tmax{{t}_{\mathsf{max}}}
\def\dd#1{{\, \mathrm{d} #1}}

\def\p{\partial}

\newcommand{\thetainv}{\tilde\tau} 

\newcommand{\Lop}{\mathcal{L}} 

\newcommand{\amp}{\mathcal{A}} 

\newcommand{\Gop}{\mathcal{G}}

\newcommand{\Gopo}{\mathcal{G}_0}
\newcommand{\Kopo}{\mathcal{K}_0}

\newcommand{\Sop}{\mathcal{S}} 
\newcommand{\SopSL}{\mathcal{S}_{\mathsf{SL}}}
\newcommand{\Dop}{\mathcal{D}} 
\newcommand{\DopDL}{\mathcal{D}_{\mathsf{DL}}}
\newcommand{\Doppr}{\overline{\mathcal{D}}} 
\newcommand{\DopDLpr}{\overline{\mathcal{D}}_{\mathsf{DL}}}

\newcommand{\NN}{\nu} 
\newcommand{\gammach}{\gamma_{\mathsf{ch}}}

\newcommand{\sltw}{\Uptheta} 
\newcommand{\slstw}{\mathcal{V}} 
\newcommand{\dltw}{\Uplambda} 
\newcommand{\dlstw}{\mathcal{W}} 

\newcommand{\Pbb}{\mathbb{P}}

\newcommand{\phiinc}{\phi_{\mathsf{inc}}}

\newcommand{\thit}{\tau_{\mathsf{rfl}}}
\newcommand{\Tarr}{T}
\newcommand{\Trfl}{T_{\mathsf{rfl}}}

\newcommand{\loc}{\mathsf{loc}} 
\newcommand{\reg}{\mathsf{reg}}

\title[Scattering in heterogeneous media with moving obstacles]{A boundary integral method for wave scattering in a heterogeneous medium with a moving obstacle}
\author{Raaghav Ramani}
\address{Center for Nonlinear Studies, Los Alamos National Laboratory, Los Alamos, NM 87545}
\email{\href{rramani@lanl.gov}{rramani@lanl.gov}}

\begin{document}

\begin{abstract}
We propose a time-domain boundary integral method to model linear 
wave propagation with refractive, focusing, and Doppler effects 
arising from medium heterogeneities and moving obstacles. 
In contrast to existing techniques, our method avoids volumetric discretization 
and yields a formulation posed only on the boundary of the obstacle.
We combine two classical ingredients: a \emph{geometric--optics parametrix} to capture 
leading-order behavior at propagating wavefronts, and a \emph{ray-based characterization} 
of the distorted causal geometry. 
The former provides a framework for defining layer potentials and deriving the associated 
boundary integral equations, while the latter enables a pure boundary-only formulation.
Taken together, these ingredients extend existing numerical techniques for the 
homogeneous, fixed-boundary case to the heterogeneous, moving-boundary setting, with appropriate modifications to capture the 
discrete causal structure arising from the intersection of distorted light cones with the worldsheet of the moving boundary.
Numerical experiments demonstrate the ability of the method to resolve Doppler 
effects from moving obstacles, including a rotating 
turbine configuration, with stable performance up to Mach $0.9$.
In heterogeneous media, the method captures strong refractive effects 
from spherical inclusions: wave propagation wrapping around a gas 
bubble in water, and defocusing from a hot fireball rising through a stratified atmosphere.
\end{abstract}

\maketitle

\setcounter{tocdepth}{1}
{\small
\tableofcontents}

\section{Introduction}
\label{sec:intro}

Let $\Gamma(t)$ denote a moving interface enclosing an obstacle 
$\Omega^-(t)$ in a heterogeneous acoustic medium 
with sound speed $c(x,t)$. Our primary interest is the scattered wave field in the exterior domain
$\Omega^+(t) = \mathbb{R}^3 \setminus \Omega^-(t)$. 
Let $\Lop$ denote the \emph{exterior wave operator} 
\begin{equation}\label{L-general-pde}
  \Lop \phi (x,t)
   \coloneqq 
  \tfrac{1}{c(x,t)^2}
  \partial_t^2 \phi (x,t)
  -
  \Delta_x \phi(x,t), \qquad (x,t) \in \Omega^+ (t) \times \mathbb{R}^+, 
\end{equation}
where the space-time--dependent sound speed field $c(x,t)$ is assumed to be smooth, positive, and uniformly bounded,
\begin{equation}\label{c-bounds}
0 < c_{\min} \le c(x,t) \le c_{\max}  < \infty, \qquad x \in \mathbb{R}^3 \ \text{and} \ t \ge 0. 
\end{equation}
The operator \eqref{L-general-pde} is linear and hyperbolic. It  
is supplemented with the initial conditions
\begin{subequations}\label{L-bcs}
\begin{equation}
\phi(x,0) = 0
\qquad \text{and} \qquad
\partial_t \phi(x,0) = 0, 
\qquad x \in \Omega^+ (0). 
\end{equation}
In this work, we consider the \emph{sound-soft} Dirichlet conditions
\begin{equation}
\phi(x,t) = f(x,t), 
\qquad (x,t) \in \Gamma(t) \times \mathbb{R}^+ , 
\end{equation}
\end{subequations}
where $f : \Gamma(t) \times \mathbb{R}^+ \to \mathbb{R}$ is the prescribed boundary data. 
The system \eqref{L-general-pde}--\eqref{L-bcs} arises in numerous applications 
with wave propagation in heterogeneous media (see, e.g., \cite{Ostashev1997} and the references therein). 

A model case is a smooth sound-speed profile
encoding an interior--exterior contrast across the interface. 
For $\delta>0$,
define the tubular neighborhood
\begin{subequations}\label{c-piecewise-smooth}
\begin{equation}\label{tubular-neighborhood}
  \Omega_\delta (t)
  \coloneqq
  \{\,x\in\mathbb{R}^3 : \operatorname{dist}(x,\Gamma)<\delta\,\}, 
\end{equation}
and an associated smooth sound speed field $c\in C^\infty(\mathbb{R}^3\times\mathbb{R}^+)$ by
\begin{equation}\label{c-smooth-two-phase}
  c(x,t)=c^-(x,t) \quad \text{for } x\in \Omega^-(t) \setminus \Omega_\delta(t),
  \qquad
  c(x,t)=c^+(x,t) \quad \text{for } x\in \Omega^+(t) \setminus \Omega_\delta(t),
\end{equation}
\end{subequations}
where $c^\pm(x,t)>0$ are smooth functions. 
In the tube $\Omega_\delta(t)$, the sound speed varies smoothly across the interface in the normal direction, in the sense that
$c(x,t) = c_\Gamma(d(x,t),x,t)$, where $d(x,t)$ denotes the signed distance to 
$\Gamma(t)$, and for each $(x,t)$ the map $s \mapsto c_\Gamma(s;x,t)$
is monotone on $(-\delta,\delta)$ and satisfies $c_\Gamma(-\delta;x,t)=c^-(x,t)$ and 
$c_\Gamma(\delta;x,t)=c^+(x,t)$. 
Such profiles arise, for example, in atmospheric acoustics, where localized 
temperature or density variations---such as thermal plumes or inversion layers---give rise to moving interfaces across 
which the effective sound speed varies smoothly but rapidly.
We will refer to the case $\min c^- > \max c^+$ as a \emph{fast inclusion}, and to the
case $\max c^- < \min c^+$ as a \emph{slow inclusion}.

\begin{figure}[ht]
    \centering
    \begin{subfigure}[t]{0.45\linewidth}
      \centering
      \includegraphics[width=0.77\textwidth]{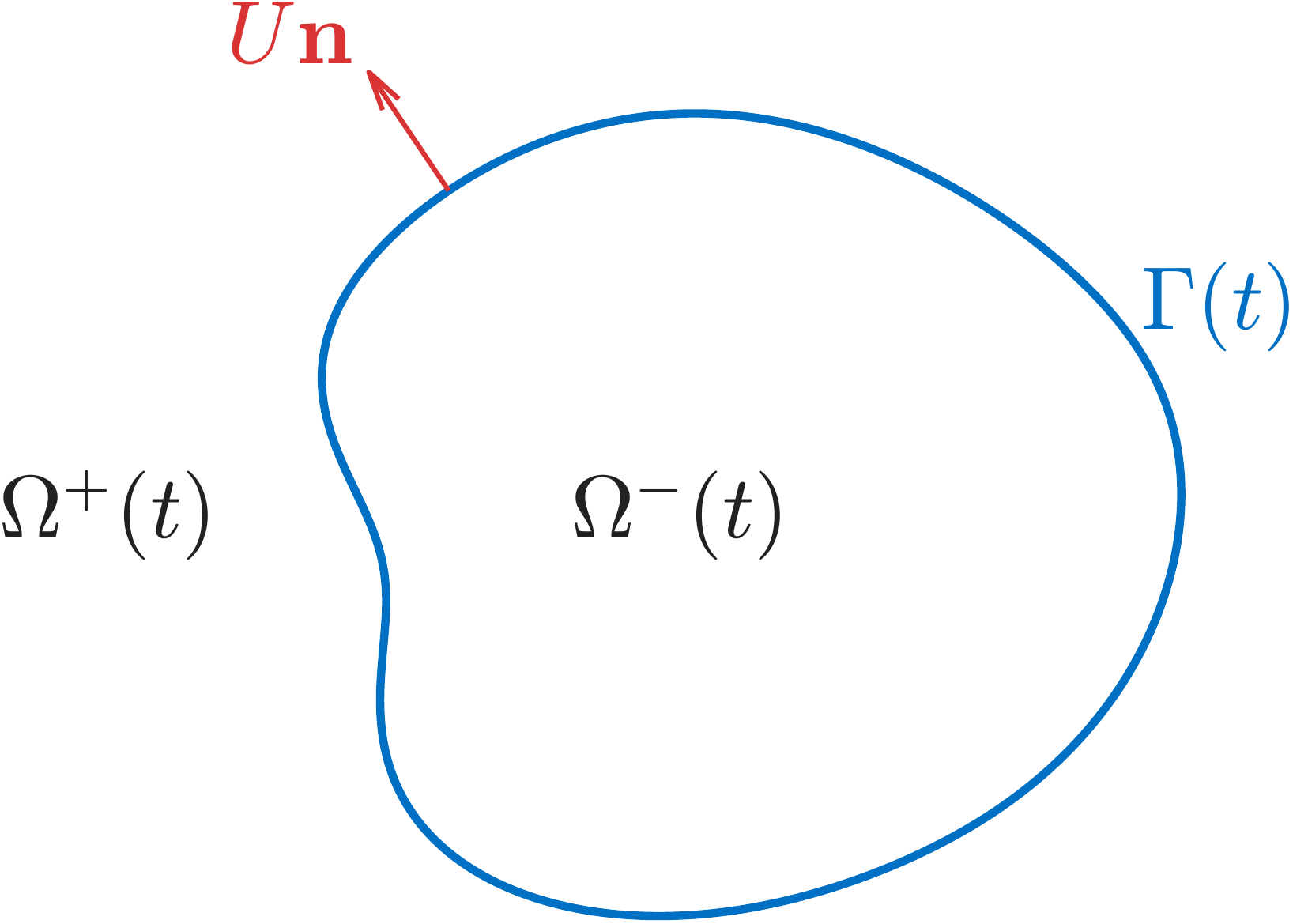}
      \caption{}
      \label{fig:basic-schematic}
    \end{subfigure}
    \hspace{1em}
    \begin{subfigure}[t]{0.45\linewidth}
      \centering
      \includegraphics[width=0.77\textwidth]{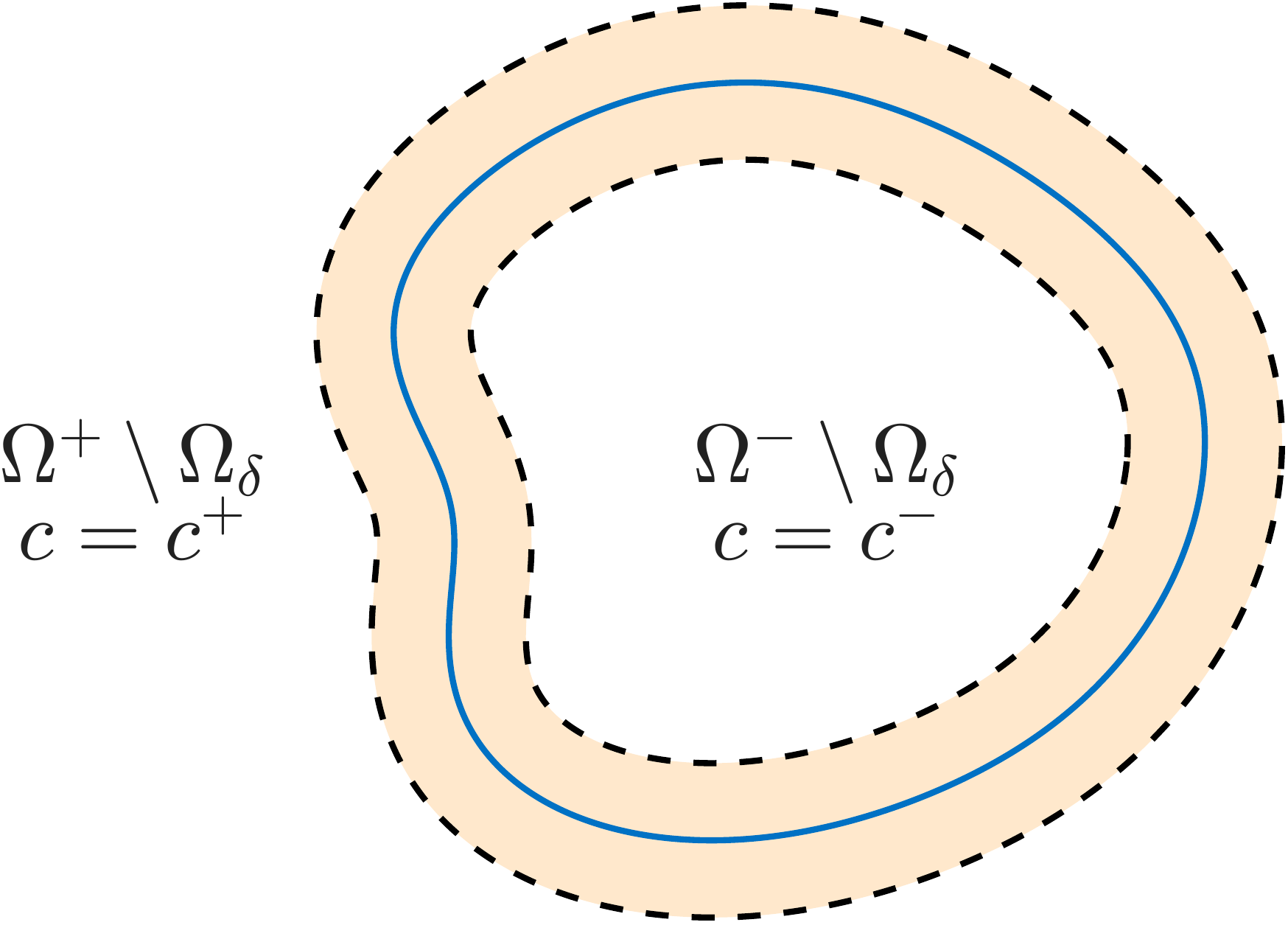}
      \caption{}
      \label{fig:sound-spd-schematic}
    \end{subfigure}
    \caption{
    \textbf{Left:}
    Two-dimensional schematic of the moving interface $\Gamma(t)$ separating the interior domain 
    $\Omega^-(t)$ from the exterior domain $\Omega^+(t)$, with outward normal velocity $U_n$.
    \textbf{Right:}
    Smooth interior--exterior sound-speed contrast, in which the sound speed 
    transitions across a thin interfacial layer $\Omega_\delta$ between the 
    interior and exterior fields $c^-(x,t)$ and $c^+(x,t)$ in $\Omega^-(t)$ and $\Omega^+(t)$, respectively.
    }
    \label{fig:sound-speed-schematic}
\end{figure}

Numerical methods for \eqref{L-general-pde}--\eqref{L-bcs} include well-established
volume-based discretizations, such as the 
finite difference \cite{Taflove1988,Taflove2005,ZhLe1997,LoPi2004,SeoMit2011} and
finite element \cite{FrPe1996,LeeLeeCan1997,GrScSc2006} methods,
typically combined with non-reflecting
boundary conditions \cite{EnMa1977,Be1994} imposed at the outer boundary of the artificially truncated exterior domain.
In the special case of a homogeneous medium with constant sound speed
$c(x,t) \equiv c_0$ and a fixed interface $\Gamma(t) \equiv \Gamma$,
the problem \eqref{L-general-pde}--\eqref{L-bcs} reduces to the classical
exterior Dirichlet scattering problem for a fixed sound-soft obstacle.
In this setting, the fundamental solution (Green's function) $\Gopo$ of the wave operator is known
explicitly, and the solution $\phi(x,t)$ may be represented by retarded
layer potentials obtained by convolving $\Gopo$ with an
unknown boundary density defined on $\Gamma$. 
Within this \emph{boundary integral} framework \cite{Sayas2013,Costabel2004,Ha2003},
the equation for the unknown density is posed on the surface $\Gamma$,
leading to a co-dimension--1 formulation of the problem. 
Boundary integral methods automatically enforce the correct 
far-field behavior of $\phi$, and offer geometric
flexibility and computational efficiency by avoiding volumetric meshing of
the exterior domain.\footnote{At the same time,
time-domain boundary integral formulations involve convolution in time,
and straightforward implementations require storage of the full temporal
history. Efficient implementations rely on fast convolution and
acceleration techniques \cite{BaSa2009,ScLoLu2006,BrKu2001,Rokhlin1990,ErShMi1998}.}
More generally, boundary integral methods naturally arise as components in hybrid 
Eulerian--Lagrangian schemes, where interface dynamics are treated in a Lagrangian 
framework while the bulk is  
handled by a complementary Eulerian discretization \cite{Peskin1972,Steinbach2008,AbJoRoTe2011,BaLuSa2015,RaSh2020}.

In heterogeneous media, an explicit retarded Green's function for the wave 
operator is not available, so classical boundary integral 
methods cannot be applied directly. 
The light cone, which is radially symmetric and straight 
in the homogeneous case, becomes distorted in variable media (see \Cref{fig:cone-x}).
In certain simplified settings---such as piecewise-constant media 
\cite{CoEr1985,LaSa2009,LuLuQi2018,DoGaSa2020,RiSaMe2022,XiYuWaCa2025} 
or contrast-based formulations built upon a reference Green's function 
\cite{CoKr1998,Vainikko2000,Martin2003,Bruno2005}---modified boundary integral methods can be applied.
Even in these settings, volumetric discretization is often required, sacrificing the  
co-dimension--1 feature of the classical boundary integral methods. 

In this work, we develop a boundary integral method for wave scattering in 
space-time heterogeneous media with subsonic moving obstacles
that does \emph{not} require any volumetric discretization.
Our method is based on a \emph{geometric--optics parametrix}
\cite{EnRu2003,AyDeMi2024} for the variable-coefficient wave operator. 
In this approach, the retarded fundamental solution $\Gop$ is approximated by 
the reparameterized kernel
\[
\Gop \approx \amp \, \Gop_0(\eta,\theta),
\]
where $(\eta,\theta)$ encode the \emph{travel time} 
geometry of propagating waves and $\amp$ is a smooth amplitude.   
The function $\eta$ solves an associated space-time 
eikonal equation, and serves locally as a nonlinear change of variables 
that rectifies the deformed light cone in travel-time coordinates (see \Cref{fig:cone-x}). 
The pair $(\eta,\amp)$ may be computed using numerical ray tracing methods that remain
valid in heterogeneous media and beyond caustics, where multiple propagation paths 
arise \cite{UmTh1987,SaKe1990,ViIvGj1993}.
In this work, motivated by the interior--exterior contrast \eqref{c-piecewise-smooth}, 
we adopt a simpler approximation: $\eta$ is obtained via a 
chord-based approximation---optionally refined by a Newton iteration---for the geometric rays,
while the amplitude $\amp$ is modeled using a simple two-point approximation that captures
the leading sound-speed dependence while neglecting higher-order geometric spreading effects.

\begin{figure}[ht]
\centering
\begin{subfigure}[t]{0.31\linewidth}
  \centering
  \includegraphics[width=0.85\linewidth]{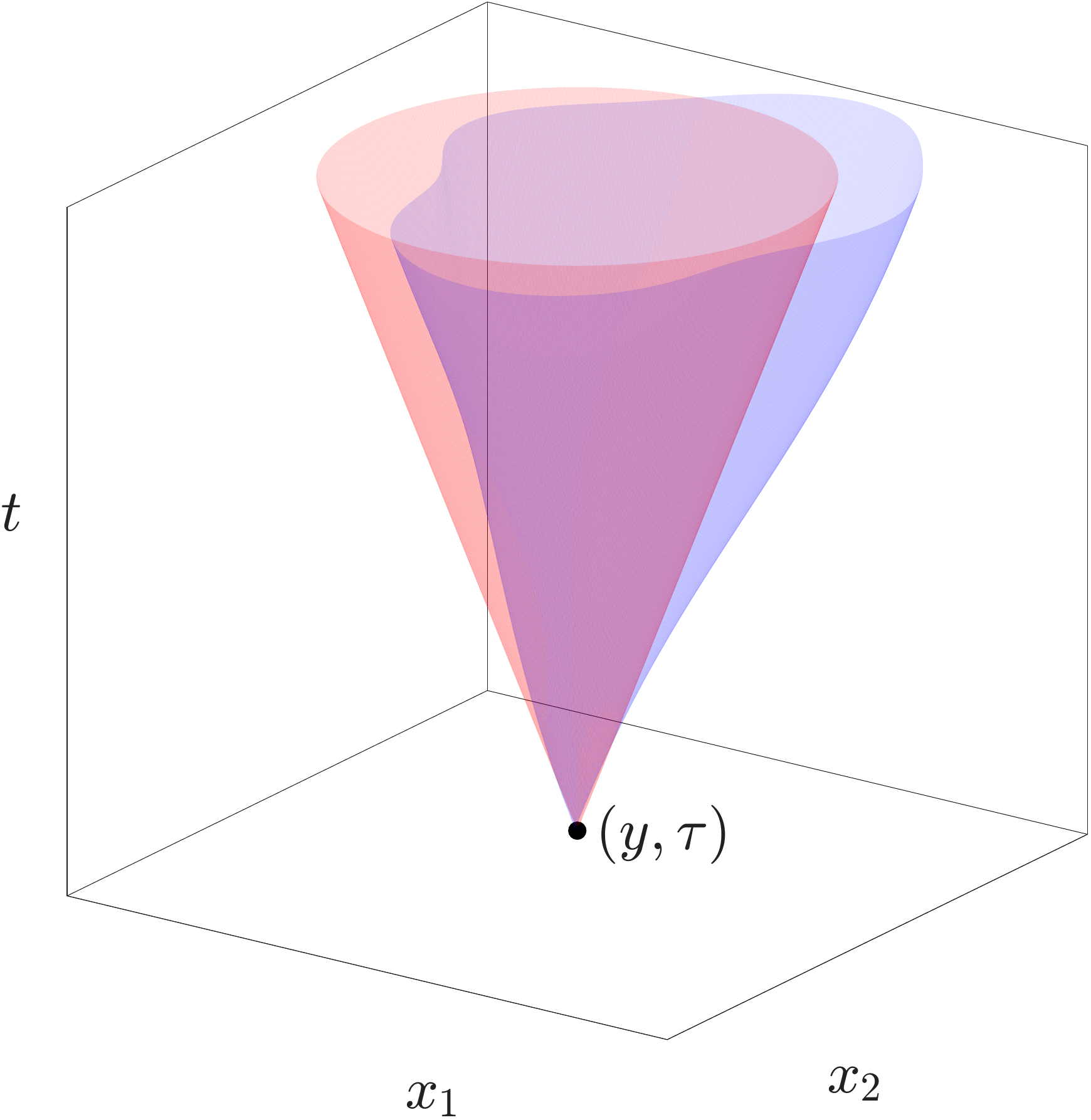}
  \caption{True and source-frozen cones}
  \label{fig:cone-x-3d}
\end{subfigure}
\hspace{1em}
\begin{subfigure}[t]{0.31\linewidth}
  \centering
  \includegraphics[width=0.8\linewidth]{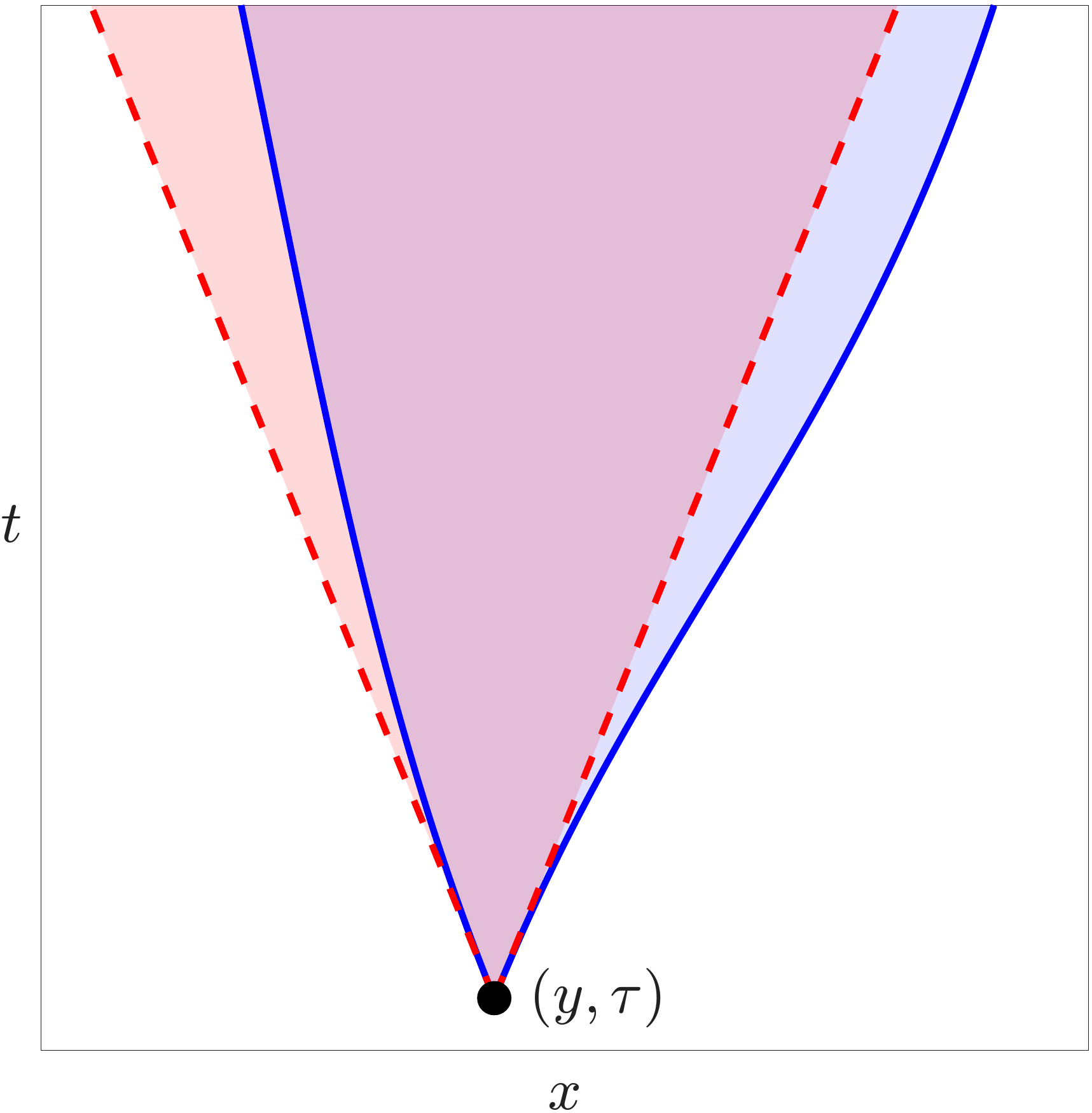}
  \caption{One-dimensional slice}
\end{subfigure}
\hspace{1em}
\begin{subfigure}[t]{0.31\linewidth}
  \centering
  \includegraphics[width=0.8\linewidth]{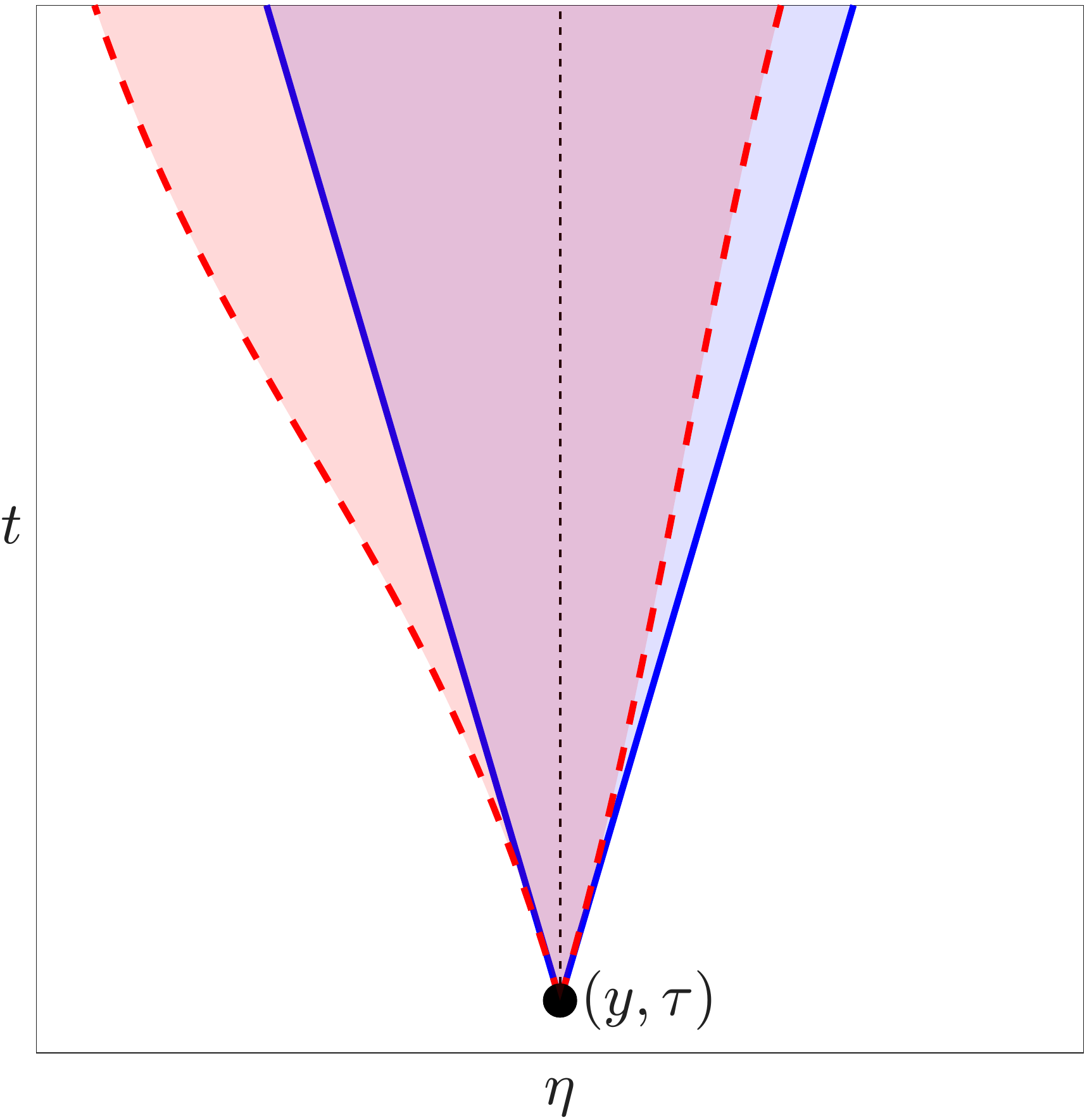}
  \caption{Rectified cone}
\end{subfigure}
\caption{
Geometry of wave propagation in a heterogeneous medium.
The true cone consists of all space-time points reachable from the source $(y,\tau)$
by signals propagating with local speed $c(x,t)$, while the source-frozen cone
assumes the constant emission speed $c(y,\tau)$.
\textbf{Left:} True (blue) and source-frozen (red) cones in $2\!+\!1$--space-time.  
The true cone is given by $t-\tau=\eta(x,t;y,\tau)$, where $\eta$ denotes the 
travel-time function, whereas the source-frozen cone corresponds to 
$|x-y|=c(y,\tau)(t-\tau)$. 
The two cones are tangent at $(y,\tau)$ but diverge as 
time increases due to spatial refraction.
\textbf{Middle:} One-dimensional spatial slice of the cones 
in physical coordinates $(x,t)$.
The true cone (blue) bends, whereas the source-frozen cone (red) remains linear.
\textbf{Right:} The same slice expressed in travel-time coordinates $(\eta,t)$. 
Under the local nonlinear change of variables defined by $\eta$, 
the distorted true cone is rectified into straight lines $\eta=\pm(t-\tau)$.
}
\label{fig:cone-x}
\end{figure}

While geometric--optics approximations for wave propagation are classical 
\cite{Keller1962,Lax1957,Hormander,Friedlander1975,Kravtsov1990}, their systematic 
development within a \emph{time-domain} boundary integral framework for heterogeneous 
media remains relatively unexplored (but see \cite{EnRu2003,ChSiGrLaSp2012} 
and the references therein for related approaches).
Likewise, relatively few boundary integral formulations have been 
developed for wave scattering by moving interfaces. 
Although such problems arise frequently in aeroacoustics, existing approaches 
typically treat the complementary problem of radiation from 
prescribed boundary sources \cite{FWH1969,BrFa1998}. 
A mathematical theory for wave scattering in a homogeneous medium with 
moving obstacles was established in \cite{Strauss1979,Cooper1979,CoSt1979,CoSt1984}. 
Boundary integral methods have been developed in the setting of 
subsonic uniform flows, where the convective effects are 
incorporated through modified Green's functions \cite{WuLe1994,LaWeMa1996,HuPiNa2017}, and, recently, several methods with deformable interfaces have been 
proposed for the forced acoustic wave equation \cite{JSV2024,de2025sound}. 
For moving obstacles, the support of the retarded kernel is determined by the intersection of the light cone with the worldsheet of the boundary, a geometric relation encoded directly in our framework through the travel-time function $\eta$.

Using the explicit parametrix form, we define single- and double-layer potential operators, 
and derive the corresponding boundary integral equations.
Since scattering in heterogeneous media is typically posed as a transmission problem, and each 
layer-potential representation implicitly induces a corresponding interior continuation, 
the residual of the geometric--optics approximation may be viewed, at least in part, as 
reflecting the interior modeling implicit in the chosen representation. 
We prove a trace formula for the double-layer operator on a moving 
boundary, obtaining an explicit jump term whose coefficient 
depends on the space-time geometry of the worldsheet. 
In contrast to the stationary case, where the jump coefficient 
is the constant $\pm \tfrac12$, this space-time--dependent coefficient 
encodes Doppler effects through its dependence on the normal velocity 
of the interface.

To solve the boundary integral equations arising from the parametrix formulation, we 
develop and implement a simple time-marching collocation method\footnote{%
Galerkin methods \cite{HaBaNe1986,YuMaCa2000,DuLuTe2003}, which 
offer enhanced stability, could also be employed; however, they require 
the evaluation of four-dimensional space-time integrals and more 
elaborate implementations, and are therefore not pursued here.}, following 
the ideas introduced in 
\cite{Rao1982,mansur1982,Mansur1983,ErShMi1998,Rynne1990,DaDu1997}. 
The resulting scheme takes the form of a marching-on-time recursion 
with coefficients defined by integrals over each space-time slab.
We introduce a \emph{slab-frozen approximation} that enables 
analytic evaluation of the time integral, leaving only spatial integrals 
to be computed by standard Gauss quadrature.

We then apply the proposed scheme to a collection of 
benchmark and physically motivated test problems to 
illustrate its accuracy, stability, and performance in the 
presence of moving obstacles and space-time heterogeneities.  
Two Doppler scattering examples involving moving obstacles
demonstrate complementary aspects of the method: the first confirms 
stability up to Mach 0.9, while the second illustrates its 
capability to simulate the complex geometry of a rotating turbine.
We then consider a gas-bubble scattering problem to validate the 
ray-based travel-time formulation in a strongly refractive 
medium, demonstrating the refraction of wavefronts around 
a slow spherical inclusion.
Finally, an atmospheric fireball example confirms that the method captures 
refraction by a localized high-speed spherical inclusion, manifested 
through earlier arrival times and reduced peak amplitudes in the scattered field.

\subsubsection*{Outline of the paper}
The remainder of the paper is organized as follows.
In \Cref{sec:universal}, we introduce the general geometric--optics parametrix framework and
summarize its fundamental properties. 
\Cref{sec:BIE} develops the associated boundary integral formulation,
including the definition of single- and double-layer potentials and the derivation of the corresponding boundary integral equations. 
In \Cref{sec:collocation}, we present the space-time collocation scheme used for their numerical solution.
Sections \ref{sec:const}--\ref{sec:space-time-dependent} contain numerical experiments in progressively more complex settings: motion-induced scattering in a homogeneous medium, wave scattering in spatially heterogeneous media, and wave scattering in space-time varying media.
Finally, \Cref{sec:conclusion} summarizes the main findings and outlines directions for future work.
Technical details concerning the trace of the double-layer potential and the Newton method for geometric ray computation are provided in \Cref{app:traceDL} and \Cref{subsec:newton-ray-continuous}, respectively.


\section{Geometric--optics parametrix framework}
\label{sec:universal}

In this section, we outline the basic geometric--optics parametric framework 
for the wave operator \eqref{L-general-pde}.  
A rigorous mathematical treatment can be 
found in, e.g., \cite{Lax1957,Hormander,Eskin2011}. 
Here, the aim is to outline the minimal structure required to formulate the parametrix.
We begin by recalling the description of the causal geometry in both the 
homogeneous and heterogeneous cases, the latter described via travel-time coordinates $(\theta,\eta)$.

\subsection{The unit-speed reference operator and Green's function}

The free-space unit-speed wave operator in $3\!+\!1$ dimensional
space-time coordinates is denoted by
\begin{equation}
\label{L0}
  \Lop_0 \phi(x,t)
  \coloneqq
  \partial_t^2 \phi(x,t)
  -
  \Delta_x \phi(x,t),
  \qquad (x,t)\in\mathbb{R}^3\times\mathbb{R}.
\end{equation}
Let $\Gopo(x,t)$ denote the causal free-space Green's function for the
unit-speed wave operator $\Lop_0$, i.e., the distributional solution of
\begin{equation}
\label{G0-def}
  \Lop_0 \Gopo(x,t)
  =
  \delta(x)\,\delta(t),
  \qquad (x,t)\in\mathbb{R}^3\times\mathbb{R},
\end{equation}
supplemented with the causality condition
\begin{equation}
\label{G0-causal}
  \Gopo(x,t)=0
  \qquad \text{for } t<0.
\end{equation}
Explicitly, writing $r = |x|$, one has
\begin{equation}
\label{G0-explicit}
  \Gopo(r,t)
  =
  \frac{\delta(t-r)}{4\pi r},
  \qquad (r,t)\in\mathbb{R}^+ \times\mathbb{R},
\end{equation}
where $\delta$ denotes the Dirac delta distribution.

In three spatial dimensions, the \emph{strong Huygens principle} is satisfied: the 
Green's function $\Gopo$ is supported on the boundary of the
unit-speed forward causal region
\begin{equation}
\label{G0-cone}
   \Sigma_0 \coloneqq \{\, (x,t)\in\mathbb{R}^3\times\mathbb{R}^+ : t \ge |x| \,\}. 
\end{equation}
The boundary of the causal region is the light cone 
\begin{equation}
\label{G0-cone-bdy}
  \partial \Sigma_0 \coloneqq \{\, (x,t)\in\mathbb{R}^3\times\mathbb{R}^+ : t = |x| \,\}, 
\end{equation}
which corresponds to the unit-speed wavefront
emanating from the source point $(0,0)$.

\subsection{Parametrization of the obstacle and its worldsheet}

\begin{subequations}\label{physical-geometry}
For each $t \ge 0$, let $\Omega^-(t) \subset \mathbb{R}^3$ be a bounded domain
with smooth boundary $\Gamma(t)$, and define the exterior domain
\begin{equation}
\Omega^+(t) = \mathbb{R}^3 \setminus \overline{\Omega^-(t)} \,.
\end{equation}
We assume that for each $t \ge 0$ the interface $\Gamma(t)$ is smooth and
diffeomorphic to the unit sphere $\mathbb{S}^2$.
Let
\begin{equation}
z : \mathbb{S}^2 \times \mathbb{R}^+ \to \mathbb{R}^3, \qquad z(\alpha,t) = z^i (\alpha,t) \, \mathbf{e}_i, 
\end{equation}
be a smooth parametrization such that for each fixed $t$, the map
$z(\cdot,t)$
is a smooth embedding whose image equals $\Gamma(t)$.
The associated \emph{worldsheet} of the interface is denoted by
\begin{equation}\label{worldsheet-def}
\mathbf{\Gamma}
\coloneqq
\{ (z(\alpha,t),t) : \alpha \in \mathbb{S}^2,\ t \ge 0 \}
\subset \mathbb{R}^3 \times \mathbb{R}.
\end{equation}

For $\beta \in \mathbb{S}^2$ and $\tau \ge 0$, define the surface Jacobian
\begin{equation}\label{Jsurf-def}
J(\beta,\tau)
\coloneqq
\left|
\partial_{\beta_1} z(\beta,\tau)
\times
\partial_{\beta_2} z(\beta,\tau)
\right|,
\end{equation}
where $(\beta_1,\beta_2)$ are local coordinates on $\mathbb{S}^2$.
We denote the surface element on $\Gamma(\tau)$ by
\begin{equation}
\dd S_y(\beta,\tau)
=
J(\beta,\tau)\,
\dd \beta_1 \dd \beta_2. 
\end{equation}
\end{subequations}
\begin{subequations}\label{normal}
The unit normal to $\Gamma(\tau)$ at $y=z(\beta,\tau)\in\Gamma(\tau)$, 
oriented from $\Omega^-(\tau)$ into $\Omega^+(\tau)$,  is denoted by $n(y,\tau)$.
In local coordinates this is given by
\begin{equation}
\NN(\beta,\tau)
\coloneqq
n(z(\beta,\tau),\tau)
=
\frac{
\partial_{\beta_1} z(\beta,\tau)
\times
\partial_{\beta_2} z(\beta,\tau)
}{
\left|
\partial_{\beta_1} z(\beta,\tau)
\times
\partial_{\beta_2} z(\beta,\tau)
\right|
}.
\end{equation}
The normal derivative of a sufficiently smooth function
$f:\mathbb{R}^3\times\mathbb{R}^+ \to \mathbb{R}$ is denoted by
\begin{equation}
\partial_{\NN} f(y,\tau)
\coloneqq
n(y,\tau)\cdot \nabla_y f(y,\tau).
\end{equation}
\end{subequations}

\subsection{Travel-time coordinates}

To describe the causal geometry in the heterogeneous case, we introduce
the \emph{travel-time coordinates}
\begin{equation}\label{travel-time}
  \theta=\theta(t;\tau)
  \qquad \text{and} \qquad 
  \eta=\eta(x,t;y,\tau),
\end{equation}
where $\theta$ denotes the temporal coordinate and $\eta$ the travel-time function. 
Throughout, $(x,t)\in\mathbb{R}^3\times\mathbb{R}$ will 
typically denote an \emph{observation} (or \emph{target})
space-time event, with $t \ge \tau$, and 
$(y,\tau)\in\Gamma(\tau)\times\mathbb{R}$ will typically 
denote a \emph{source} space-time event which lies on the worldsheet of the interface. 
A natural choice for $\theta$ is the \emph{canonical gauge}
\begin{equation} \label{canonical-gauge}
  \theta(t;\tau)=t-\tau,
\end{equation}
so that the causal geometry is determined
entirely by the travel-time function $\eta(x,t;y,\tau)$.
We begin by identifying some structural properties that this travel-time
function must satisfy.

\begin{definition}[Travel-time function]
\label{def:characteristic-coordinates}

A function
\begin{equation}\label{eta}
  \eta=\eta(x,t;y,\tau)
\end{equation}
is called a \emph{travel-time function} if the following
conditions hold:

\begin{enumerate}[label=(\alph*)]

  \item \textbf{Positivity and regularity.}
  $\eta$ is nonnegative and
  $C^2$ in $(x,t)$ and in $(y,\tau)$ away from the diagonal $x=y$,
  and satisfies
  \begin{subequations}\label{eta-struct}
  \begin{equation}
    \eta(x,t;y,\tau)\ge 0,
    \qquad
    \eta(x,t;y,\tau)=0 \iff x=y.
  \end{equation}
  Near the space-time diagonal, $\eta$ is governed to leading order\footnote{%
  Geometrically, this means that the true light 
  cone (see \Cref{fig:cone-x} and Definition \ref{def:causal-geometry}) at $(y,\tau)$ is tangent 
  to the source-frozen spherical cone obtained by freezing the medium at $(y,\tau)$.} 
  by the instantaneous \emph{slowness}
  \begin{equation}\label{slowness}
    \kappa(y,\tau)\coloneqq \frac{1}{c(y,\tau)}.
  \end{equation}
  That is, as $(x,t)\to (y,\tau)$,
  \begin{equation}\label{eta-local-diag-expansion}
    \eta(x,t;y,\tau)
    =
    \kappa(y,\tau)\,|x-y|
    +
    \mathcal O\!\left(
      |x-y|^2 + |t-\tau|\,|x-y|
    \right).
  \end{equation}
  Moreover, writing
  \[
    x-y = r\,\omega,
    \qquad
    r = |x-y|,
    \qquad
    \omega \in \mathbb{S}^2,
  \]
  then for fixed direction $\omega$, 
  \begin{align}
    \nabla_x \eta(x,t;y,\tau)
    &=
    \phantom{-}\kappa(y,\tau)\,\omega
    +
    \mathcal O\!\left(r+|t-\tau|\right),
    \label{eta-local-diag-gradx}
    \\
    \nabla_y \eta(x,t;y,\tau)
    &=
    -\kappa(y,\tau)\,\omega
    +
    \mathcal O\!\left(r+|t-\tau|\right).
    \label{eta-local-diag-grady}
  \end{align}
  In particular, if $y=z(\beta,\tau)\in\Gamma(\tau)$ and $x$ approaches $y$
  along the unit normal direction $\NN(\beta,\tau)$, $x = y \pm \varepsilon \NN(\beta,\tau)$ 
  with $\varepsilon \downarrow 0$, then
  \begin{align}
    \partial_{\NN_x}\eta(x,t;y,\tau)
    &\to
    \pm \kappa(y,\tau),
    \label{eta-local-diag-normal-x}
    \\
    \partial_{\NN_y}\eta(x,t;y,\tau)
    &\to
    \mp \kappa(y,\tau),
    \label{eta-local-diag-normal-y}
  \end{align}
  On the equal-time slice $t=\tau$, the Laplacian has the local form
  \begin{equation}\label{eta-local-diag-laplacian}
    \Delta_x \eta(x,\tau;y,\tau)
    =
    \frac{2\kappa(y,\tau)}{|x-y|}
    +
    \Delta_x\eta_{\mathsf{reg}}(x,\tau;y,\tau),
    \qquad x\neq y,
  \end{equation}
  where $\Delta_x \eta_{\mathsf{reg}}(x,\tau;y,\tau)=\mathcal O(1)$ near $x=y$.
  Finally, $\partial_\tau \eta(x,t;y,\tau)$ is continuous for $x\neq y$, and
  \begin{equation}\label{eta-local-diag-partial-tau-trace}
    \partial_\tau \eta(x,t;y,\tau)\to 0
    \qquad
    \text{as }(x,t)\to (y,\tau),\ t>\tau.
  \end{equation}
  \end{subequations}

  \item \textbf{Non-degeneracy.}
  For each fixed observation point $(x,t)$ and each boundary parameter $\beta$,
  the function
  \[
    \tau \longmapsto \theta(t;\tau) - \eta(x,t;z(\beta,\tau),\tau)
  \]
  is strictly decreasing, where $\theta$ is defined by \eqref{canonical-gauge}. Equivalently,
  \begin{equation}\label{monotone-intersection}
    1 + \frac{\mathrm d}{\mathrm d\tau}\eta(x,t;z(\beta,\tau),\tau) > 0,
  \end{equation}
  where the total $\tau$-derivative is
  \[
    \frac{\mathrm d}{\mathrm d\tau}\eta(x,t;z(\beta,\tau),\tau)
    =
    \partial_\tau \eta(x,t;y,\tau)\big|_{y=z(\beta,\tau)}
    +
    \partial_\tau z(\beta,\tau)\cdot
    \nabla_y \eta(x,t;y,\tau)\big|_{y=z(\beta,\tau)}.
  \]
  Consequently, the wavefront equation
  \begin{equation}\label{wavefront-canonical}
    \theta(t;\tau)=\eta(x,t;z(\beta,\tau),\tau)
  \end{equation}
  has at most one solution in $\tau$ for each fixed $(x,t,\beta)$, and any such
  solution is simple.

\end{enumerate}

\end{definition}

The travel-time coordinates naturally determine forward 
and backward causal regions, which encode the propagation 
of wavefronts in space-time.
\begin{definition}[Causal geometry]
\label{def:causal-geometry}

Let $(\theta,\eta)$ be travel-time coordinates in the sense of
Definition \ref{def:characteristic-coordinates}.

\begin{enumerate}[label=(\alph*)]

  \item \textbf{Forward causal geometry.}
  For each source point $(y,\tau)$, define 
  \begin{subequations}\label{Sigma-plus}
  \begin{align}
    \Sigma^{+}(y,\tau)
    &\coloneqq
    \left\{
      (x,t)\in\mathbb{R}^3\times\mathbb{R}
      :
      \theta(t;\tau)\ge \eta(x,t;y,\tau)
    \right\}, \label{Sigma-plus-def}
    \\
    \partial\Sigma^{+}(y,\tau)
    &\coloneqq
    \left\{
      (x,t)
      :
      \theta(t;\tau)=\eta(x,t;y,\tau)
    \right\}. \label{Sigma-plus-bdy}
  \end{align}
  \end{subequations}
  The volume $\Sigma^{+}(y,\tau)$ is called the \emph{forward causal region}, and 
  its boundary $\partial\Sigma^{+}(y,\tau)$ is called the
  \emph{forward light cone}
  emitted from $(y,\tau)$. 

  \item \textbf{Backward causal geometry.}
  For each observation point $(x,t)$, define
  \begin{subequations}\label{Sigma-minus}
  \begin{align}
    \Sigma^{-}(x,t)
    &\coloneqq
    \left\{
      (y,\tau)\in\mathbb{R}^3\times\mathbb{R}
      :
      \theta(t;\tau)\ge \eta(x,t;y,\tau)
    \right\},
    \\
    \partial\Sigma^{-}(x,t)
    &\coloneqq
    \left\{
      (y,\tau)
      :
      \theta(t;\tau)=\eta(x,t;y,\tau)
    \right\}.
  \end{align}
  \end{subequations}
  The volume $\Sigma^{-}(y,\tau)$ is called the \emph{backward causal region}, and 
  its boundary $\partial\Sigma^{-}(x,t)$ is called the
  \emph{backward light cone} associated with $(x,t)$.

\end{enumerate}

\end{definition}

The forward and backward light cones can equivalently be characterized 
via the \emph{arrival time}. 

\begin{definition}[Arrival time]
\label{def:arrival-time}
Fix a source point $(y,\tau)$ and an observation location $x$.
The \emph{arrival time} of the signal emitted from $(y,\tau)$ at $x$
is defined as the unique time $ \Tarr (x; y,\tau)$ 
satisfying the wavefront condition
\begin{equation}\label{arrival-time-def}
  \theta ( \Tarr (x;y,\tau);\tau )
  =
  \eta (x,\Tarr(x;y,\tau);y, \tau).
\end{equation}
\end{definition}

In other words, $\Tarr (x;y,\tau)$ is the smallest time $t\ge\tau$
such that $(x,t)$ lies on the forward light cone $\partial \Sigma^+(y,\tau)$.
Similarly, for a fixed observation point $(x,t)$,
the backward light cone
$\partial\Sigma^{-}(x,t)$ 
consists of all source events $(y,\tau)$ satisfying $ t = \Tarr (x;y,\tau)$.

\subsection{Subsonic interface motion}

In order to have a well-posed problem \cite{Cooper1979}, we assume 
that the motion of the interface is \emph{subsonic}. 
\begin{definition}[Subsonic condition]
\label{def:subsonic}
Denote the normal velocity of the interface by
\begin{equation}\label{normal-velocity}
  V_\NN(\beta,t)
  \;\coloneqq\;
  \NN(\beta,t)\cdot \partial_t z(\beta,t).
\end{equation}
We say that the \emph{subsonic condition} holds if
\begin{equation}\label{subsonic-condition}
  \bigl|V_\NN(\beta,t)\bigr|
  <
  c(z(\beta,t),t),
  \qquad
  (\beta,t)\in\mathbb{S}^2\times\mathbb{R}^+. 
\end{equation}
\end{definition}
The subsonic condition guarantees that the 
non-degeneracy property \eqref{monotone-intersection} holds locally near the 
space-time diagonal $(x,t)=(y,\tau)$, with $y = z(\beta,\tau)$. The expansions 
\eqref{eta-local-diag-grady} and \eqref{eta-local-diag-partial-tau-trace} imply that
\[
\frac{\mathrm d}{\mathrm d\tau}\eta(x,t;y,\tau)
=
-\kappa(y,\tau)\,\partial_t z(\beta,\tau)\cdot \omega
+
\mathcal O\!\left(r+|t-\tau|\right), \qquad \text{as } (x,t) \to (y,\tau), 
\]
where $x-y = r \omega$. 
Taking the limit along the normal direction $\omega=\NN(\beta,\tau)$ and using \eqref{normal-velocity}, we obtain
\[
\lim_{(x,t)\to(y,\tau)}
\frac{\mathrm d}{\mathrm d\tau}\eta(x,t;z(\beta,\tau),\tau)
=
-\kappa(y,\tau)\,V_\NN(\beta,\tau).
\]
The subsonic condition \eqref{subsonic-condition} implies
\[
\left|
\lim_{(x,t)\to(y,\tau)}
\frac{\mathrm d}{\mathrm d\tau}\eta(x,t;z(\beta,\tau),\tau)
\right|
<
1, 
\]
which, by continuity, implies a local version of \eqref{monotone-intersection}.

\subsection{Prescribed incident wave and reflected causal region}

In the scattering setting considered throughout this work,
the Dirichlet boundary data $f$ is induced by a prescribed, analytically known
\emph{incident wave} $\phiinc : \mathbb{R}^3 \times \mathbb{R}^+ \to \mathbb{R}$ via the trace
\begin{equation}
  f(x,t)
  =
  -\phiinc (x,t),
  \qquad (x,t) \in \Gamma(t) \times \mathbb{R}^+, 
\end{equation}
where $\phiinc$ is assumed to satisfy the homogeneous equation
\begin{subequations}\label{incident}
\begin{equation}
  \left(  \tfrac{1}{c_0^2} \p_t^2 - \Delta_x  \right) \phiinc = 0, 
  \qquad \text{in } \mathbb{R}^3 \times \mathbb{R}^+, 
\end{equation}
with $c_0 > 0$ a constant.  
We further assume that the incident wave is normalized so that
\begin{equation}
  \max_{(x,t)\in\mathbb{R}^3\times\mathbb{R}^+} \phiinc(x,t)=1.
\end{equation}
\end{subequations}
In the simplest case, $\phiinc$ consists of a single localized pulse in time,
so that for each spatial point $x$ the function $t \mapsto \phiinc(x,t)$
attains its maximum at a unique time. More generally, we also allow
\emph{wave trains}, for which multiple such maxima may occur.

Associated with the incident wave \eqref{incident} is the 
\emph{reflected wavefront} in the exterior domain $\Omega^+(t)$. For the case of a moving boundary, a rigorous mathematical 
analysis of the reflected wavefront from a delta-function incident plane wave was provided in \cite{CoSt1984}. 
Here, we provide a simple description of the reflection(s) associated with \eqref{incident}. 
For each $\beta \in \mathbb{S}^2$, define the \emph{reflection source time}
$\thit(\beta)$ to be the first time at which the peak of the incident wave
reaches the moving boundary point $z(\beta,t)$, namely
\begin{equation}\label{thit}
   \thit (\beta)
  \coloneqq
  \inf \big\{
    t>0 :
    \phiinc (z(\beta,t),t) = 1
  \big\}.
\end{equation}
For problems with an incident wave train, the reflection source times form
a collection $\{\thit^\ell(\beta)\}_{\ell\ge 1}$, defined recursively by
\begin{equation}\label{thit-train}
  \thit^\ell(\beta)
  \coloneqq
  \inf\big\{
    t>\thit^{\ell-1}(\beta)
    :
    \phiinc(z(\beta,t),t)=1
  \big\},
  \qquad \ell\ge 2,
\end{equation}
with $\thit^1(\beta)\equiv \thit(\beta)$.

The reflection source time $\thit(\beta)$ defined in \eqref{thit}
marks the first instant at which the incident wave activates
the boundary point $z(\beta,\thit(\beta)) \in \Gamma(\thit(\beta))$, and each activated boundary
point $z(\beta,\thit(\beta))$ subsequently acts as a secondary source.
The reflected wavefront at time $t=T$ is therefore obtained as the envelope
of all forward light cones emitted from the boundary.
Any reflected signal observed at $(x,t)\in\Omega^+(t) \times \mathbb{R}^+$
must therefore be causally connected to at least one boundary
event $(z(\beta,\thit(\beta)),\thit(\beta))$.
We define the \emph{reflected causal region} as the union of forward
causal regions emitted from the boundary at their hitting times,
\begin{subequations}\label{Sigma-rfl}
\begin{equation}
  \Sigma_{\mathsf{rfl}}
  \coloneqq
  \bigcup_{\beta \in \mathbb{S}^2}
  \Sigma^{+} (z(\beta,\thit(\beta)),\thit(\beta) ).
\end{equation}
Using \eqref{Sigma-plus-def}, this may be written equivalently as
\begin{equation}\label{Sigma-rfl-characterization}
  \Sigma_{\mathsf{rfl}}
  =
  \Big\{
  (x,t)
  :
  x \in\Omega^+(t), \ t \ge 0, \ 
  \exists\, \beta\in\mathbb{S}^2
  \ \text{such that}\
  \theta (t;\thit(\beta) )
  \ge
  \eta (x,t;z(\beta,\thit(\beta)),\thit(\beta) )
  \Big\}.
\end{equation}
\end{subequations}
The \emph{reflected wavefront} is the boundary
$\partial\Sigma_{\mathsf{rfl}}$ of this set: 
\begin{equation}\label{reflected-wavefront}
  \partial \Sigma_{\mathsf{rfl}}
  =
  \Big\{
  (x,t) : 
  x \in\Omega^+(t), \ t \ge 0, \ 
  \exists\, \beta\in\mathbb{S}^2
  \ \text{such that}\
  \theta (t;\thit(\beta) )
  =
  \eta (x,t;z(\beta,\thit(\beta)),\thit(\beta) )
  \Big\}.
\end{equation}
Equivalently, in terms of the arrival time function
$\Tarr(x;y,\tau)$ defined in \eqref{arrival-time-def},
the earliest reflected arrival time at $x$ is
\begin{equation}\label{Tref-def}
  \Trfl (x)
  \coloneqq
  \inf_{\beta \in \mathbb{S}^2} \,
  \Tarr (x; z(\beta,\thit(\beta)),\thit(\beta) ).
\end{equation}
Then
\[
  (x,t)\in\partial\Sigma_{\mathsf{rfl}}
  \ \iff \
  t = \Trfl (x). 
\]

In the case of an incident wave train, the reflected field is generated by
a collection of boundary emission events
$(z(\beta,\thit^\ell(\beta)),\thit^\ell(\beta))$, where
$\{\thit^\ell(\beta)\}_{\ell\ge 1}$ is the family of reflection source times
defined in \eqref{thit-train}.
The reflected causal region is then obtained by taking the union over all such
secondary emission events,
\begin{subequations}\label{Sigma-rfl-train}
\begin{equation}
  \Sigma_{\mathsf{rfl}}^{\mathsf{train}}
  \coloneqq
  \bigcup_{\ell\ge 1}\ \bigcup_{\beta \in \mathbb{S}^2}
  \Sigma^{+}(z(\beta,\thit^\ell(\beta)),\thit^\ell(\beta)).
\end{equation}
Equivalently,
\begin{equation}
  \Sigma_{\mathsf{rfl}}^{\mathsf{train}}
  =
  \Big\{
  (x,t)
  :
  x\in\Omega^+(t),\ t\ge 0,\ 
  \exists\,\beta\in\mathbb{S}^2,\ \exists\,\ell\ge 1
  \ \text{such that}\
  \theta(t;\thit^\ell(\beta))
  \ge
  \eta(x,t;z(\beta,\thit^\ell(\beta)),\thit^\ell(\beta))
  \Big\}.
\end{equation}
\end{subequations}
Its boundary $\partial\Sigma_{\mathsf{rfl}}^{\mathsf{train}}$ is the union of
the corresponding reflected wavefronts, and the earliest reflected arrival
time is
\begin{equation}
  \Trfl^{\mathsf{train}}(x)
  \coloneqq
  \inf_{\ell\ge 1}\ \inf_{\beta\in\mathbb{S}^2}
  \Tarr(x;z(\beta,\thit^\ell(\beta)),\thit^\ell(\beta)).
\end{equation}

\begin{remark}[The role of the incident wave]
In the classical scattering problem, the total field
$\phi_{\mathsf{tot}}=\phiinc+\phi$ and the incident field $\phiinc$
both satisfy the same homogeneous wave equation in the exterior domain,
and hence so does the scattered field $\phi$.
In the present heterogeneous setting, this decomposition is not available:
the superposition $\phi_{\mathsf{tot}}=\phiinc+\phi$ does \emph{not} produce a
total field satisfying the exterior equation
$\Lop \phi_{\mathsf{tot}}=0$.
Instead, we formulate the exterior problem directly for $\phi$,
and the incident field serves only to generate the boundary forcing. 
The terminology ``incident'' should therefore be understood in this limited sense.
\end{remark}

\subsection{Geometric--optics parametrix representation}

We conclude this section by summarizing the 
geometric--optics parametrix.  
The formal structure outlined below arises in the 
classical high-frequency WKB analysis of \eqref{L-general-pde} and 
has been rigorously justified in that setting \cite{Lax1957,Hormander,Eskin2011}.

Assume that the fundamental solution for $\Lop$
admits a (causal) parametrix representation of the form
\begin{equation}
\label{G-pmtx}
  \Gop(x,t; y,\tau)
  =
  \amp (x,t; y,\tau)\,
  \Gopo (\eta(x,t; y,\tau), \theta(t ; \tau) ), 
\end{equation}
for $(x,t)\neq(y,\tau)$, and impose the homogeneous equation
\begin{equation}\label{homog-away-st}
  \Lop\, \Gop(x,t;y,\tau) = 0,
\end{equation}
in the sense that \eqref{G-pmtx} satisfies 
\eqref{homog-away-st} up to lower-order smooth terms, with 
singular support confined to the characteristic surface $\theta=\eta$.
Then the travel-time function $\eta$ and amplitude $\amp$ are 
determined by the coupled eikonal--transport 
system \eqref{eikonal-delay-st}--\eqref{transport-st} below.

\subsubsection{The travel-time eikonal equation for $\eta(x,t;y,\tau)$.}

For \eqref{homog-away-st} to hold at leading order on the wavefront $\theta = \eta$, the 
function $\eta$ must satisfy the \emph{travel-time eikonal equation}
\begin{subequations}\label{eikonal-delay-st}
\begin{equation}\label{eikonal-delay-st-a}
  \frac{\big( 1 - \partial_t\eta(x,t;y,\tau) \big)^2}{c(x,t)^2}
  -
  |\nabla_x \eta(x,t;y,\tau) |^2
  =
  0,
  \qquad (x,t)\neq(y,\tau). 
\end{equation}
It is supplemented with the initial condition
\begin{equation}\label{eikonal-delay-st-ic}
  \eta(y,\tau;y,\tau)=0.
\end{equation}
\end{subequations}

\subsubsection{Eikonal equation for the arrival time function $\Tarr$.}

Differentiating \eqref{arrival-time-def} with respect to $x$,
applying the chain rule, and using \eqref{eikonal-delay-st-a}, 
we obtain the \emph{arrival-time (implicit) eikonal equation}
\begin{subequations}\label{arrival-eikonal}
\begin{equation}
  |\nabla_x \Tarr(x;y,\tau)|^2
  =
  \frac{1}{c (x,\Tarr(x;y,\tau) )^2},
  \qquad x\neq y,
\end{equation}
with initial condition
\begin{equation}\label{arrival-eikonal-ic}
  \Tarr(y;y,\tau)=\tau.
\end{equation}
\end{subequations}

\subsubsection{Transport equation for the amplitude $\amp(x,t;y,\tau)$.}

Imposing cancellation of \eqref{homog-away-st} at next order on the wavefront
$\theta=\eta$ yields the transport equation
\begin{subequations}\label{transport-st}
\begin{align}\label{transport-eq}
  &2\,\nabla_x \amp(x,t;y,\tau)\cdot\nabla_x\eta(x,t;y,\tau)
  + \amp(x,t;y,\tau)\,\Delta_x\eta(x,t;y,\tau)
  \notag\\
  &
  =
  \frac{\amp(x,t;y,\tau)}{c(x,t)^2}
  \Bigg(
    \frac{2\,\big(1-\partial_t\eta(x,t;y,\tau)\big)}
         {\eta(x,t;y,\tau)}
    \;+\; \partial_{t}^2\eta(x,t;y,\tau)
  \Bigg)
  \;-\;
  \frac{2\big(1-\partial_t\eta(x,t;y,\tau)\big)}{c(x,t)^2}\,
  \partial_t \amp(x,t;y,\tau), 
\end{align}
valid in regions where $\eta$ is smooth. 
The transport equation determines $\amp$ only up to its value on the space-time diagonal. 
This diagonal value is fixed by the Green's-function normalization: substituting the 
parametrix ansatz into $\Lop\Gop$ and matching the coefficient of $\delta(x-y)\delta(t-\tau)$ gives
\begin{equation}\label{A-normalization-st}
  \amp(y,\tau;y,\tau)
  =
  \kappa(y,\tau).
\end{equation}
\end{subequations}

\subsubsection{Structure of the geometric--optics parametrix}

The following result records the resulting formal structure of the parametrix 
(see, e.g., \cite{Lax1957,Hormander,Eskin2011} for rigorous statements and proofs).

\begin{theorem}[Formal structure of the geometric--optics parametrix]
\label{prop:tt-parametrix-structural}
Assume that $\Gop$ admits the representation \eqref{G-pmtx}, where
$\eta(x,t;y,\tau)$ and $\theta(t;\tau)$ are smooth away from $x=y$, and suppose
that $\eta$ and $\amp$ solve the travel-time eikonal equation
\eqref{eikonal-delay-st} and the transport equation \eqref{transport-st}, respectively.
Then the following properties hold.

\begin{enumerate}

\item \textbf{Wavefront geometry.}
For a fixed source point $(y,\tau)$, the distribution $\Gop(\cdot,\cdot;y,\tau)$
is supported on the forward light cone $\partial\Sigma^+(y,\tau)$ 
defined in \eqref{Sigma-plus-bdy}.

\item \textbf{Formal parametrix identity.}
There exists a distribution $\mathcal R(x,t;y,\tau)$, supported in the forward
causal region $\Sigma^+(y,\tau)$, such that
\begin{equation}\label{general-parametrix-identity}
  \Lop \Gop(x,t;y,\tau)
  =
  \delta(x-y)\delta(t-\tau)
  +
  \mathcal R(x,t;y,\tau).
\end{equation}
Moreover, $\mathcal R$ is of lower singular order than the leading wavefront
contribution carried by $\Gop$.
\end{enumerate}
\end{theorem}


\section{Geometric--optics parametrix boundary integral formulation}\label{sec:BIE}

Let $\phi$ solve the exterior Dirichlet problem 
\begin{equation}\label{phi-exact}
\Lop \phi (x,t) = 0, \qquad (x,t) \in \Omega^+(t) \times \mathbb{R}^+,
\end{equation}
with the initial and boundary conditions \eqref{L-bcs}.  
Motivated by the parametrix representation
\eqref{G-pmtx}, we approximate $\phi$ by layer potentials 
built from the parametrix kernel $\Gop$.

Since we consider the exterior problem, all traces on the moving
boundary are taken from the exterior domain $\Omega^+(t)$.

\begin{definition}[Exterior trace operator]
\label{def:exterior-trace-operator}
Let $w(x,t)$ be defined for $(x,t)\in\Omega^+(t)\times\mathbb{R}^+$ and
assume that, for each $(\alpha,t)\in\mathbb{S}^2\times\mathbb{R}^+$, the
one-sided limit from the exterior exists.  The \emph{exterior trace} of $w$
on $\Gamma(t)$ is the function $\gamma^+ w$ defined by
\begin{equation}\label{def-gamma-plus}
  (\gamma^+ w)(\alpha,t)
  \coloneqq
  \lim_{\varepsilon\downarrow0}
  w\bigl(z(\alpha,t)+\varepsilon \NN(\alpha,t),\,t\bigr).
\end{equation}
When $w=w(x,t;y,\tau)$ also depends on source variables
$(y,\tau)\in\Gamma(\tau)\times\mathbb{R}^+$, the notation
$\gamma^+ w$ always refers to the trace in the observation variables $(x,t)$
with $(y,\tau)$ held fixed.
\end{definition}

\begin{definition}[Boundary pullbacks of kernel quantities]
\label{def:boundary-pullbacks}
For $\alpha,\beta\in\mathbb{S}^2$ and $t\ge \tau$, define
\begin{subequations}\label{eta-pullback}
\begin{align}
  \upeta(\alpha,t;\beta,\tau)
  &\coloneqq
  \gamma^+\eta(\alpha,t;z(\beta,\tau),\tau),
  \\
  \partial_{\NN_y}\upeta(\alpha,t;\beta,\tau)
  &\coloneqq
  \NN(\beta,\tau)\cdot
  \gamma^+\nabla_y\eta(\alpha,t;y,\tau)\big|_{y=z(\beta,\tau)},
  \\
  \partial_\tau \upeta(\alpha,t;\beta,\tau)
  &\coloneqq
  \gamma^+\partial_\tau\eta(\alpha,t;z(\beta,\tau),\tau),
  \\
  \frac{\mathrm d}{\mathrm d\tau}\upeta(\alpha,t;\beta,\tau)
  &\coloneqq
  \partial_\tau\upeta(\alpha,t;\beta,\tau)
  +
  \partial_t z(\beta,\tau)\cdot
  \gamma^+\nabla_y\eta(\alpha,t;y,\tau)\big|_{y=z(\beta,\tau)}.
\end{align}
\end{subequations}
The boundary pullbacks of the amplitude and its normal derivative are
\begin{subequations}\label{A-pullback}
\begin{align}
  A(\alpha,t;\beta,\tau)
  &\coloneqq
  \gamma^+\amp(\alpha,t;z(\beta,\tau),\tau),
  \label{A-pullback-a}\\
  \partial_{\NN_y}A(\alpha,t;\beta,\tau)
  &\coloneqq
  \NN(\beta,\tau)\cdot
  \gamma^+\nabla_y\amp(\alpha,t;y,\tau)\big|_{y=z(\beta,\tau)}.
  \label{A-pullback-b}
\end{align}
\end{subequations}
The boundary pullbacks of the Dirichlet data and sound speed are
\begin{equation}\label{F-pullback}
  F(\alpha,t)
  \coloneqq
  \gamma^+ f(\alpha,t), 
\end{equation}
and 
\begin{equation}\label{C-pullback}
  C(\alpha,t)
  \coloneqq
  \gamma^+ c(\alpha,t). 
\end{equation}
\end{definition}

\subsection{Layer-potential formulations}

Let
\[
  \sigma,\mu:\mathbb S^2\times\mathbb{R}^+\to\mathbb R
\]
be causal boundary densities. 
For $(x,t)\in\Omega^+(t)\times\mathbb{R}^+$, we define the 
\emph{geometric--optics parametrix single- and double-layer potentials} by
\begin{subequations}\label{layer-potentials}
\begin{align}
  \SopSL[\sigma](x,t)
  &\coloneqq
  \int_0^t\int_{\mathbb S^2}
    \Gop(x,t;z(\beta,\tau),\tau)\,
    \sigma(\beta,\tau)\,
    \mathrm dS_y(\beta,\tau)\,\mathrm d\tau,
  \label{SopSL}\\
  \DopDL[\mu](x,t)
  &\coloneqq
  \int_0^t\int_{\mathbb S^2}
    \partial_{\nu_y}\Gop(x,t;z(\beta,\tau),\tau)\,
    \mu(\beta,\tau)\,
    \mathrm dS_y(\beta,\tau)\,\mathrm d\tau,
  \label{DopDL}
\end{align}
where
\begin{equation}
  \partial_{\nu_y}\Gop(x,t;z(\beta,\tau),\tau)
  \coloneqq
  \nu(\beta,\tau)\cdot
  \nabla_y\Gop(x,t;y,\tau)\big|_{y=z(\beta,\tau)}.
\end{equation}
Using the parametrix representation \eqref{G-pmtx}, these become
\begin{align}
  \SopSL[\sigma](x,t)
  &=
  \int_0^t\int_{\mathbb S^2}
    \amp(x,t;z(\beta,\tau),\tau)\,
    \Gopo\bigl(
      \eta(x,t;z(\beta,\tau),\tau),
      \theta(t;\tau)
    \bigr)\,
    \sigma(\beta,\tau)\,
    \mathrm dS_y(\beta,\tau)\,\mathrm d\tau,
  \label{SopSL-explicit}\\
  \DopDL[\mu](x,t)
  &=
  \int_0^t\int_{\mathbb S^2}
    \partial_{\nu_y}
    \Big[
      \amp(x,t;z(\beta,\tau),\tau)\,
      \Gopo\bigl(
        \eta(x,t;z(\beta,\tau),\tau),
        \theta(t;\tau)
      \bigr)
    \Big]\,
    \mu(\beta,\tau)\,
    \mathrm dS_y(\beta,\tau)\,\mathrm d\tau .
  \label{DopDL-explicit}
\end{align}
\end{subequations}
Since $\Gop$ is supported on the wavefront
$\theta(t;\tau)=\eta(x,t;z(\beta,\tau),\tau)$, both layer 
potentials localize to points on the worldsheet 
$(z(\beta,\tau),\tau)$ satisfying this relation, and 
hence involve only retarded interactions along the distorted light cone.
Motivated by this structure, we approximate the 
solution of \eqref{phi-exact} by any of the following:
\begin{subequations}\label{layer-ansatze}
\begin{alignat}{5}
  \tilde\phi(x,t)
  &\,=\;\,&& \SopSL[\sigma](x,t),
  &\qquad& &\text{\emph{(single-layer)}},
  \label{layer-ansatze-a}\\
  \tilde\phi(x,t)
  &\,=\;\,&& \DopDL[\mu](x,t),
  &\qquad& &\text{\emph{(double-layer)}},
  \label{layer-ansatze-b}\\
  \tilde\phi(x,t)
  &\,=\;\,&& \DopDL[\mu](x,t) - \SopSL[\sigma](x,t),
  &\qquad& &\text{\emph{(Kirchhoff)}}. 
  \label{layer-ansatze-c}
\end{alignat}
\end{subequations}

\subsection{Single-layer boundary integral equation}
For each fixed \(t\), the kernel \(\Gop(\cdot,t;\cdot,\tau)\) has the same
spatial singularity structure as \(\Gopo(\cdot,\theta(t;\tau))\), modulo the
smooth amplitude factor \(\amp\). Consequently, the single-layer potential is
continuous across \(\Gamma(t)\), and its exterior trace defines the boundary
operator
\begin{subequations}\label{Sop}
\begin{align}
  \Sop[\sigma](\alpha,t)
  &\coloneqq
  \gamma^+\SopSL[\sigma](\alpha,t)
  \label{Sop-a}\\
  &=
  \int_0^t\int_{\mathbb S^2}
    A(\alpha,t;\beta,\tau)\,
    \Gopo\bigl(
      \upeta(\alpha,t;\beta,\tau),
      \theta(t;\tau)
    \bigr)\,
    \sigma(\beta,\tau)\,
    \mathrm dS_y(\beta,\tau)\,\mathrm d\tau .
  \label{Sop-b}
\end{align}
\end{subequations}
Imposing the Dirichlet boundary condition yields the single-layer boundary
integral equation
\begin{equation}\label{BIE-SL}
  \Sop[\sigma](\alpha,t)
  =
  F(\alpha,t),
  \qquad
  (\alpha,t)\in\mathbb S^2\times\mathbb{R}^+.
\end{equation}

\subsection{Double-layer boundary integral equation}
The double-layer potential is discontinuous across $\Gamma(t)$ and
its exterior trace satisfies the jump relation
\begin{subequations}\label{DL-jump}
\begin{equation}
  \gamma^+\DopDL[\mu](\alpha,t)
  =
  \tfrac12 \lambda(\alpha,t)\,\mu(\alpha,t)
  +
  \Dop[\mu](\alpha,t),
\end{equation}
where
\begin{equation}\label{lambda-main}
  \lambda(\alpha,t)
  \coloneqq
  1 + 
  \frac{V_\NN(\alpha,t)^2}
  {C(\alpha,t)^2-V_\NN(\alpha,t)^2},
\end{equation}
and $V_\NN(\alpha,t)$ denotes the normal velocity \eqref{normal-velocity}.
\end{subequations}
The regular boundary operator is
\begin{subequations}\label{Dop}
\begin{align}
  \Dop[\mu](\alpha,t)
  &\coloneqq
  \mathrm{p.v.}\!\int_0^t\int_{\mathbb S^2}
    \gamma^+\partial_{\nu_y}\Gop(\alpha,t;z(\beta,\tau),\tau)\,
    \mu(\beta,\tau)\,
    \mathrm dS_y(\beta,\tau)\,\mathrm d\tau
  \label{Dop-a}\\
  &=
  \mathrm{p.v.}\!\int_0^t\int_{\mathbb S^2}
    \mu(\beta,\tau)\,
    \Big[
      \partial_{\nu_y}A(\alpha,t;\beta,\tau)\,
      \Gopo\bigl(
        \upeta(\alpha,t;\beta,\tau),
        \theta(t;\tau)
      \bigr)
  \notag\\
  &\hspace{8em} \ 
      +
      A(\alpha,t;\beta,\tau)\,
      \partial_r\Gopo\bigl(
        \upeta(\alpha,t;\beta,\tau),
        \theta(t;\tau)
      \bigr)\,
      \partial_{\nu_y}\upeta(\alpha,t;\beta,\tau)
    \Big]\,
    \mathrm dS_y(\beta,\tau)\,\mathrm d\tau .
  \label{Dop-b}
\end{align}
\end{subequations}
Here $\mathrm{p.v.}$ denotes the Cauchy principal value with respect to the
singularity at $(\beta,\tau)=(\alpha,t)$. The proof of
\eqref{DL-jump} is given in the Appendix.

Imposing the Dirichlet boundary condition yields the double-layer boundary
integral equation
\begin{equation}\label{BIE-DL}
  \tfrac12 \lambda(\alpha,t)\,\mu(\alpha,t)
  +
  \Dop[\mu](\alpha,t)
  =
  F(\alpha,t),
  \qquad
  (\alpha,t)\in\mathbb S^2\times\mathbb{R}^+.
\end{equation}

\subsection{Kirchhoff representation formula}

Using Green's theorem, the boundary densities $\mu$ and $\sigma$ in the
Kirchhoff formulation\eqref{layer-ansatze-c} 
are identified with the boundary trace and the
corresponding exterior normal derivative trace, respectively. Thus,
\begin{equation}\label{Kirchhoff-densities}
  \mu(\alpha,t)=\gamma^+ \tilde{\phi}(\alpha,t) = F(\alpha,t)
  \qquad
  \text{and}
  \qquad 
  \sigma(\alpha,t)= \left( \gamma^+ \partial_{\NN_x} \tilde\phi \right) (\alpha,t) .  
\end{equation}
and the boundary integral equation is 
\begin{equation}\label{BIE-K}
  \Sop[\sigma](\alpha,t)
  =
  \left(-1+\tfrac12 \lambda(\alpha,t) \right)F(\alpha,t)
  +
  \Dop[F](\alpha,t)  ,
  \qquad
  (\alpha,t)\in\mathbb S^2\times\mathbb{R}^+.
\end{equation}

\subsection{Integration-by-parts double-layer formulation}
\label{subsec:DL-ibp-universal}

To obtain a form suitable for numerical discretization, we rewrite the
double-layer potential by converting the normal derivative on the kernel
into a temporal derivative acting on the density. 
Using the chain rule and integration by
parts in $\tau$, it can be shown that the double-layer 
operator \eqref{DopDL-explicit} admits the decomposition
\begin{equation}\label{DopDL-ibp}
\DopDL[\mu](x,t)
=
\DopDLpr[\mu](x,t)
+
\DopDL^{\amp}[\mu](x,t)
+
\DopDL^{\tau}[\mu](x,t). 
\end{equation}
Here, $\DopDLpr$ denotes the principal contribution 
obtained after integration by parts: it is the only term that survives in the 
case that the sound speed is constant and the interface is fixed.  
The correction terms $\DopDL^{\amp}$ and $\DopDL^{\tau}$ arise 
from differentiation of the amplitude and from residual time derivatives, respectively.
Taking the exterior trace $x\to z(\alpha,t)$ defines the boundary operator
\begin{equation}\label{Dop-IBP}
\Dop[\mu](\alpha,t)
=
\Doppr[\mu](\alpha,t)
+
\Dop^{\amp}[\mu](\alpha,t)
+
\Dop^{\tau}[\mu](\alpha,t),
\qquad
(\alpha,t)\in\mathbb S^2\times\mathbb{R}^+.
\end{equation}

Introduce the \emph{modulated unit-speed Green's function}
\begin{equation}\label{K0}
\Kopo(\eta,\theta)
\coloneqq
\frac{\Gopo(\eta,\theta)}{\eta}, 
\end{equation}
and let 
\begin{subequations}\label{QP-def}
\begin{align}
D(\alpha,t;\beta,\tau) 
&=
1 
+ 
\frac{\mathrm d}{\mathrm d\tau}
\upeta(\alpha,t;\beta,\tau), 
\\[0.5em]
Q(\alpha,t;\beta,\tau)
&=
-
\frac{
A(\alpha,t;\beta,\tau)\,
\partial_{\NN_y}\upeta(\alpha,t;\beta,\tau)
}{
D(\alpha,t;\beta,\tau)
},
\\[0.5em]
P(\alpha,t;\beta,\tau)
&=
J(\beta,\tau)\,Q(\alpha,t;\beta,\tau).
\end{align}
\end{subequations}
The principal contribution is
\begin{equation}\label{Dop-ibp-principal}
\begin{split}
\Doppr[\mu](\alpha,t)
&=
\int_0^t
\int_{\mathbb{S}^2}
Q(\alpha,t;\beta,\tau)\,
\Big[
\Gopo(\upeta,\theta)\,\partial_\tau\mu(\beta,\tau)
+
\Kopo(\upeta,\theta)\,\mu(\beta,\tau)
\Big]
\,\mathrm d S_y\,\mathrm d\tau, 
\end{split}
\end{equation}
and the correction terms are given by
\begin{align}
\Dop^{\amp}[\mu](\alpha,t)
&=
\int_0^t\int_{\mathbb{S}^2}
\partial_{\NN_y}A(\alpha,t;\beta,\tau)\,
\Gopo(\upeta,\theta)\,
\mu(\beta,\tau)\,
\mathrm d S_y\,\mathrm d\tau,
\\
\Dop^{\tau}[\mu](\alpha,t)
&=
\int_0^t\int_{\mathbb S^2}
\frac{\mathrm d P}{\mathrm d\tau}(\alpha,t;\beta,\tau)\,
\Gopo(\upeta,\theta)\,
\mu(\beta,\tau)\,
\mathrm d\beta\,\mathrm d\tau. \label{Dop-ibp-tcorr}
\end{align}

\begin{remark}[Induced interior extension]
In the constant-speed case, the parametrix \eqref{G-pmtx} is exact,
and the single-layer, double-layer, and Kirchhoff representations
correspond to different extensions of the same exterior solution
into $\Omega^-(t)$. Uniqueness ensures that all such
representations recover the same scattered field in $\Omega^+(t)$;
the induced interior field is therefore merely representational bookkeeping.
In the heterogeneous setting, the induced interior extension is no longer
mere bookkeeping but becomes part of the approximation itself, since the
chosen travel-time geometry and amplitude determine how waves propagate
and refract within the interior region.
In this sense, different layer-potential representations amount to
different modeling choices, and part of the residual of the
geometric--optics parametrix may be viewed as reflecting the
interior modeling implicit in that choice.
\end{remark}


\section{Space-time collocation scheme} \label{sec:collocation}

In this section we present a simple space-time collocation scheme for the
boundary integral equations derived in \Cref{sec:universal}. 
The method can be viewed as a generalization of the collocation schemes 
in \cite{mansur1982,Mansur1983,Rynne1990,DaDu1997}.  
The main novelty is the construction of discrete light cones and worldsheets via
the slab-frozen approximation, which is described in \Cref{subsec:slab-frozen}. 
Throughout, we assume that the amplitude $\amp$ and travel-time function $\eta$ are given. 
In later sections, we introduce ray-based approximations for 
$\eta$ and two-point models for $\amp$.

\subsection{Temporal discretization}

\begin{subequations} \label{t-discrete}
Fix a final time $\tmax > 0$, partition the time interval $(0,\tmax]$ into $K$ uniform subintervals,
\begin{equation}
0 = t_0 < t_1 < \cdots < t_K = \tmax, \qquad
\Delta t = \frac{\tmax}{K},
\end{equation}
and denote
\begin{equation}
T_\ell = (t_{\ell-1}, t_\ell], \qquad \ell = 1,\ldots,K.
\end{equation}
\end{subequations}

\subsubsection{$\mathbb{P}^1_t$ temporal approximations}

The boundary densities $\sigma(\alpha,t)$ and 
$\mu(\alpha,t)$ are approximated by continuous,
piecewise linear functions in time.
Let $\{\sigma_\ell(\alpha)\}_{\ell=0}^K$ and 
$\{\mu_\ell(\alpha)\}_{\ell=0}^K$ denote the nodal values of 
the single- and double-layer densities, respectively, at the temporal
grid points $\{t_\ell\}$, with $\sigma_0 \equiv 0$ and $\mu_0 \equiv 0$.
On each time slab $T_\ell$ we write
\begin{equation}\label{P1-time-sigma}
  \sigma_\Delta(\alpha,t)\big|_{T_\ell}
  =
  \frac{t_\ell - t}{\Delta t}\,\sigma_{\ell-1}(\alpha)
  +
  \frac{t - t_{\ell-1}}{\Delta t}\,\sigma_\ell(\alpha),
  \qquad \ell=1,\dots,K,
\end{equation}
and
\begin{equation}\label{P1-time-mu}
  \mu_\Delta(\alpha,t)\big|_{T_\ell}
  =
  \frac{t_\ell - t}{\Delta t}\,\mu_{\ell-1}(\alpha)
  +
  \frac{t - t_{\ell-1}}{\Delta t}\,\mu_\ell(\alpha),
  \qquad \ell=1,\dots,K.
\end{equation}
The time derivative of the double-layer density is constant on each slab and given by
\begin{equation}\label{P1-time-mu-derivative}
  \partial_t \mu_\Delta(\alpha,t)\big|_{T_\ell}
  =
  \frac{\mu_\ell(\alpha)-\mu_{\ell-1}(\alpha)}{\Delta t},
  \qquad \ell = 1, \ldots, K.
\end{equation}

\subsubsection{Auxiliary temporal weights for the single-layer operator}

Collocation of the single-layer boundary integral 
operator \eqref{Sop} at time $t=t_k$ yields
contributions from the two linear temporal basis functions supported on each
time slab.
We therefore introduce auxiliary temporal weights
$\sltw^{(-)}_{k,\ell}(\alpha;\beta)$ and
$\sltw^{(+)}_{k,\ell}(\alpha;\beta)$,
corresponding to the left and right temporal shape functions on $T_\ell$.
For a fixed observation index $k$ and source slab $\ell$, these auxiliary
weights are defined by
\begin{subequations}\label{sltw}
\begin{align}
  \sltw^{(-)}_{k,\ell}(\alpha ; \beta)
  &\coloneqq 
  \frac{1}{\Delta t}
  \int_{T_\ell} 
  (t_\ell-\tau)\,
  A (\alpha, t_k ; \beta, \tau)\,
  \Gopo  ( \upeta(\alpha, t_k ; \beta, \tau), \theta(t_k ;\tau)  )\,
  J(\beta,\tau)\,
  \dd \tau, 
  \\
  \sltw^{(+)}_{k,\ell}(\alpha ; \beta)
  &\coloneqq 
  \frac{1}{\Delta t}
  \int_{T_\ell} 
  (\tau-t_{\ell-1})\,
  A (\alpha, t_k ; \beta, \tau)\,
  \Gopo ( \upeta(\alpha, t_k ; \beta, \tau), \theta(t_k ;\tau)  )\,
  J(\beta,\tau)\,
  \dd \tau. 
\end{align}
\end{subequations}
where $\Gopo$ is defined by \eqref{G0-explicit}. 
The auxiliary weights \eqref{sltw} are combined into effective temporal weights
$\sltw_{k,\ell}(\alpha;\beta)$ by the assembly rule
\begin{equation}\label{sltw-eff}
  \sltw_{k,\ell}(\alpha ; \beta)
  \;\coloneqq\;
  \sltw^{(+)}_{k,\ell}(\alpha ; \beta)
  \;+\;
  \begin{cases}
    \sltw^{(-)}_{k,\ell+1}(\alpha ; \beta), & \ell=1,\ldots,k-1,\\[0.3em]
    0, & \ell=k.
  \end{cases}
\end{equation}

\subsubsection{Auxiliary temporal weights for the double-layer operator}

Similarly, collocation of the principal component \eqref{Dop-ibp-principal} of the
integration-by-parts double-layer boundary operator at time $t=t_k$
produces contributions from the two linear temporal basis functions on each
slab $T_\ell$, together with an additional term arising from the
piecewise-constant approximation \eqref{P1-time-mu-derivative} of
$\partial_t\mu$.
For a fixed observation index $k$ and source slab $\ell$, define the 
auxiliary weights
\begin{subequations}\label{dltw}
\begin{align}
  \dltw^{\partial}_{k,\ell}(\alpha;\beta)
  &\coloneqq 
  \frac{1}{\Delta t}
  \int_{T_\ell}
    Q(\alpha, t_k ;\beta,\tau) \, 
    \Gopo( \upeta(\alpha, t_k ; \beta, \tau), \theta(t_k ; \tau) ) \, 
     J(\beta,\tau) \, 
  \dd\tau,
  \\[0.5em]
\begin{split}
  \dltw^{(-)}_{k,\ell}(\alpha;\beta)
  &\coloneqq 
  \frac{1}{\Delta t} 
 \int_{T_\ell}
    (t_\ell-\tau) \, Q(\alpha, t_k ;\beta,\tau) \, 
    \Kopo(\upeta(\alpha, t_k ; \beta, \tau), \theta(t_k ;\tau))
    J(\beta,\tau)
  \dd\tau,
\end{split}
\\[0.5em]
\begin{split}
  \dltw^{(+)}_{k,\ell}(\alpha;\beta)
  &\coloneqq
  \frac{1}{\Delta t}
  \int_{T_\ell}
    { (\tau-t_{\ell-1})} \,  Q(\alpha, t_k ;\beta,\tau) \, 
    \Kopo(\upeta(\alpha, t_k ; \beta, \tau), \theta(t_k ;\tau))
    J(\beta,\tau)
  \dd\tau.
\end{split}
\end{align}
\end{subequations}
Here $\Kopo$ is the modulated unit-speed Green's function \eqref{K0} and $Q$ is defined in \eqref{QP-def}.
The auxilary temporal weights are then assembled to form the effective weight
\begin{equation}\label{dltw-effective}
\begin{aligned}
\dltw_{k,\ell}(\alpha;\beta)
\;\coloneqq\;
&\dltw^{(+)}_{k,\ell}(\alpha;\beta)
\;+\;
\dltw^{\partial}_{k,\ell}(\alpha;\beta)
\;+\;
\begin{cases}
\dltw^{(-)}_{k,\ell+1}(\alpha;\beta)
-\dltw^{\partial}_{k,\ell+1}(\alpha;\beta), 
& \ell=1,\ldots,k-1,\\[0.4em]
0,
& \ell=k.
\end{cases}
\end{aligned}
\end{equation}

\subsubsection{Semi-discrete marching-on-time formulation}

Using the effective temporal weights introduced above, we define the
semi-discrete operators
\begin{equation} \label{op-semi-discrete}
  \Sop_{k,\ell}[\sigma](\alpha)
  \coloneqq 
  \int_{\mathbb{S}^2}
  \sltw_{k,\ell}(\alpha;\beta) \,
  \sigma (\beta)  \dd \beta
  \qquad\text{and}\qquad
  \Dop_{k,\ell}[\mu](\alpha)
   \coloneqq 
  \int_{\mathbb{S}^2}
  \dltw_{k,\ell}(\alpha;\beta) \,
  \mu(\beta) \dd \beta. 
\end{equation}
Set $F_k (\alpha) \coloneqq F( \alpha, t_k)$, for $\alpha \in \mathbb{S}^2$. 
Collocating the single-layer equation \eqref{BIE-SL} at time $t=t_k$ and
separating the contribution of the current time level gives the recursive
marching-on-time (MOT) relation
\begin{equation} \label{mot-SL}
  \Sop_{k,k}[\sigma_k]
  =
  F_k
  -
  \sum_{\ell=1}^{k-1} \Sop_{k,\ell}[\sigma_\ell],
  \qquad k=1,\dots,K.
\end{equation}
Similarly, the MOT recursions associated with the double-layer  
\eqref{BIE-DL} and Kirchhoff \eqref{BIE-K} boundary integral equations are 
\begin{equation}\label{DL-mot}
  \bigl(\tfrac12 \Lambda_k + \Dop_{k,k}\bigr)[\mu_k]
  =
  F_k
  -
  \sum_{\ell=1}^{k-1} \Dop_{k,\ell}[\mu_\ell],
  \qquad k=1,\dots,K,
\end{equation}
and
\begin{equation}\label{K-mot}
  \Sop_{k,k}[\sigma_k]
  =
  \left(-\mathcal{I}+\tfrac12 \Lambda_k\right)F_k
  +
  \Dop_{k,k}[F_k]
  +
  \sum_{\ell=1}^{k-1} 
  \Big(
  \Dop_{k,\ell}[F_\ell]
  -
  \Sop_{k,\ell}[\sigma_\ell]
  \Big),
  \qquad k=1,\dots,K,
\end{equation}
respectively, where $\mathcal{I}$ is the identity operator and 
$\Lambda_k$ is the multiplication operator
$(\Lambda_k v)(\alpha) \coloneqq \lambda_k(\alpha)\,v(\alpha)$.

\subsection{Slab-frozen approximation and closed-form expressions for the temporal weights}
\label{subsec:slab-frozen}

In this subsection, we introduce the \emph{slab-frozen approximation} to 
the causal geometry. Using piecewise constant temporal approximations 
for the geometry, travel-time, and amplitude, we reduce the retarded integrals defining 
the temporal weights \eqref{sltw} and \eqref{dltw} to integrals that can be 
evaluated analytically to yield closed-form expressions.

\subsubsection{Change-of-variables to $\theta$-space}

Fix an observation time $t=t_k$ and a source slab $T_\ell=(t_{\ell-1},t_\ell]$.  Define 
\begin{equation}
  \theta_{k,\ell} \coloneqq \theta(t_k;t_\ell) = t_k - t_\ell, 
  \qquad \ell=1,\ldots,k.
\end{equation}
Denote the inverse map by $\thetainv(t_k;\theta) = t_k - \theta$. 
We perform the change of variables 
$\tau=\thetainv(t_k;\theta)$,
and write the auxiliary single-layer \eqref{sltw} and 
double-layer \eqref{dltw} weights in the form
\begin{subequations}\label{sltw-theta}
\begin{align}
  \sltw^{(-)}_{k,\ell}(\alpha;\beta)
  &=
  \frac{1}{\Delta t}
  \int_{\theta_{k,\ell}}^{\theta_{k,\ell-1}}
  \bigl(t_\ell-\thetainv(t_k;\theta)\bigr)\,
  A(\alpha,t_k;\beta,\thetainv(t_k;\theta))\,
  \nonumber\\[-0.5em]
  &\hspace{12em}\times
  \Gopo (
      \upeta(\alpha,t_k;\beta,\thetainv(t_k;\theta)),
      \theta
    )\,
   J(\beta,\thetainv(t_k;\theta))\,
  \dd\theta,
  \\[0.7em]
  \sltw^{(+)}_{k,\ell}(\alpha;\beta)
  &=
  \frac{1}{\Delta t}
  \int_{\theta_{k,\ell}}^{\theta_{k,\ell-1}}
  \bigl(\thetainv(t_k;\theta)-t_{\ell-1}\bigr)\,
  A(\alpha,t_k;\beta,\thetainv(t_k;\theta))\,
  \nonumber\\[-0.5em]
  &\hspace{12em}\times
  \Gopo (
      \upeta(\alpha,t_k;\beta,\thetainv(t_k;\theta)),
      \theta
    )\,
   J(\beta,\thetainv(t_k;\theta))\,
  \dd\theta, 
\end{align}
\end{subequations}
and
\begin{subequations}\label{dltw-theta}
\begin{align}
  \dltw^{\partial}_{k,\ell}(\alpha;\beta)
  &=
  \frac{1}{\Delta t}
  \int_{\theta_{k,\ell}}^{\theta_{k,\ell-1}}
  Q(\alpha, t_k ; \beta, \thetainv(t_k;\theta)) 
   \nonumber\\[-0.5em]
  &\hspace{11em}\times \ 
  \Gopo  (
      \upeta(\alpha,t_k;\beta,\thetainv(t_k;\theta)),
      \theta
   ) \,
   J(\beta, \thetainv(t_k ; \theta) )  \,
  \dd\theta,
  \\[0.5em]
  \dltw^{(-)}_{k,\ell}(\alpha;\beta)
  &=
  \frac{1}{\Delta t}
  \int_{\theta_{k,\ell}}^{\theta_{k,\ell-1}}
  \bigl(
    t_\ell-\thetainv(t_k;\theta)
  \bigr) \,
  Q(\alpha,t_k;\beta,\thetainv(t_k;\theta))\,
  \nonumber\\[-0.5em]
  &\hspace{11em}\times
  \Kopo (
      \upeta(\alpha,t_k;\beta,\thetainv(t_k;\theta)),
      \theta
   ) \,
  J(\beta, \thetainv(t_k ; \theta) )  \,
  \dd\theta,
  \\[0.5em]
  \dltw^{(+)}_{k,\ell}(\alpha;\beta)
  &=
  \frac{1}{\Delta t}
  \int_{\theta_{k,\ell}}^{\theta_{k,\ell-1}}
  \bigl(
    \thetainv(t_k;\theta)-t_{\ell-1}
  \bigr) \,
  Q(\alpha,t_k;\beta,\thetainv(t_k;\theta))\,
  \nonumber\\[-0.5em]
  &\hspace{11em}\times
  \Kopo (
      \upeta(\alpha,t_k;\beta,\thetainv(t_k;\theta)),
      \theta
   ) \,
   J (\beta, \thetainv(t_k ; \theta) )  \,
  \dd\theta.
\end{align}
\end{subequations}
The mapped slab endpoints are
\begin{equation}\label{theta-mapped-limits}
  \theta_{k,\ell} = t_k - t_\ell
  \qquad
  \text{and}
  \qquad
  \theta_{k,\ell-1} = t_k - t_{\ell-1}.
\end{equation}
The time integral on $T_\ell$ contributes only when the \emph{wavefront equation}
\begin{subequations}\label{theta-limits}
\begin{equation}\label{theta-star-def}
  \theta
  -
  \upeta(\alpha,t_k;\beta,\thetainv(t_k;\theta))
  =0
\end{equation}
has a solution in the interval $[\theta_{k,\ell},\theta_{k,\ell-1})$, see \Cref{fig:frozen-slab-a}.
We denote this solution, when it exists, by $\theta^{*,\mathsf{ex}}_{k,\ell}(\alpha;\beta)$:
\begin{equation}\label{theta-star}
  \theta^{*,\mathsf{ex}}_{k,\ell}(\alpha;\beta)
  =
  \upeta (
    \alpha,t_k;\beta,\thetainv(t_k;\theta^{*,\mathsf{ex}}_{k,\ell}(\alpha;\beta))
   ) 
   \in
  [\theta_{k,\ell},\theta_{k,\ell-1}).
\end{equation}
The corresponding physical intersection time is 
\begin{equation}
\tau^{*,\mathsf{ex}}_{k,\ell}(\alpha;\beta) = t_k -  \theta^{*,\mathsf{ex}}_{k,\ell}(\alpha;\beta) \in (t_{\ell-1},t_\ell].
\end{equation}
By non-degeneracy \eqref{monotone-intersection}, any such
solution is unique and simple. Hence either no contribution arises from $T_\ell$,
or the slab contribution reduces to evaluation at $\theta=\theta^{*,\mathsf{ex}}_{k,\ell}$ via 
the distributional identity
\begin{equation}\label{delta-wavefront}
  \delta (
    \theta-\upeta(\alpha,t_k;\beta,\thetainv(t_k;\theta))
   )
  =
  \frac{
    \delta\bigl(\theta-\theta^{*,\mathsf{ex}}_{k,\ell}(\alpha;\beta)\bigr)
  }{
    \left|
      1
      -
      \frac{\mathrm{d}}{\mathrm{d}\theta}
      \upeta(\alpha,t_k;\beta,\thetainv(t_k;\theta))
      \big|_{\theta=\theta^{*,\mathsf{ex}}_{k,\ell}(\alpha;\beta)}
    \right|
  }.
\end{equation}
\end{subequations}

\subsubsection{Slab-frozen approximation}

Fix an observation time $t=t_k$ and a source slab $T_\ell=(t_{\ell-1},t_\ell]$.
We adopt a simple slab endpoint freezing strategy in which all
$\tau$-dependent quantities appearing in the kernel are evaluated at the
right endpoint $\tau=t_\ell$.
For $\tau\in T_\ell$, we make the following approximations. 

\begin{itemize}

\begin{subequations}\label{frozen-slab}
\item \textbf{Slab-frozen geometry:}
\begin{alignat}{3}
  z(\beta,\tau)
  &\;\approx\;
  z_{\ell}(\beta)
  &&\;\coloneqq\;
  z(\beta,t_\ell),
  \\
  J(\beta,\tau)
  &\;\approx\;
  J_{\ell}(\beta)
  &&\;\coloneqq\;
  J (\beta,t_\ell). 
\end{alignat}

\item \textbf{Slab-frozen travel-time:}
\begin{alignat}{3}\label{eta-frozen}
  \upeta(\alpha,t_k;\beta,\tau)
  &\;\approx\;
  \upeta_{k,\ell}(\alpha;\beta)
  &&\;\coloneqq\;
  \upeta(\alpha,t_k;\beta,t_\ell), \\
  \p_{\NN_y} \upeta(\alpha,t_k;\beta,\tau)
  &\;\approx\;
  \p_{\NN_y} \upeta_{k,\ell}(\alpha;\beta)
  &&\;\coloneqq\;
  \p_{\NN_y} \upeta(\alpha,t_k;\beta,t_\ell), \\
  \p_{\tau} \upeta(\alpha,t_k;\beta,\tau)
  &\;\approx\;
  \p_{\tau} \upeta_{k,\ell}(\alpha;\beta)
  &&\;\coloneqq\;
  \p_{\tau} \upeta(\alpha,t_k;\beta,t_\ell). 
\end{alignat}

\item \textbf{Slab-frozen amplitude:}
\begin{align}
  A(\alpha,t_k;\beta,\tau)
  &\;\approx\;
  A_{k,\ell}(\alpha;\beta)
  \;\coloneqq\;
  A(\alpha,t_k;\beta,t_\ell).
\end{align}
\end{subequations}

\end{itemize}
In the approximations above, $\approx$ denotes an $\mathcal{O}(\Delta t)$ 
approximation in time on each slab.
Under the slab-frozen approximation, the only remaining 
$\tau$-dependence in the temporal integrals arises
through the separation variable $\theta(t_k;\tau)$.

\subsubsection{Causal intersection under slabwise freezing}
\label{subsec:theta-limits-linearized}

\begin{subequations}\label{theta-limits-linearized}
Under the slab-frozen approximation \eqref{frozen-slab}, the approximate wavefront
intersection time is
\begin{equation}\label{theta-star-linearized-explicit}
  \theta^*_{k,\ell}(\alpha;\beta)
  =
  \upeta_{k,\ell}(\alpha;\beta).
\end{equation}
The corresponding physical intersection time is
\begin{equation}\label{tau-star-frozen}
  \tau^*_{k,\ell}(\alpha;\beta)
  =
  t_k
  - 
  \theta^*_{k,\ell}(\alpha;\beta)
  =
  t_k 
  - 
  \upeta_{k,\ell}(\alpha;\beta).
\end{equation}
\end{subequations}

Introduce the \emph{slab-frozen causal interaction set}
\begin{equation}\label{def:slab-causal-set}
  \mathcal C_{k,\ell}
  \coloneqq
  \left\{
    (\alpha,\beta)\in\mathbb{S}^2 \times \mathbb{S}^2
    :
    \theta_{k,\ell}
    \le
    \upeta_{k,\ell}(\alpha;\beta)
    \le
    \theta_{k,\ell-1}
  \right\}. 
\end{equation}
Because the parametrix kernel is supported on the
wavefront $\theta=\eta$, a slab contributes for a given pair
$(\alpha,\beta)$ if and only if
$(\alpha,\beta)\in\mathcal C_{k,\ell}$.
Two distinct geometric configurations arise:

\begin{enumerate}

\item \textbf{No intersection.}  
	If $(\alpha,\beta)\notin\mathcal C_{k,\ell}$,
      the vertical source-time segment over $T_\ell$
      does not intersect the discrete backward light cone.
	No contribution arises from this slab.

\item \textbf{Single intersection inside the slab.}  
If $(\alpha,\beta)\in\mathcal C_{k,\ell}$,
      the vertical source-time segment intersects the discrete light cone at the
      unique time $\tau^*_{k,\ell}(\alpha;\beta)$,
      and the slab contributes through evaluation at this
      intersection point via the identity 
      \begin{equation}\label{delta-wavefront-frozen}
      \delta (
        \theta-\upeta(\alpha,t_k;\beta,\thetainv(t_k;\theta))
       )
      =
        \delta\bigl(\theta-\theta^{*}_{k,\ell}(\alpha;\beta)\bigr), 
      \end{equation}
      which is \eqref{delta-wavefront} in the slab-frozen approximation.
\end{enumerate}

\begin{figure}[ht]
    \centering
    \begin{subfigure}[t]{.32\linewidth}
    \centering
    \includegraphics[width=0.99\linewidth]{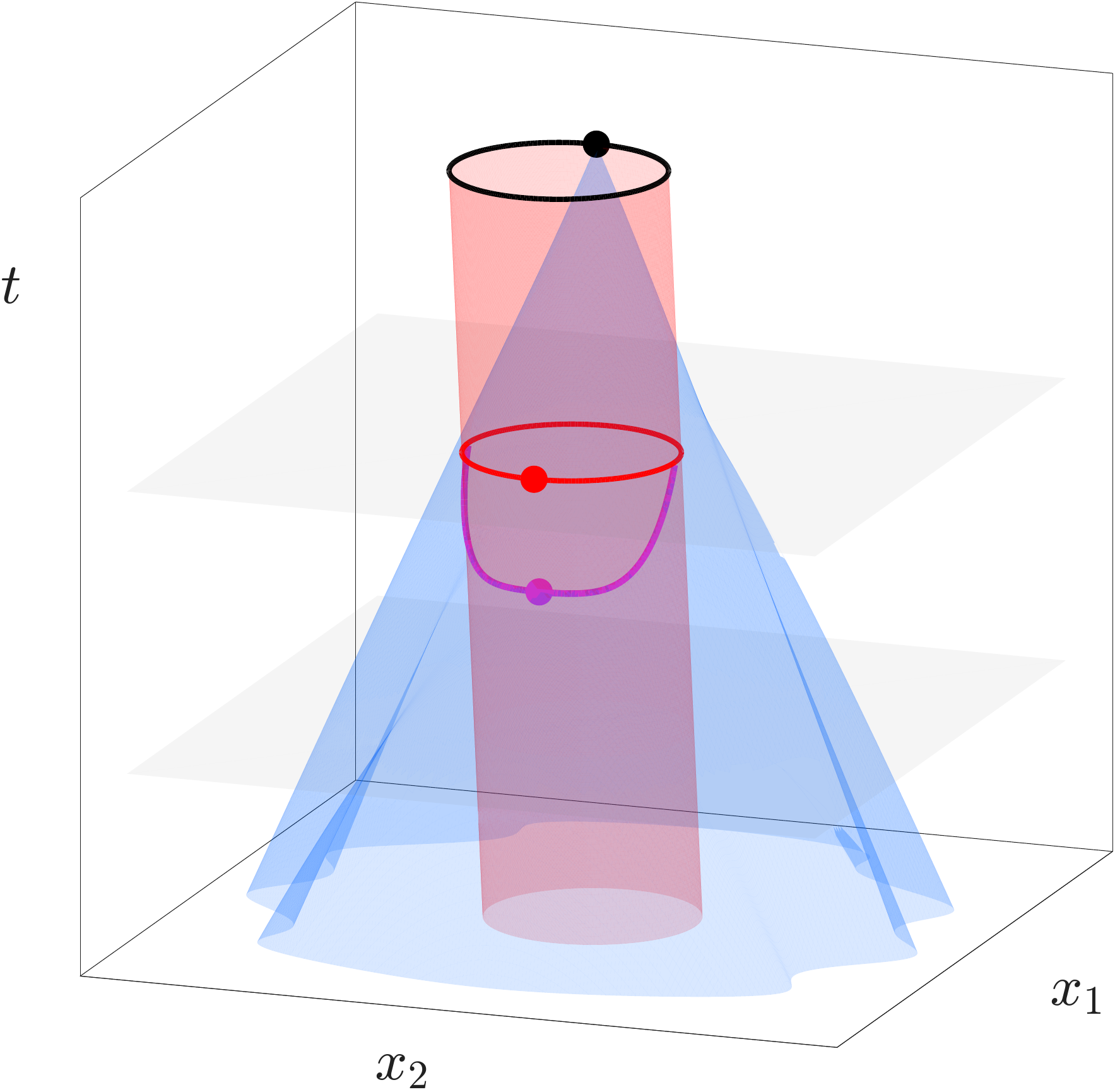}
    \caption{}
    \label{fig:worldsheet}
    \end{subfigure}
    \begin{subfigure}[t]{.32\linewidth}
    \centering
    \includegraphics[width=0.98\linewidth]{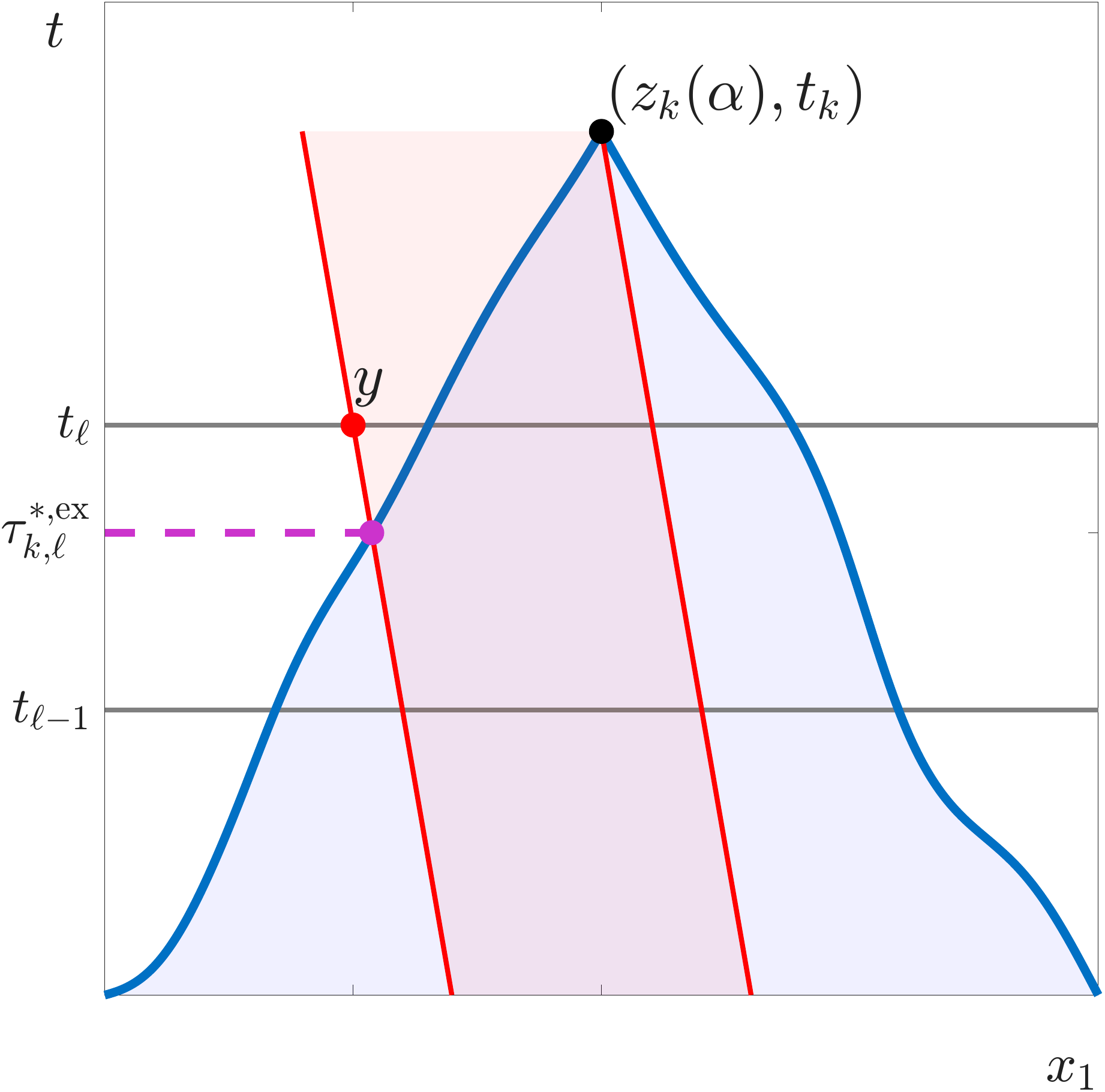}
    \caption{True causal geometry}
    \label{fig:frozen-slab-a}
    \end{subfigure}
    \begin{subfigure}[t]{0.32\linewidth}
    \centering
    \includegraphics[width=0.971\linewidth]{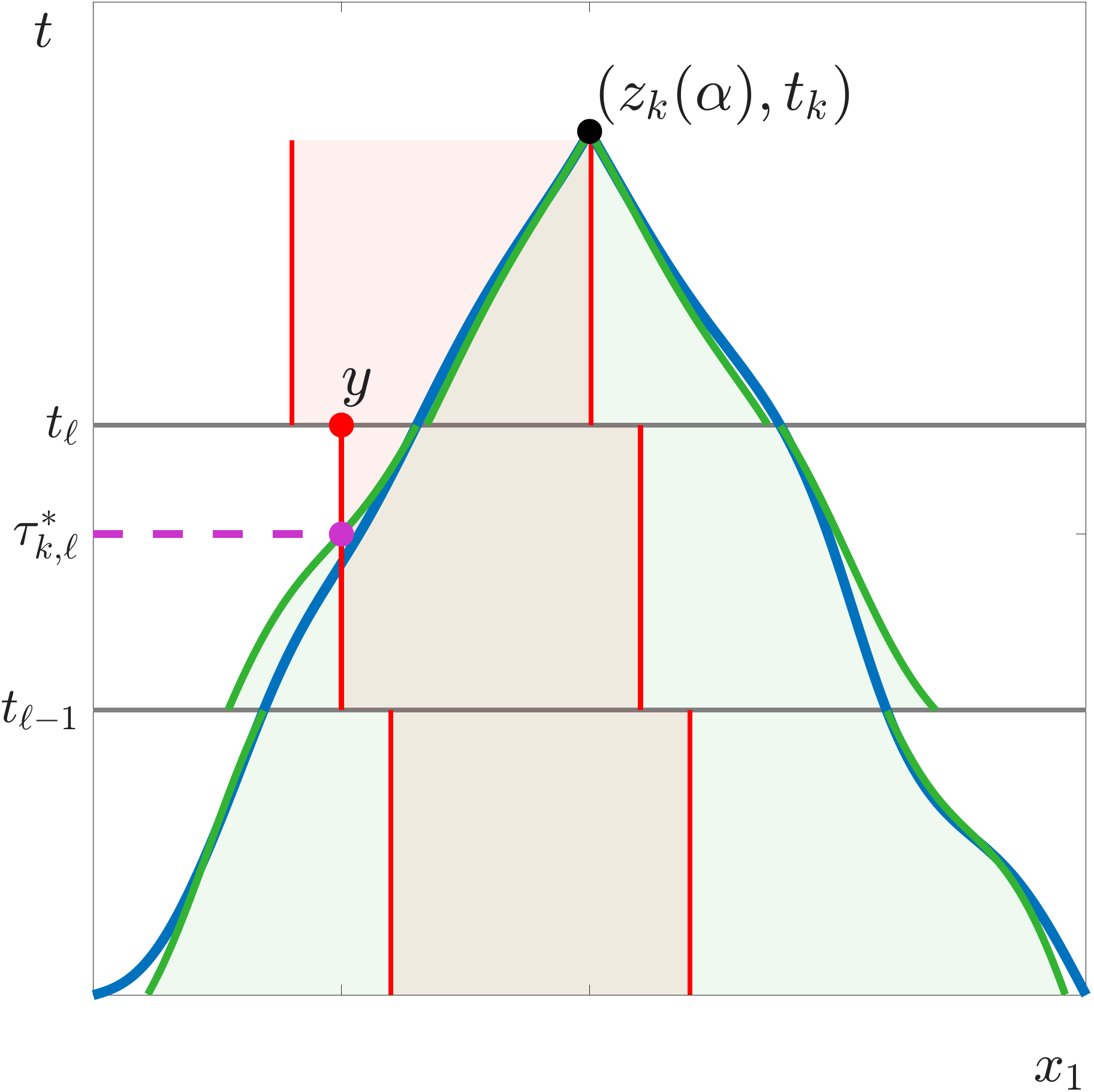}
    \caption{Slab-frozen causal geometry}
    \label{fig:frozen-slab-b}
    \end{subfigure}
\caption{
\textbf{Left:}
Backward light cone (blue surface) and worldsheet (red surface) plotted in $2\!+\!1$ space-time. 
The purple curve is the intersection of the light cone with the worldsheet. 
\textbf{Middle:}
Backward light cone emanating from $(z_k(\alpha),t_k)$ shown in a
one-dimensional spatial slice.
The blue shaded region represents the backward causal region, and the horizontal band
indicates the current source time slab $T_\ell=(t_{\ell-1},t_\ell]$.
The blue curve is the true backward light cone.
The red curve is the space-time trajectory of a material point on the interface (a slice of the worldsheet).
At time $t_\ell$, this trajectory passes through the source point 
$y = z_\ell(\beta)$; tracing the trajectory 
backwards in time, it intersects the cone at $\tau^{*,\mathsf{ex}}_{k,\ell}$.
\textbf{Right:}
Time-discrete slab-frozen approximation of the backward light cone on
$T_\ell=(t_{\ell-1},t_\ell]$.
The true cone (blue curve)
$t_k-\tau=\eta(z_k(\alpha),t_k;y,\tau)$
is replaced by its value at $\tau=t_\ell$, i.e.,
$t_k-\tau=\eta(z_k(\alpha),t_k;y,t_\ell)$ (green curve). 
Across slabs, the slab-frozen cone is generally discontinuous.
The true worldsheet is approximated on each slab by a
space-time cylinder, so that material trajectories become vertical segments.
The vertical segment $\tau\mapsto(y,\tau)\in T_\ell$ intersects the
approximate cone at $\tau^\ast_{k,\ell}$.
}
\label{fig:frozen-slab-geometry}
\end{figure}

We compare the exact and discrete causal geometry in \Cref{fig:frozen-slab-geometry}. 
The exact backward light cone emanating from $(z_k(\alpha),t_k)$ is the smooth surface
\[
  t_k - \tau = \eta(z_k(\alpha),t_k;y,\tau),
\]
displayed in blue in \Cref{fig:worldsheet}, with a 
one-dimensional slice shown in blue in \Cref{fig:frozen-slab-a,fig:frozen-slab-b}.
Recall the \emph{worldsheet} $\mathbf{\Gamma}$ of the 
interface, defined in \eqref{worldsheet-def}, and 
displayed in red in \Cref{fig:worldsheet,fig:frozen-slab-a}. 
The intersection time $\tau^{\ast,\mathsf{ex}}_{k,\ell}$ defined by \eqref{theta-star-def} is given by 
the intersection of the worldsheet with the backwards light cone, which is displayed as the 
purple curve and point in \Cref{fig:worldsheet,fig:frozen-slab-a}, respectively. 

In \Cref{fig:frozen-slab-b}, we illustrate the slab-frozen approximation 
of this causal geometry. 
The true backward light cone is
replaced on each $T_\ell$ by the slab-frozen surrogate
\[
  t_k - \tau = \eta_{k,\ell}(z_k(\alpha);y) \coloneqq \eta(z_k(\alpha), t_k ; y, t_\ell).  
\]
A one-dimensional slice of the slab-frozen cone 
is displayed as the green curve in \Cref{fig:frozen-slab-b}. 
Meanwhile, the \emph{slab-frozen worldsheet} is the space-time cylinder
\begin{equation}\label{frozen-worldsheet}
\mathbf{\Gamma}_\Delta = \bigcup_{\ell = 1}^k \{ (z_\ell(\beta), \tau) : \beta \in \mathbb{S}^2, \ \tau \in T_\ell \}. 
\end{equation}
The approximate intersection time
$\tau^\ast_{k,\ell}$, defined in \eqref{tau-star-frozen}, is
computed by intersecting the slab-frozen cone and worldsheet.
The slab-frozen causal interaction set
$\mathcal C_{k,\ell}$ consists of all source--target pairs
$(\alpha,\beta)$ whose corresponding intersection time
lies within the time slab $[t_{\ell-1},t_\ell]$.
Equivalently, it is induced by the intersection of the
slab-frozen cone and worldsheet restricted to the slab
$[t_{\ell-1},t_\ell]$.

\begin{remark}[Discrete approximations to the light cone]
The slab-frozen light cone is generally discontinuous across slab interfaces.
A continuous-in-time approximation is
obtained by replacing $\eta(z_k(\alpha),t_k;y,\tau)$ with the
piecewise linear approximation
\[
  \eta(z_k(\alpha),t_k;y,\tau)
  \approx
  \eta(z_k(\alpha), t_k ; y, t_{\ell-1})
  +
  \frac{\tau - t_{\ell-1}}{t_\ell - t_{\ell-1}}
  \bigl(
    \eta(z_k(\alpha), t_k ; y, t_{\ell})
    -
    \eta(z_k(\alpha), t_k ; y, t_{\ell-1})
  \bigr).
\]
Substitution yields a causal boundary that is affine in $\tau$ on each slab and continuous across interfaces. 
Similarly, a piecewise affine-in-time approximation to the worldsheet yields a surface that is continuous across slabs. 
Although a piecewise affine-in-time model provides a more faithful
geometric approximation of the light cone, slab endpoint freezing is adopted 
in the present formulation because it isolates all remaining $\tau$-dependence in
$\theta(t_k;\tau)$, leading to simple closed-form expressions for the slab integrals.
\end{remark}

\subsubsection{Closed-form expressions for the auxiliary temporal weights}

Applying the approximations
\eqref{frozen-slab} directly to the $\theta$-space representations
\eqref{sltw-theta} and \eqref{dltw-theta}, and evaluating 
the resulting integrals, we obtain the closed-form expressions
\begin{subequations}\label{sltw-theta-frozen-closed}
\begin{alignat}{3}
  \sltw^{(-)}_{k,\ell}(\alpha;\beta)
  & = 
  - 
   \frac{A_{k,\ell}(\alpha;\beta)\,J_\ell(\beta)}
         {4\pi \,\Delta t \,\upeta_{k,\ell}(\alpha;\beta) }
  \bigl(\theta_{k,\ell} - \upeta_{k,\ell}(\alpha;\beta) \bigr),
  \qquad
  &&(\alpha,\beta)\in\mathcal C_{k,\ell},
  \\
  \sltw^{(+)}_{k,\ell}(\alpha;\beta)
  &=
  \phantom{-}
  \frac{A_{k,\ell}(\alpha;\beta)\,J_\ell(\beta)}
       {4\pi \,\Delta t \,\upeta_{k,\ell}(\alpha;\beta) }
  \bigl(\theta_{k,\ell-1} - \upeta_{k,\ell}(\alpha;\beta) \bigr),
  \qquad
  &&(\alpha,\beta)\in\mathcal C_{k,\ell}, 
\end{alignat}
\end{subequations}
and
\begin{subequations}\label{dltw-theta-frozen-closed}
\begin{alignat}{3}
  \dltw^{\partial}_{k,\ell}(\alpha;\beta)
  &=
  \phantom{-}
  \frac{
    Q_{k,\ell}(\alpha;\beta)\,
    J_\ell(\beta)
  }{
    4\pi\,\Delta t \,\upeta_{k,\ell}(\alpha;\beta) 
    }\,
  , \
  &&(\alpha,\beta)\in\mathcal C_{k,\ell}, 
  \\
  \dltw^{(-)}_{k,\ell}(\alpha;\beta)
  &=
  -
  \frac{
    Q_{k,\ell}(\alpha;\beta)\,
    J_\ell(\beta)
  }{
    4\pi\,\Delta t\,\upeta^2_{k,\ell}(\alpha;\beta)
  }\,
  \bigl( \theta_{k,\ell} - \upeta_{k,\ell}(\alpha;\beta)  \bigr)\,
  , \
  &&(\alpha,\beta)\in\mathcal C_{k,\ell}, 
  \\
  \dltw^{(+)}_{k,\ell}(\alpha;\beta)
  &=
  \phantom{-}
  \frac{
    Q_{k,\ell}(\alpha;\beta)\,
    J_\ell(\beta)
  }{
    4\pi\,\Delta t  \,\upeta^2_{k,\ell}(\alpha;\beta)
    }\,
  \bigl(\theta_{k,\ell-1} - \upeta_{k,\ell}(\alpha;\beta) \bigr)\,
  ,\qquad
  &&(\alpha,\beta)\in\mathcal C_{k,\ell}, 
\end{alignat}
\end{subequations}
where both sets of auxiliary weights vanish for $(\alpha,\beta)\notin\mathcal C_{k,\ell}$.

\subsection{Spatial discretization}
\label{subsec:spatial-discretization}

For spatial discretization, we use a piecewise linear 
approximation for the boundary density, together with standard Gauss quadrature. 
More sophisticated spatial quadrature schemes \cite{SaSc2010,KlBaGr2013} 
provide improved accuracy and robustness, including spectrally accurate schemes 
specialized to smooth surfaces diffeomorphic to the sphere \cite{GaGrSi1998,GaGr2004,GiVe2013}.
In this work, the emphasis is placed on simplicity of implementation and on 
providing a baseline discretization for the parametrix boundary integral 
methodology.

We begin by discretizing the unit sphere $\mathbb{S}^2$
by a conforming triangular mesh, with
\begin{subequations}\label{spatial-mesh}
\begin{equation}\label{triangulation}
\mathcal{T}_h
=
\{ E \}_{E\in\mathcal{E}_h}
\end{equation}
denoting a shape-regular triangulation of $\mathbb{S}^2$, 
where each element $E \subset \mathbb{S}^2$ is a spherical triangle
and $\mathcal{E}_h$ denotes the set of all triangles. 
Let \begin{equation}\label{mesh-vertices}
\mathcal{T}_h^v = \{ \alpha_i \}_{i=1}^{N_h}
\end{equation}
\end{subequations}
be the set of mesh vertices,
and let $\mathcal{T}_h^e$ denote the set of mesh edges.
The mesh size $h > 0$ is defined as the maximum edge length in the triangulation
$h=\max_{e\in\mathcal{T}_h^e} |e|$.

For each fixed time $t$,
the discrete interface is obtained as the image of the unit sphere
under the parametrization $z(\cdot,t) : \mathbb{S}^2 \to \Gamma(t)$.
Define the physical triangles
\begin{subequations}\label{interface-discrete}
\begin{equation}
\Gamma_E(t)
=
z(E,t),
\qquad E\in\mathcal{E}_h,
\end{equation}
so that
\begin{equation}
\Gamma(t)
=
\bigcup_{E\in\mathcal{E}_h} \Gamma_E(t).
\end{equation}
\end{subequations}

%

\subsubsection{Reference element mapping and $\Pbb^1_\alpha$ basis}

Let $\hat E$ denote the reference element
\begin{equation}\label{ref-triangle}
\hat E
=
\bigl\{
(\xi_1,\xi_2)\in\mathbb{R}^2
:\;
\xi_1\ge0,\;
\xi_2\ge0,\;
\xi_1+\xi_2\le1
\bigr\}.
\end{equation}
For each element $E\in\mathcal{E}_h$
with vertices $\alpha_{i_1},\alpha_{i_2},\alpha_{i_3}\in\mathbb{S}^2$,
define the affine  map $\hat\alpha_E : \hat E \to E$
\begin{equation}\label{affine-triangle}
\hat\alpha_E(\xi_1,\xi_2)
=
\alpha_{i_1}
+
\xi_1(\alpha_{i_2}-\alpha_{i_1})
+
\xi_2(\alpha_{i_3}-\alpha_{i_1}).
\end{equation}
The Jacobian matrix of the $\hat\alpha_E$ is constant on $E$ 
and its determinant satisfies
\begin{equation}\label{triangle-jacobian}
\left|
\det D\hat\alpha_E
\right|
=
2 |E|,
\end{equation}
where $|E|$ denotes the area of the spherical triangle $E$.

On $\hat E$ we introduce the linear $\Pbb^1$ nodal basis
(barycentric basis functions)
\begin{subequations}\label{P1-basis}
\begin{align}
\hat\varphi_1(\xi_1,\xi_2)
&= 1-\xi_1-\xi_2, \label{P1-basis-a}\\
\hat\varphi_2(\xi_1,\xi_2)
&= \xi_1, \label{P1-basis-b}\\
\hat\varphi_3(\xi_1,\xi_2)
&= \xi_2, \label{P1-basis-c} \\
\hat\varphi_m(\hat\alpha_n) 
&= \delta_{mn}, \qquad m,n=1,2,3,  \label{P1-nodal-property}
\end{align}
\end{subequations}
where $\hat\alpha_n$ denotes the $n$-th vertex of $\hat E$.

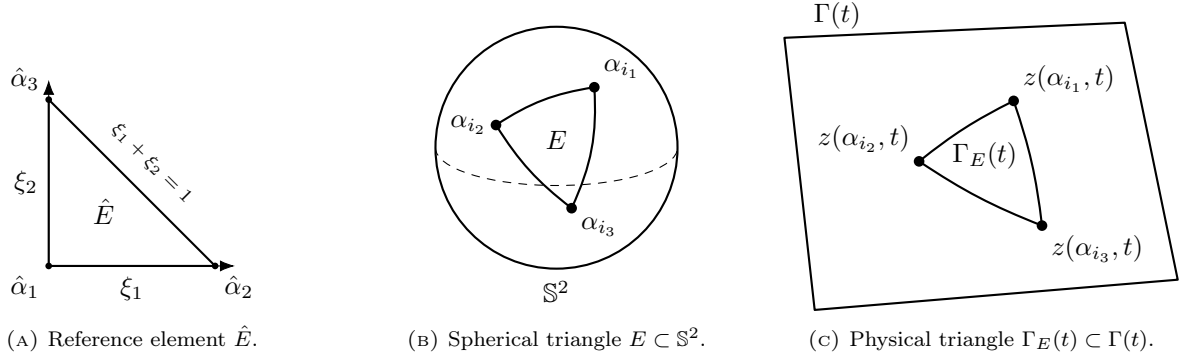
\begin{figure}[ht]
    \centering

    \begin{subfigure}{0.32\textwidth}
        \centering
        \begin{tikzpicture}[scale=2.2, line cap=round, line join=round]
            \coordinate (A) at (0,0);
            \coordinate (B) at (1,0);
            \coordinate (C) at (0,1);

            \draw[-{Latex[length=2mm]}, thin] (0,0) -- (1.12,0);
            \draw[-{Latex[length=2mm]}, thin] (0,0) -- (0,1.12);

            \draw[thick] (A) -- (B) -- (C) -- cycle;

            \fill (A) circle (0.020);
            \fill (B) circle (0.020);
            \fill (C) circle (0.020);

            \node[below left] at (A) {$\hat\alpha_1$};
            \node[below right] at (B) {$\hat\alpha_2$};
            \node[above left] at (C) {$\hat\alpha_3$};

            \node[below] at (0.5,-0.01) {$\xi_1$};
            \node[left]  at (-0.01,0.5) {$\xi_2$};

            \node at (0.33,0.33) {$\hat E$};

            \node[rotate=-45] at (0.62,0.62) {\scriptsize $\xi_1+\xi_2=1$};
        \end{tikzpicture}
        \caption{Reference element $\hat E$.}
        \label{fig:ref-element}
    \end{subfigure}
    \hfill
    \begin{subfigure}{0.32\textwidth}
        \centering
        \begin{tikzpicture}[scale=1.0, line cap=round, line join=round]
    
            \draw[thick] (0,0) circle (1.6);
    
            \draw[dashed, thin] (-1.6,0) arc (180:360:1.6 and 0.5);
    
            \coordinate (B1) at (0.5,0.8);
            \coordinate (B2) at (-0.8,0.3);
            \coordinate (B3) at (0.2,-0.8);
    
            \draw[thick] (B2) to[bend left=12]  (B1);
            \draw[thick] (B3) to[bend left=10]  (B2);
            \draw[thick] (B1) to[bend left=14]  (B3);
    
            \fill (B1) circle (2pt);
            \fill (B2) circle (2pt);
            \fill (B3) circle (2pt);
    
            \node[above right] at (B1) {$\alpha_{i_1}$};
            \node[left]        at (B2) {$\alpha_{i_2}$};
            \node[below right] at (B3) {$\alpha_{i_3}$};
    
            \node at (0.0,0.1) {$E$};
    
            \node at (0,-1.9) {$\mathbb{S}^2$};
    
        \end{tikzpicture}
        \caption{Spherical triangle $E\subset\mathbb{S}^2$.}
        \label{fig:sphere-element}
    \end{subfigure}
    \hfill
    \begin{subfigure}{0.32\textwidth}
        \centering
        \begin{tikzpicture}[scale=1.0, line cap=round, line join=round]
    
            \coordinate (S1) at (-2.4,-1.8);
            \coordinate (S2) at ( 2.4,-1.4);
            \coordinate (S3) at ( 1.7, 2.0);
            \coordinate (S4) at (-2.8, 1.8);
    
            \draw[thick] (S1) -- (S2) -- (S3) -- (S4) -- cycle;
    
            \coordinate (P1) at (0.20,0.82);
            \coordinate (P2) at (-0.80,0.18);
            \coordinate (P3) at (0.5,-0.5);
    
            \coordinate (Pc) at ($(P1)!1/3!(P2)!1/3!(P3)$);
    
            \def\s{1.25}
    
            \coordinate (Q1) at ($(Pc)!\s!(P1)$);
            \coordinate (Q2) at ($(Pc)!\s!(P2)$);
            \coordinate (Q3) at ($(Pc)!\s!(P3)$);
    
            \draw[thick] (Q2) to[bend left=6] (Q1);
            \draw[thick] (Q3) to[bend left=6] (Q2);
            \draw[thick] (Q1) to[bend left=6] (Q3);
    
            \fill (Q1) circle (2pt);
            \fill (Q2) circle (2pt);
            \fill (Q3) circle (2pt);
    
            \node[above right] at (Q1) {$z(\alpha_{i_1},t)$};
            \node[above left]    at (Q2) {$z(\alpha_{i_2},t)$};
            \node[below right] at (Q3) {$z(\alpha_{i_3},t)$};
    
            \node at (-2.1,2.1) {$\Gamma(t)$};
            \node at (-0.15,0.25) {$\Gamma_E(t)$};
    
        \end{tikzpicture}
        \caption{Physical triangle $\Gamma_E(t)\subset\Gamma(t)$.}
        \label{fig:phys-element}
    \end{subfigure}
    \caption{
    Element maps used in the $\Pbb^1$ discretization.
    \textbf{Left:} Reference element $\hat E$.
    \textbf{Middle:} Image under the affine map $\hat\alpha_E$ onto
    $E\subset\mathbb{S}^2$.
    \textbf{Right:} Physical surface triangle
    $\Gamma_E(t)=z(E,t)$ on $\Gamma(t)$.
    }
    \label{fig:element-maps}
\end{figure}

\subsubsection{$\Pbb^1$ approximation of boundary densities}

We approximate the boundary densities
$\sigma_\ell(\alpha)$ and $\mu_\ell(\alpha)$,
$\alpha\in\mathbb{S}^2$,
by continuous piecewise linear functions
on $\mathcal{T}_h$.
Let $\{\alpha_i\}_{i=1}^{N_h} \in \mathcal{T}_h^v$ denote the mesh vertices.
Define the nodal values
\begin{equation}\label{nodal-dofs}
\sigma_\ell^i
\coloneqq \sigma_\ell(\alpha_i), 
\qquad 
\text{and} 
\qquad
\mu_\ell^i
\coloneqq \mu_\ell(\alpha_i). 
\end{equation}
For each element $E$ with vertices
$\alpha_{i_1},\alpha_{i_2},\alpha_{i_3}$,
the densities are represented locally as
\begin{align}\label{P1-sigma-mu}
\sigma_\ell(\alpha) 
=
\sum_{m=1}^3
\sigma_\ell^{\,i_m}\,
\varphi_{i_m}(\alpha),
\qquad
\text{and} \qquad 
\mu_\ell(\alpha) 
=
\sum_{m=1}^3
\mu_\ell^{\,i_m}\,
\varphi_{i_m}(\alpha),
\qquad
\alpha \in E.
\end{align}
where the local shape functions are defined by
\begin{equation}\label{P1-shape-functions}
\varphi_{i_m}(\alpha)
=
\hat\varphi_m
\left(
\hat\alpha_E^{-1}(\alpha)
\right).
\end{equation}
These local functions patch together to define the global $\mathbb P^1$ nodal
basis. More precisely, for each global vertex $\alpha_i\in\mathcal T_h^v$, the
associated global basis function $\varphi_i$ is defined by
\begin{equation}\label{global-P1-basis}
  \varphi_i|_E
  =
  \begin{cases}
    \varphi_{i_m}, & \text{if } \alpha_i \text{ is the $m$-th vertex of } E,\\
    0, & \text{if } \alpha_i\notin E,
  \end{cases}
\end{equation}
for every element $E\in\mathcal E_h$.
In particular,
\begin{equation}\label{global-P1-nodal-property}
  \varphi_i(\alpha_j)=\delta_{ij},
  \qquad i,j=1,\ldots,N_h.
\end{equation}

\subsubsection{Space-time weights}

Collocating the single-layer semi-discrete operator \eqref{op-semi-discrete} at the
target vertex $\alpha_r \in \mathcal{T}_h^v$ gives
\[
\Sop_{k,\ell}[\sigma_\ell](\alpha_r)
=
\int_{\mathbb{S}^2}
  \sigma_\ell(\beta)\,
  \sltw_{k,\ell}(\alpha_r;\beta)
  \,\dd \beta,
\qquad r=1,\ldots,N_h.
\]
Decomposing the integral into spherical triangles $E\in\mathcal{E}_h$ yields
\[
\Sop_{k,\ell}[\sigma_\ell](\alpha_r)
=
\sum_{E\in\mathcal{E}_h}
\int_{E}
  \sigma_\ell(\beta)\,
  \sltw_{k,\ell}(\alpha_r;\beta)
  \,\dd \beta .
\]
Substituting the $\Pbb^1$ approximation \eqref{P1-sigma-mu}, we isolate the
dependence on the nodal values $\{\sigma_\ell^i\}_{i=1}^{N_h}$ by introducing
the elementwise auxiliary space-time weights:
for each element $E$ with vertices $\alpha_{i_1},\alpha_{i_2},\alpha_{i_3}$, define
\begin{equation}\label{slstw}
\hat\slstw_{k,\ell}^{\,r ; \, m} (E)
\coloneqq
\int_{E}
  \varphi_{i_m}(\beta)\,
  \sltw_{k,\ell}(\alpha_r;\beta)
  \,\dd \beta,
\qquad m=1,2,3.
\end{equation}
Then
\[
\Sop_{k,\ell}[\sigma_\ell](\alpha_r)
=
\sum_{E\in\mathcal{E}_h}
\sum_{m=1}^3
\hat\slstw_{k,\ell}^{\,r;\,m} (E) \,
\sigma_\ell^{\,i_m}.
\]
To obtain a single weight for each global vertex index $i=1,\ldots,N_h$, we
assemble all element contributions associated with that vertex.  If
$i_m(E)$ denotes the global index of the $m$-th local vertex of $E$, we define
the effective space-time weight by
\begin{equation}\label{slstw-effective}
\slstw_{k,\ell}^{\,r,i}
\coloneqq
\sum_{E\in\mathcal{E}_h}
\sum_{m=1}^3
\delta_{i,i_m(E)}\,
\hat\slstw_{k,\ell}^{\,r;\,m} (E) ,
\qquad i=1,\ldots,N_h,
\end{equation}
where $\delta_{ij}$ denotes the Kronecker delta.

Similarly, define
\begin{equation}\label{dlstw}
\hat\dlstw_{k,\ell}^{\,r;\,m} (E)
\coloneqq
\int_{E}
  \varphi_{i_m}(\beta)\,
  \dltw_{k,\ell}(\alpha_r;\beta)
  \,\dd \beta,
\qquad m=1,2,3, 
\end{equation}
and the effective space-time weight
\begin{equation}\label{dlstw-effective}
\dlstw_{k,\ell}^{\,r,i}
\coloneqq
\sum_{K\in\mathcal{E}_h}
\sum_{m=1}^3
\delta_{\,i,\,i_m(E)}\,
\hat\dlstw_{k,\ell}^{\,r;\,m}(E). 
\qquad i=1,\ldots,N_h,
\end{equation}

The discrete single-layer and double-layer operators then take the nodal form
\begin{align}
\Sop_{k,\ell}[\sigma_\ell](\alpha_r)
&=
\sum_{i=1}^{N_h}
\slstw_{k,\ell}^{\,r,i}\,
\sigma_\ell^{i},
\qquad \text{and} \qquad
\Dop_{k,\ell}[\mu_\ell](\alpha_r)
=
\sum_{i=1}^{N_h}
\dlstw_{k,\ell}^{\,r,i}\,
\mu_\ell^{i},
\qquad \text{for } r=1,\ldots,N_h.
\end{align}

\subsection{Quadrature and discrete space-time weights}

To evaluate the space-time weights
\eqref{slstw} and \eqref{dlstw},
we work elementwise on each triangle
$E\in\mathcal{E}_h$
using the affine mapping
\eqref{affine-triangle}
from the reference triangle $\hat E$.
Let $\{ (\xi_a^1,\xi_a^2), \, \omega_a \}_{a=1}^{Q}$
denote a $Q$-point Gaussian quadrature rule on $\hat E$.
The corresponding quadrature nodes on $E$ are
$
\beta_E^a
=
\hat\alpha_E(\xi_a^1,\xi_a^2)$, for $a=1,\ldots,Q$.
Under the change of variables $\beta=\hat\alpha_E(\xi)$,
the surface element transforms as
$\dd \beta
=
\left|
\det D\hat\alpha_E
\right|
\dd \xi
=
2|E| \, \dd \xi$.

We approximate the elementwise single-layer and double-layer
weights by standard triangular Gaussian quadrature:
\begin{equation}\label{slstw-discrete}
\hat\slstw_{k,\ell}^{\,r;\,m} (E)
=
2|E|
\sum_{a=1}^{Q}
\omega_a
\,
\hat\varphi_m(\xi_a^1,\xi_a^2)
\,
\sltw_{k,\ell}
\bigl(
\alpha_r; \beta_E^a
\bigr)
\;+\;
\mathbb{E}_{k,\ell}^{\,r;E;m},
\end{equation}
and
\begin{equation}\label{dlstw-discrete}
\hat\dlstw_{k,\ell}^{\,r;\,m} (E)
=
2|E|
\sum_{a=1}^{Q}
\omega_a
\,
\hat\varphi_m(\xi_a^1,\xi_a^2)
\,
\dltw_{k,\ell}
\bigl(
\alpha_r; \beta_E^a
\bigr)
\;+\;
\widetilde{\mathbb{E}}_{k,\ell}^{\,r;E;m}.
\end{equation}
If the corresponding integrands are smooth on $E$
(in particular, $C^{2Q}$ on $E$),
then the Gaussian rule is high-order accurate and
$
\mathbb{E}_{k,\ell}^{(\cdot)}
=
\mathcal{O}\!\left(h^{2Q}\right)
$
and
$
\widetilde{\mathbb{E}}_{k,\ell}^{(\cdot)}
=
\mathcal{O}\!\left(h^{2Q}\right).
$
If the integrand loses regularity on $E$,
for example due to a causal cutoff across the triangle,
then the quadrature error degrades to the algebraic order
permitted by the available regularity.

\subsection{Fully discrete marching-on-time system}

We now assemble the fully discrete marching-on-time (MOT) systems resulting
from the space-time discretizations introduced above.
In both the single-layer and double-layer formulations, collocation at time
$t=t_k$ yields a linear system relating the unknown boundary density at the
current time level to previously computed values.

Throughout this subsection, denote by
\begin{subequations}\label{F-nodal}
\begin{equation}
F_k^{r} = F(\alpha_{r},t_k),
\quad r=1,\ldots,N_h,
\end{equation}
the discrete boundary data at time $t_k$ evaluated at the mesh vertices
$\{\alpha_r\}_{r=1}^{N_h}=\mathcal{T}_h^v$ and let
\begin{equation}\label{F-vector}
F_k
=
\bigl(
F_k^{1}, F_k^{2}, \ldots, F_k^{N_h}
\bigr)^{\top}
\in \mathbb{R}^{N_h}.
\end{equation}
\end{subequations}

\subsubsection{Single-layer formulation}

Denote the nodal density values by
\begin{subequations}\label{sigma-nodal}
\begin{equation}
\sigma_k^{r}
=
\sigma(\alpha_r,t_k),
\quad r=1,\ldots,N_h,
\end{equation}
and define the vector
\begin{equation}\label{sigma-vector}
\sigma_k
=
\bigl(
\sigma_k^{1}, \sigma_k^{2}, \ldots, \sigma_k^{N_h}
\bigr)^{\top}
\in \mathbb{R}^{N_h}.
\end{equation}
\end{subequations}
For each observation vertex $\alpha_r$, the discrete collocation equation reads
\begin{equation}\label{SL-mot-discrete-expanded}
\sum_{\ell=1}^{k}
\sum_{i=1}^{N_h}
\slstw_{k,\ell}^{\,r,i}\,\sigma_\ell^{i}
\;=\;
F_k^{r}.
\end{equation}
Introducing the matrices
\begin{equation}\label{Skell-def}
(\mathbf S_{k,\ell})_{r,i}
=
\slstw_{k,\ell}^{\,r,i},
\qquad
r,i=1,\ldots,N_h,
\end{equation}
and separating the contribution at the current time level, 
the system \eqref{SL-mot-discrete-expanded} becomes 
the marching recursion
\begin{equation}\label{SL-mot-recursion}
\mathbf S_{k,k}\,\sigma_k
=
F_k
-
\sum_{\ell=1}^{k-1}
\mathbf S_{k,\ell}\,\sigma_\ell,
\qquad k=1,\ldots,K.
\end{equation}

\subsubsection{Double-layer formulation}

Denote the nodal double-layer density values by
\begin{subequations}\label{mu-nodal}
\begin{equation}
\mu_k^{r}
=
\mu(\alpha_r,t_k),
\quad r=1,\ldots,N_h,
\end{equation}
and define the vector
\begin{equation}\label{mu-vector}
\mu_k
=
\bigl(
\mu_k^{1}, \mu_k^{2}, \ldots, \mu_k^{N_h}
\bigr)^{\top}
\in \mathbb{R}^{N_h}.
\end{equation}
\end{subequations}
For each observation vertex $\alpha_r$, the discrete collocation equation
corresponding to \eqref{BIE-DL} reads
\begin{equation}\label{DL-mot-discrete-expanded}
\tfrac12\,\lambda_k^r \,\mu_k^{r}
\;+\;
\sum_{\ell=1}^{k}
\sum_{i=1}^{N_h}
\dlstw_{k,\ell}^{\,r,i}\,\mu_\ell^{i}
\;=\;
F_k^{r},
\end{equation}
where
\[
\lambda_k^r = \lambda(\alpha_r,t_k).
\]
Introducing the matrices
\begin{equation}\label{Dkell-def}
(\mathbf D_{k,\ell})_{r,i}
=
\dlstw_{k,\ell}^{\,r,i},
\qquad
r,i=1,\ldots,N_h,
\end{equation}
and
\begin{equation}\label{Lambda-k-def}
\mathbf \Lambda_k
=
\operatorname{diag}\bigl(\lambda_k^1,\lambda_k^2,\ldots,\lambda_k^{N_h}\bigr),
\end{equation}
the system becomes
\begin{equation}\label{DL-mot-recursion}
\bigl(\tfrac12 \mathbf \Lambda_k + \mathbf D_{k,k}\bigr)\mu_k
=
F_k
-
\sum_{\ell=1}^{k-1}
\mathbf D_{k,\ell}\,\mu_\ell,
\qquad k=1,\ldots,K.
\end{equation}

\subsubsection{Kirchhoff formulation}

Using the notation introduced in
\eqref{sigma-nodal}, \eqref{Skell-def}, \eqref{Dkell-def}, and \eqref{Lambda-k-def},
we write the discrete marching recursion for the Kirchhoff boundary integral equation \eqref{BIE-K} as
\begin{equation}\label{K-mot-recursion}
\mathbf S_{k,k}\,\sigma_k
=
\left(-\mathbf I+\tfrac12 \mathbf \Lambda_k\right)F_k
+
\sum_{\ell=1}^{k}
\mathbf D_{k,\ell}\,F_\ell
-
\sum_{\ell=1}^{k-1}
\mathbf S_{k,\ell}\,\sigma_\ell,
\qquad k=1,\ldots,K.
\end{equation}

\subsection{Numerical stabilization}\label{rmk:stable}
To control spurious high-frequency oscillations, we employ 
two complementary stabilization procedures.
First, we use the Rynne time-averaging technique \cite{Rynne1990,DaDu1997},
in which the density at the previous time step is replaced by a local
temporal average, thereby damping high-frequency components in the
marching history. 
Second, we apply spatial smoothing on the surface $\Gamma(t_k)$.
Given the density $\sigma$, a smoothed density $\tilde\sigma$ is
obtained by performing $n_{\mathrm{pass}}$ implicit steps of a surface
heat equation,
\begin{equation}\label{rynne-params}
(M + \nu L)\,\sigma^{(\ell+1)} = M\,\sigma^{(\ell)}, 
\quad \ell=0,\dots,n_{\mathrm{pass}}-1,
\quad \sigma^{(0)}=\sigma,
\quad \tilde\sigma=\sigma^{(n_{\mathrm{pass}})}.
\end{equation}
Here $\nu > 0$ is a viscosity parameter, $M$ is the lumped mass matrix, and
$L$ is the cotangent stiffness matrix associated with the
Laplace--Beltrami operator on the triangulated surface. 
High-frequency spatial modes are damped in the density $\tilde{\sigma}$, which is 
subsequently used in the marching recursion. 
For all the numerical experiments considered in this work we use the parameter values
$\nu=\frac12 h^2$ and $n_{\mathrm{pass}}=10$.

\subsection{Evaluation of acoustic fields}
\label{subsec:field-eval}

Once the boundary densities have been computed, acoustic fields may be
evaluated \emph{a posteriori} on a background mesh covering the computational
domain. At a fixed observation time $t_k$, the reflected wavefront is given by the
level set
\begin{equation}
t_k = \Trfl(x),
\end{equation}
where the reflected arrival time $\Trfl(x)$ is defined in
\eqref{Tref-def}. In practice, the arrival time $\Tarr(x;y,\tau)$ is computed
using the same ray-based travel-time approximation used in the boundary
discretization, and $\Trfl(x)$ is obtained by minimizing over all active
source events.

The scattered field is supported in the reflected 
causal region $\{ x \in \Omega^+(t_k) : \Trfl(x) \le t_k \}$,
with the wavefront itself given by the level set $t_k = \Trfl(x)$.
The computation of the scattered field $\tilde \phi$ by \eqref{layer-ansatze} 
is restricted either to the interior of the reflected wavefront or to a narrow band about the front,
\begin{equation}
  \bigl| \Trfl(x) - t_k \bigr| \le \delta,
\end{equation}
for a prescribed band width $\delta > 0$.
The acoustic potential $\tilde\phi$ is then evaluated by applying the same space--time
quadrature used in the boundary collocation scheme,
with the boundary-traced kernel replaced by its full space--time counterpart.

\begin{remark}[Two-dimensional setting]\label{remark-2d}
Although the travel-time parametrix construction was presented in three
dimensions, the same framework applies in two dimensions with only minor
modifications. In particular, the three-dimensional reference kernel
$\Gop_0$ is replaced by its two-dimensional counterpart
$\Gop_0^{\mathsf{2d}}$, and the associated temporal weight formulas are
modified accordingly.
In three dimensions, the fundamental solution is 
supported on the boundary of the causal cone,
whereas in two dimensions it has support throughout the interior of the
cone. Consequently, the temporal weights arise from integrating over
causal intervals rather than evaluating contributions on a sharp
wavefront.
\end{remark}

\section{Motion-induced scattering in a homogeneous medium}
\label{sec:const}

In this section we consider the operator \eqref{L-general-pde} with constant sound speed
\begin{equation}\label{c-const}
  c(x,t)
  \;\equiv\;
  C > 0. 
\end{equation}
The unique solution to the eikonal equation \eqref{eikonal-delay-st-a} is
\begin{subequations}\label{const-eikonal-transport}
\begin{equation}\label{eta-const}
  \eta(x;y) = C^{-1} | x - y | . 
\end{equation}
Substituting \eqref{eta-const} into the transport equation
\eqref{transport-eq} yields
\[
  (x-y)\cdot \nabla_x \amp(x,t;y,\tau)
  =
  -\frac{|x-y|}{C}\,\partial_t \amp(x,t;y,\tau),
  \qquad x\neq y.
\]
In the constant-speed setting, $\amp$ is independent of $t$, and hence
\[
  (x-y)\cdot \nabla_x \amp(x,t;y,\tau) = 0,
  \qquad x\neq y.
\]
The normalization \eqref{A-normalization-st} then implies
\begin{equation}
  \amp(x,t;y,\tau)\equiv C^{-1}. 
\end{equation}
\end{subequations}
With the choices \eqref{const-eikonal-transport}, the representation 
\eqref{G-pmtx} is exact, so that the potentials $\tilde\phi$ defined 
by \eqref{layer-ansatze} are exact solutions to the constant-speed 
wave operator with boundary conditions \eqref{L-bcs}.

\subsection{Numerical instability and stabilization for moving interfaces}

The aim of this section is to demonstrate the onset of numerical instability in the
time-marching scheme for moving interfaces, and the need for an
explicit stabilization mechanism.
To this end, we employ \emph{manufactured solutions}: 
we prescribe a target boundary density (either $\sigma(\alpha,t)$ for the
single-layer formulation or $\mu(\alpha,t)$ for the double-layer
formulation) and determine boundary data $F(\alpha,t)$ so that the
corresponding boundary integral equation is satisfied.

We perform these tests in two dimensions to enable systematic accuracy
studies with manufactured solutions in a setting where exact solutions
can be constructed efficiently.
Set the sound speed $c \equiv 1$ and take the interface
$\Gamma(t)$ to be a circle of radius $R(t)$. In this setting, the boundary
integral operators diagonalize with respect to the Fourier modes
$e^{i n \alpha}$, reducing the construction to independent scalar
relations in time that can be solved explicitly.
In all tests, we take the target boundary density to be the radially symmetric function
\begin{equation}\label{exact-sol}
\sigma(\alpha,t) = t \cos(\omega t)\, H(t), \qquad \text{(and similarly for $\mu$)}, 
\end{equation}
with $\omega = 6\pi$.

\subsubsection{Fixed interface: baseline accuracy}

We consider a fixed interface with $R(t) \equiv 1$.
Both single- and double-layer solvers are employed \emph{without} the
stabilization procedures introduced in \Cref{rmk:stable}, and convergence is assessed using
$N = 50, 100, 200, 400$ cells with time step $\Delta t = 5/(2N)$.  
\Cref{fig:fixed_convergence} displays the numerical results at the final time $\tmax = 1$.
The left panel shows the time history of the spatially averaged
numerical densities $\sigma_\Delta$ and $\mu_\Delta$ for $N=200$,
compared with the exact solution $f(t)$; both exhibit good qualitative
agreement with the exact solution.
The middle panel shows the
numerical single-layer potential $\phi_\Delta(x,t)$ at the final time. 
The right panel shows the convergence of the $L^2_\alpha L^2_t$
and $L^\infty_\alpha L^\infty_t$ errors as $N$ is refined.
Because no special quadrature techniques are employed
here to treat the (weakly) singular integrals, the overall accuracy is limited to first order. The
single-layer solver exhibits this behavior, while the double-layer solver
shows larger errors and a reduced convergence rate, reflecting its
greater sensitivity to quadrature error. More accurate quadrature
techniques \cite{Steinbach2008} can be used to improve these results, but this is not pursued
in the present work.

\begin{figure}[ht]
    \centering
    \begin{subfigure}{0.32\textwidth}
        \centering
        \includegraphics[width=0.86\textwidth]{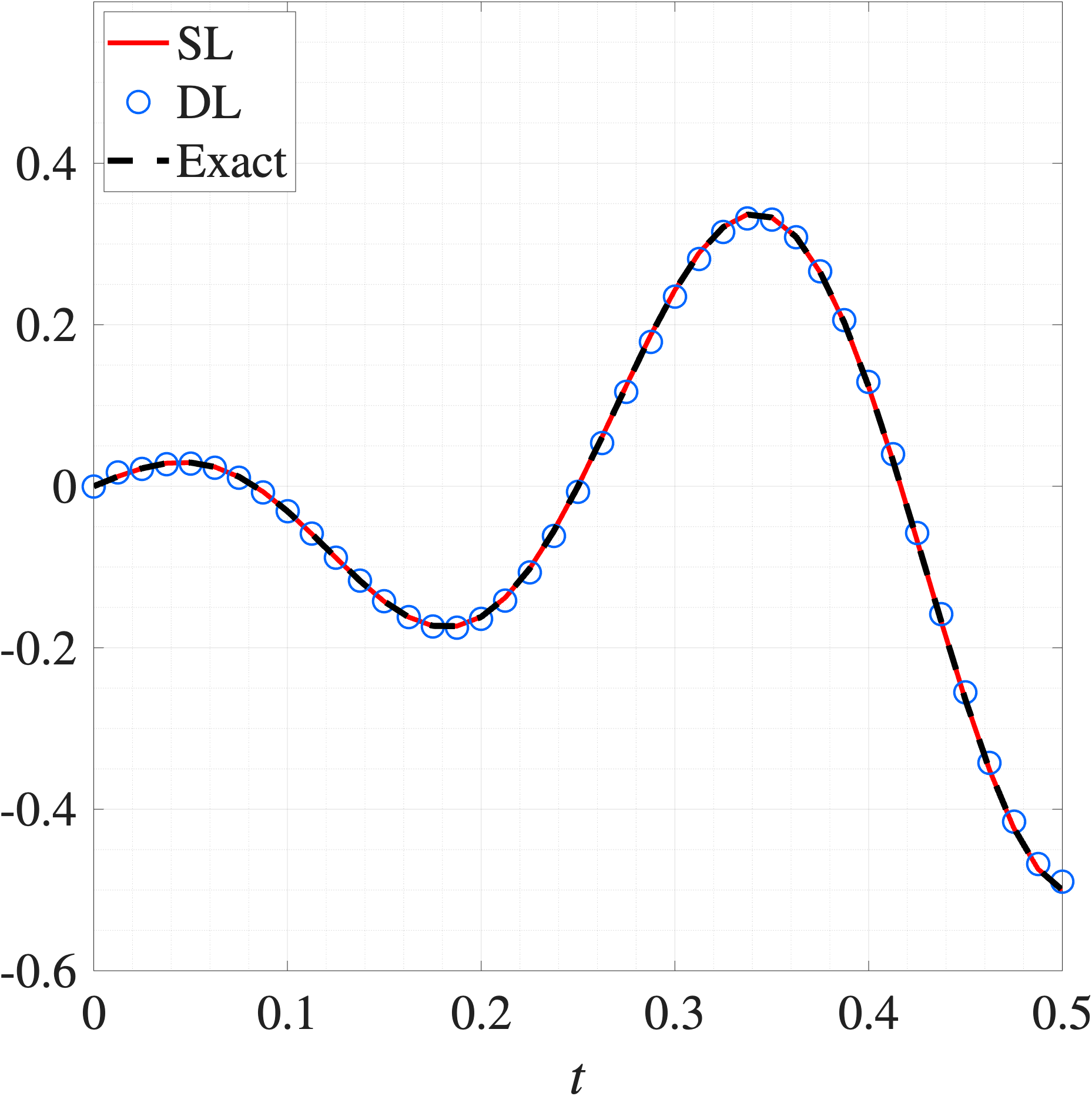}
        \caption{Numerical vs exact density}
    \end{subfigure}
    \hfill
    \begin{subfigure}{0.32\textwidth}
        \centering
        \includegraphics[width=\textwidth]{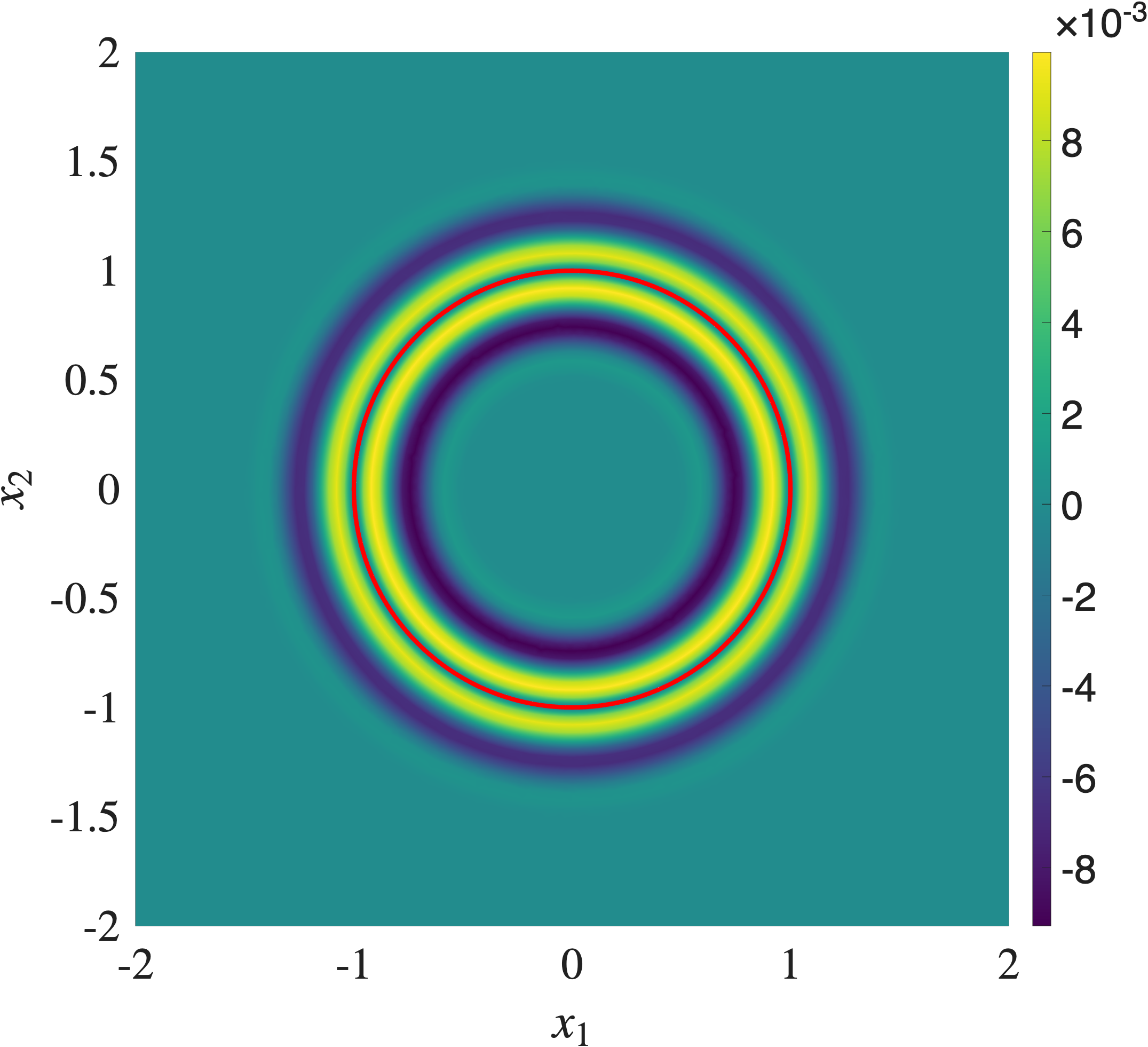}
        \caption{SL numerical potential}
    \end{subfigure}
    \hfill
    \begin{subfigure}{0.32\textwidth}
        \centering
        \includegraphics[width=0.88\textwidth]{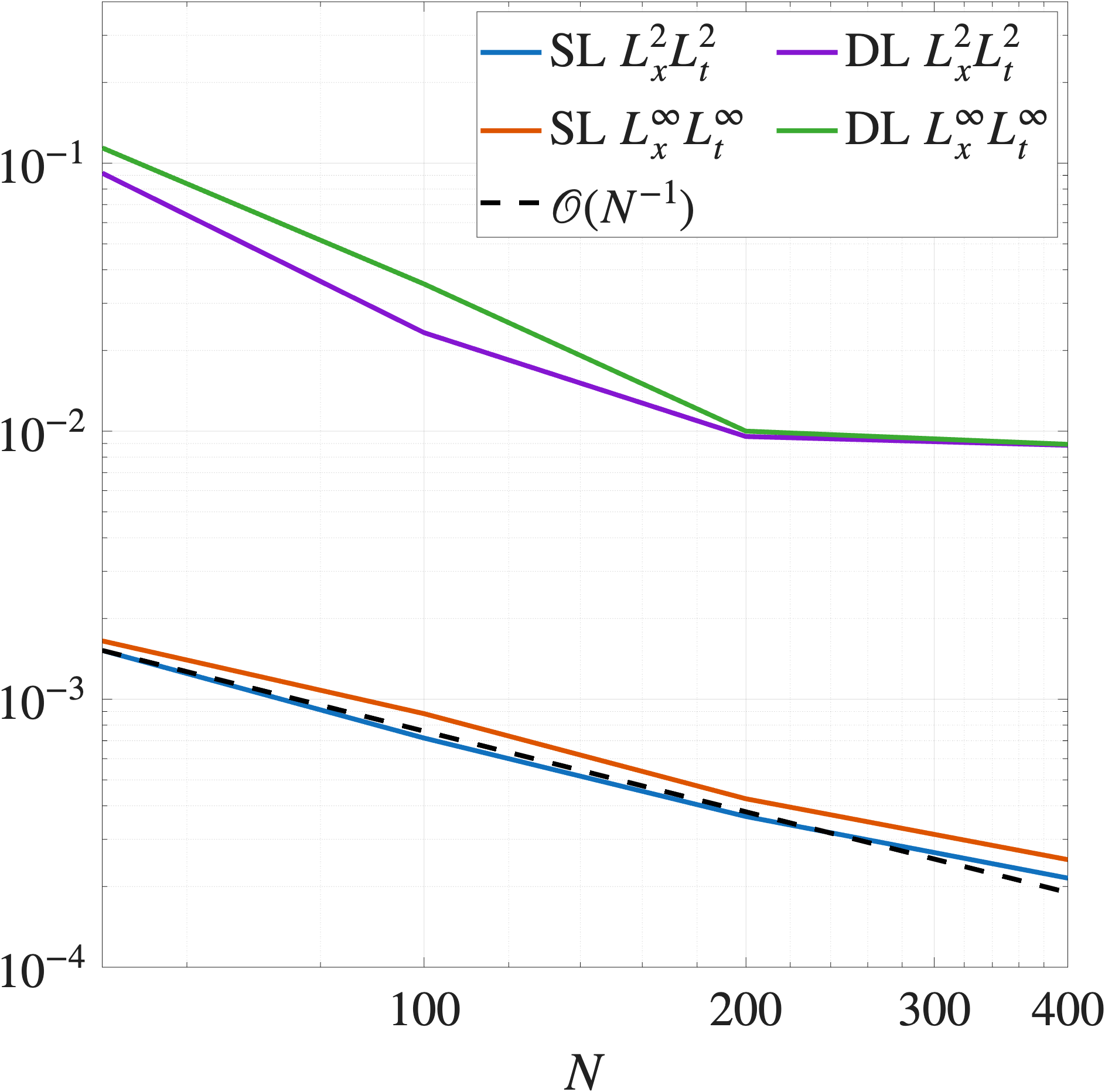}
        \caption{Convergence of relative error}
    \end{subfigure}
    \caption{%
    Fixed circle $R(t) = 1$ test.
    \textbf{Left:} Time history of the spatially averaged boundary density,
    comparing the numerical solutions $\sigma_\Delta$ and $\mu_\Delta$
    for $N=200$ with the exact solution.
    \textbf{Center:} Numerical single-layer potential $\phi_\Delta(x,t)$ obtained from the $N=200$ solution at the final time.
    \textbf{Right:} Convergence of the error in the $L^2_x L^2_t$ and $L^\infty_x L^\infty_t$ norms 
    for both single-layer and double-layer formulations.
    }
    \label{fig:fixed_convergence}
\end{figure}

\subsubsection{Moving interface: the need for stabilization}

We now consider the case of a moving interface in order to assess the
robustness of the scheme in a dynamic setting. Specifically, we take
$R(t) = 1 + U t$, where $U = 0.5$ denotes the (dimensionless) interface speed, i.e., the Mach number.
The problem is simulated using $N=200$ cells up to the final time
$\tmax = 1$, with time step $\Delta t = 0.01$. Both the single- and
double-layer solvers were used.

The left panel in \Cref{fig:test2} shows the single-layer density
computed without any explicit stabilization. Although the numerical profile initially agrees
well with the exact solution, spurious oscillations eventually emerge and the
approximation becomes unstable as time evolves. 
This behavior demonstrates the need for an explicit stabilization
mechanism, which in our solver is provided by the Rynne time-averaging
procedure (c.f. \Cref{rmk:stable}). With this stabilization in place, both the SL and DL
solutions remain stable, as shown in the middle panel. 
We observe that the double-layer solution exhibits a larger error than
the single-layer solution, primarily due to a small phase shift relative
to the exact solution.
The right panel shows the scattered potential $\phi_\Delta(x,t)$
corresponding to the stabilized solution at the final time. The resulting
field is smooth and captures the expected outward propagation induced by
the expanding interface.

\begin{figure}[ht]
    \centering
    \begin{subfigure}{0.32\textwidth}
        \centering
        \includegraphics[width=0.86\textwidth]{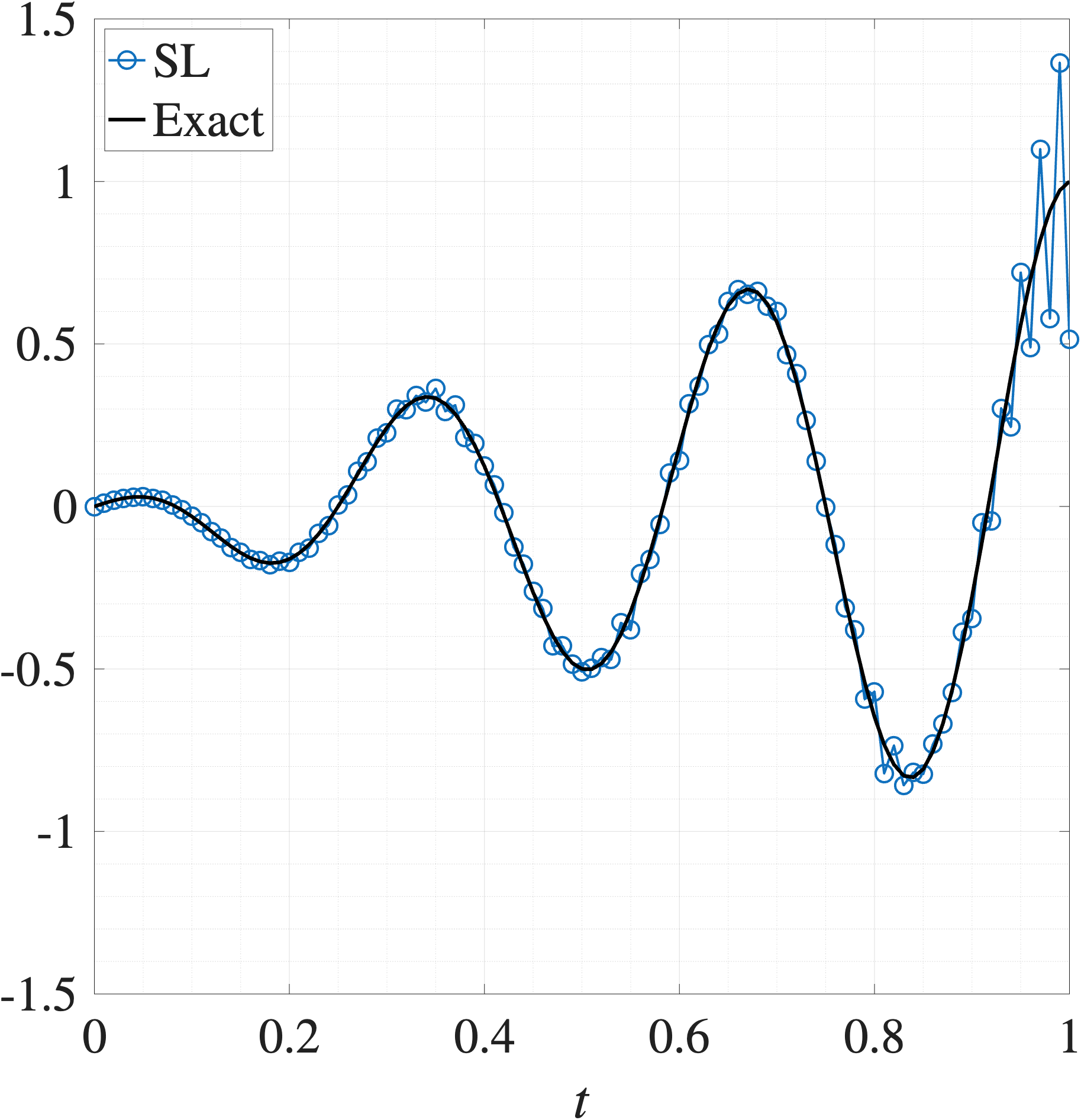}
        \caption{SL without stabilization}
    \end{subfigure}
    \hfill
    \begin{subfigure}{0.32\textwidth}
        \centering
        \includegraphics[width=0.86\textwidth]{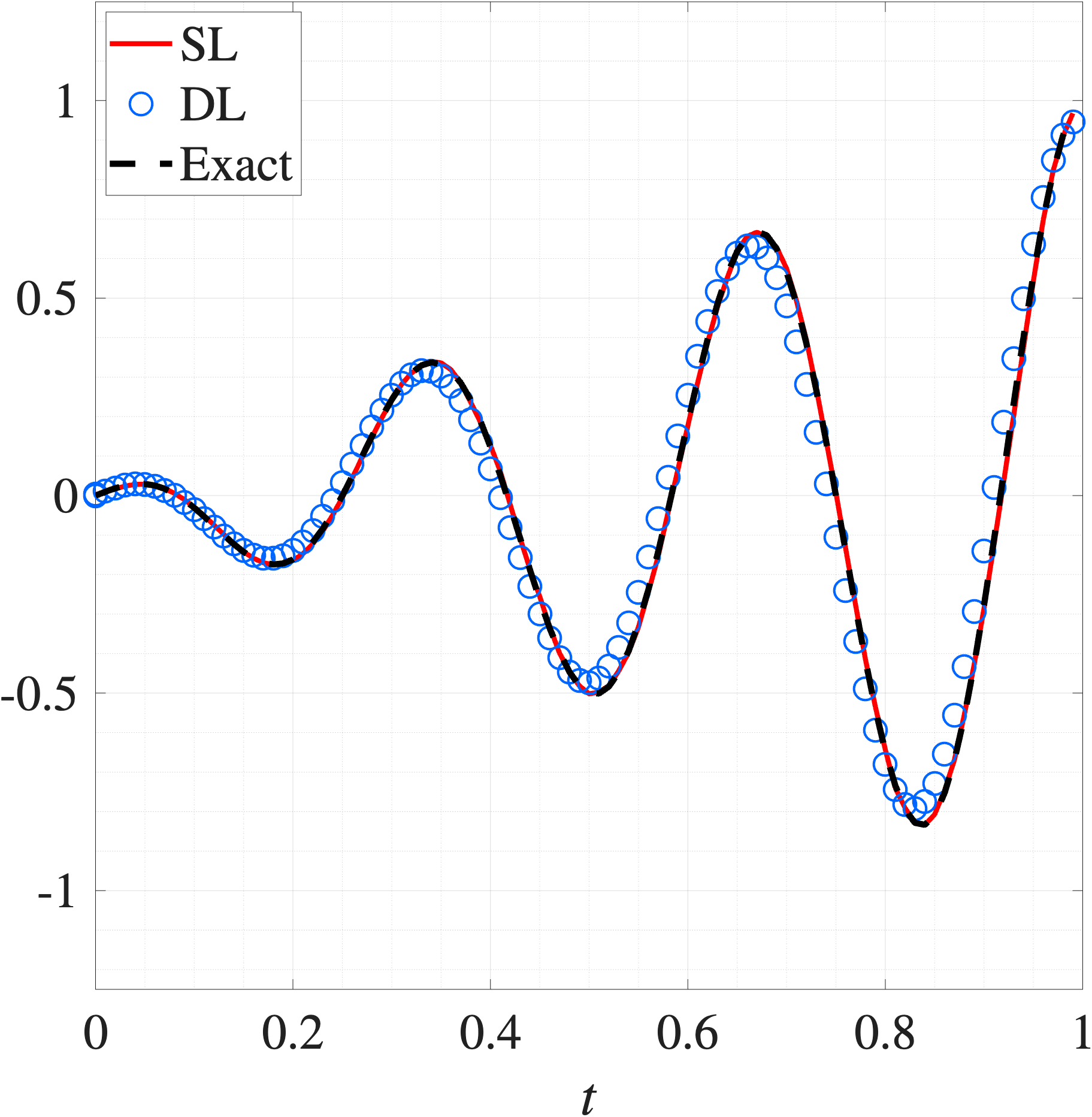}
        \caption{Stabilized SL and DL solutions}
    \end{subfigure}
    \hfill
    \begin{subfigure}{0.32\textwidth}
        \centering
        \includegraphics[width=1.04\textwidth]{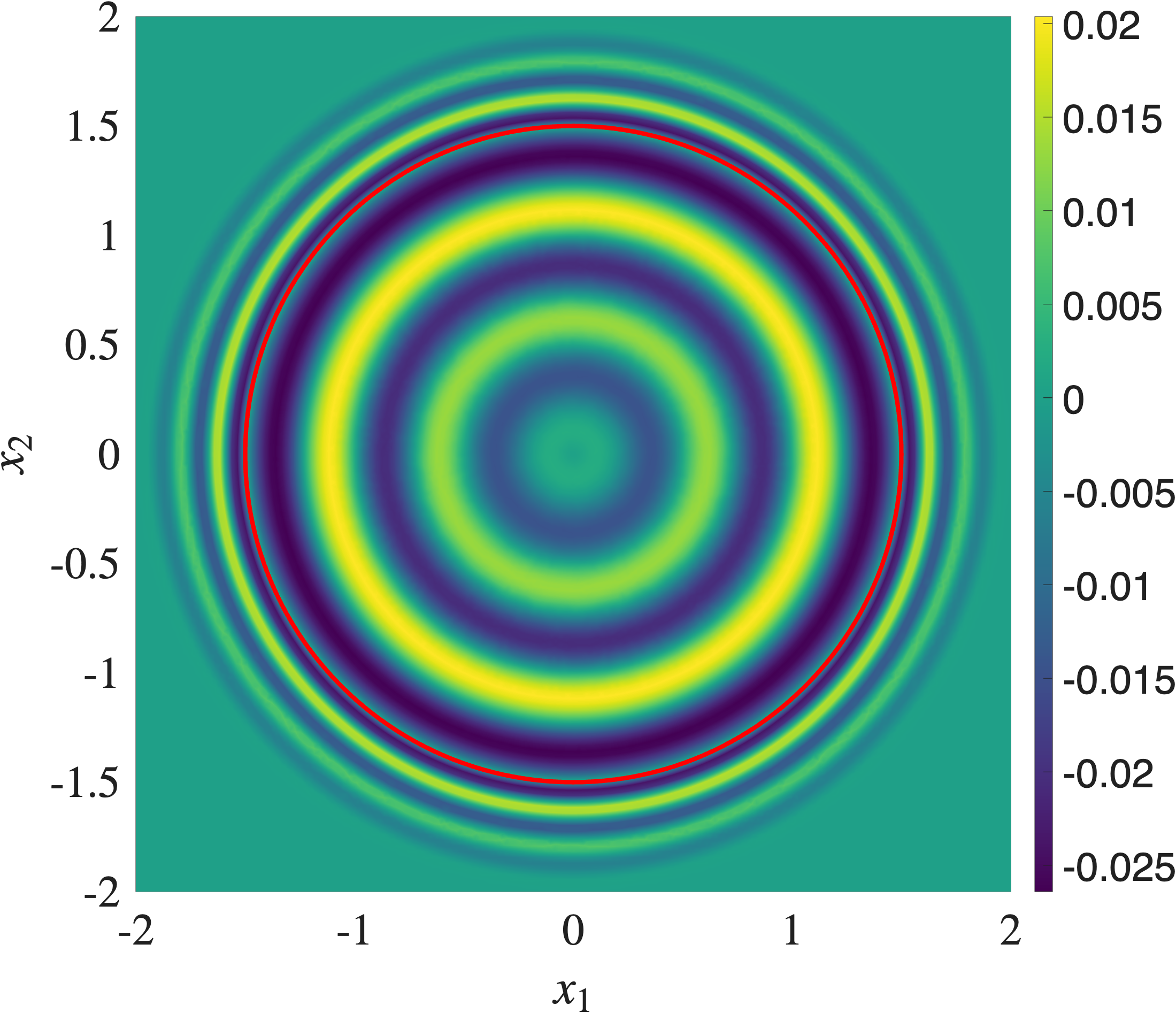}
        \caption{Numerical SL potential}
    \end{subfigure}
    \caption{%
    Expanding circle $R(t)=1+Ut$ test with $U=0.5$.
    \textbf{Left:} Time history of the spatially averaged single-layer density
    $\sigma_\Delta$ computed without stabilization, showing the onset of
    spurious oscillations and eventual instability.
    \textbf{Center:} Stabilized boundary densities $\sigma_\Delta$ and
    $\mu_\Delta$ obtained using the Rynne time-averaging procedure,
    which remain stable and in good agreement with the exact solution.
    \textbf{Right:} Numerical single-layer potential $\phi_\Delta(x,t)$ at the
    final time corresponding to the stabilized single-layer solution.
    }
    \label{fig:test2}
\end{figure}

The results of the two tests presented in this section show that the
single-layer formulation yields stable and more accurate results than the
double-layer solver. 
In the present implementation, only the principal integration-by-parts 
contribution \eqref{Dop-ibp-principal} is discretized in the double-layer solver, while the 
temporal-derivative correction term \eqref{Dop-ibp-tcorr} is neglected. 
This is consistent with the assumption that geometric quantities 
vary slowly with respect to the emission time $\tau$ in the subsonic regime.
As the interface speed increases, however, this assumption 
breaks down, and the correction term becomes non-negligible.
As such, we focus primarily on the single-layer formulation in the
numerical experiments that follow.

\subsection{Numerical example: Doppler scattering by a rigidly moving sphere}

We consider an example in which the sound speed is constant,
$c(x,t)\equiv 1$, and the scatterer is the unit sphere. 
The incident field is taken to be the Gaussian plane-wave packet
\begin{equation}\label{doppler-inc}
  \phiinc(x_1,x_2,x_3,t)
  =
  \exp\!\left(
    -\left(\tfrac{x_1-t+x_0}{a}\right)^2
  \right),
  \qquad
  x_0=1.5,
  \qquad
  a=0.2.
\end{equation}
Four test cases are studied: \emph{fixed}, \emph{left}, \emph{right}, and \emph{rising}.
In the \emph{fixed} case, the sphere is stationary. In the \emph{right}
(\emph{left}) case, the sphere translates in the positive (negative)
$x_1$-direction with speed $U = 0.5$. In the \emph{rising} case, the sphere
translates vertically upward in the positive $x_3$-direction with the same
speed $U = 0.5$. Since $c(x,t)\equiv 1$, the maximum Mach number is
$\max_{\beta,t} |V_\nu(\beta,t)|$, which in these rigid-translation examples is
equal to $U = 0.5$.

We simulate the four test cases using the single-layer solver on a triangulated
sphere with $713$ vertices and $1422$ faces, corresponding to a target
surface edge length of $h=0.2$. To compare the solutions at a common stage of
the interaction, we evaluate each case at the time $\tmax$ for which the
peak of the incident packet reaches the eastmost point of the sphere. 
The final time and time step are as follows: for the \emph{fixed} and
\emph{rising} cases, $\tmax=2.5$ and $\Delta t=0.1$; for the
\emph{left}-moving case, $\tmax=5/3$ and $\Delta t=5/48$; and for the
\emph{right}-moving case, $\tmax=5.0$ and $\Delta t=0.1$.

The solutions are shown in \Cref{fig:doppler-phi}. Each panel displays the sphere as a gray surface,
the reflected wavefront as a black curve, and the scattered potential on the
plane $x_2=0$ localized near the wavefront. A direct comparison of the
reflected wavefronts is provided in
\Cref{fig:doppler-comparison-wavefront}.

\begin{figure}[ht]
    \centering

    \begin{subfigure}{0.24\textwidth}
        \centering
        \includegraphics[width=1.0\textwidth]{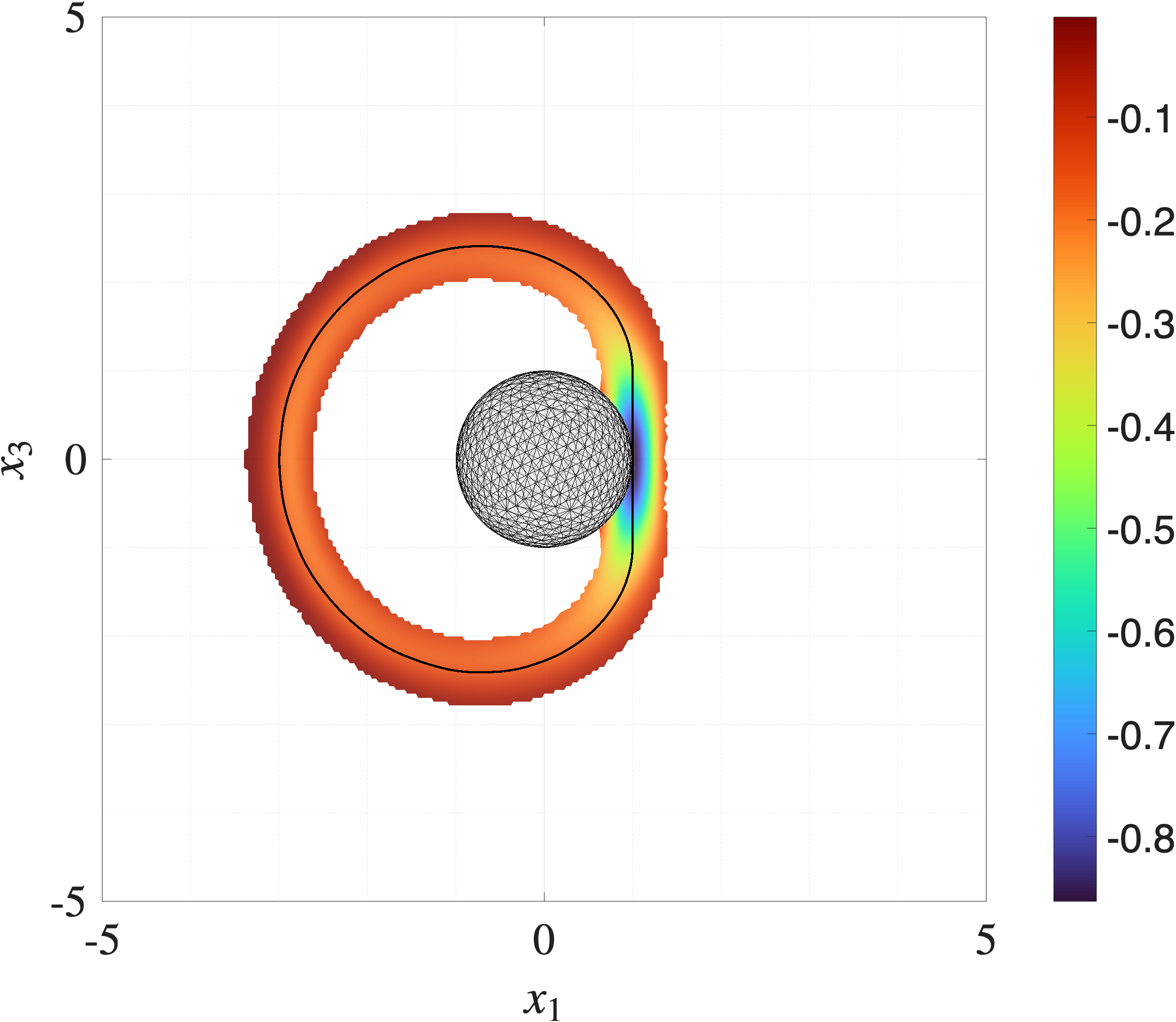}
        \caption{\emph{Fixed}: $t=2.5$.}
    \end{subfigure}
    \begin{subfigure}{0.24\textwidth}
        \centering
        \includegraphics[width=1.0\textwidth]{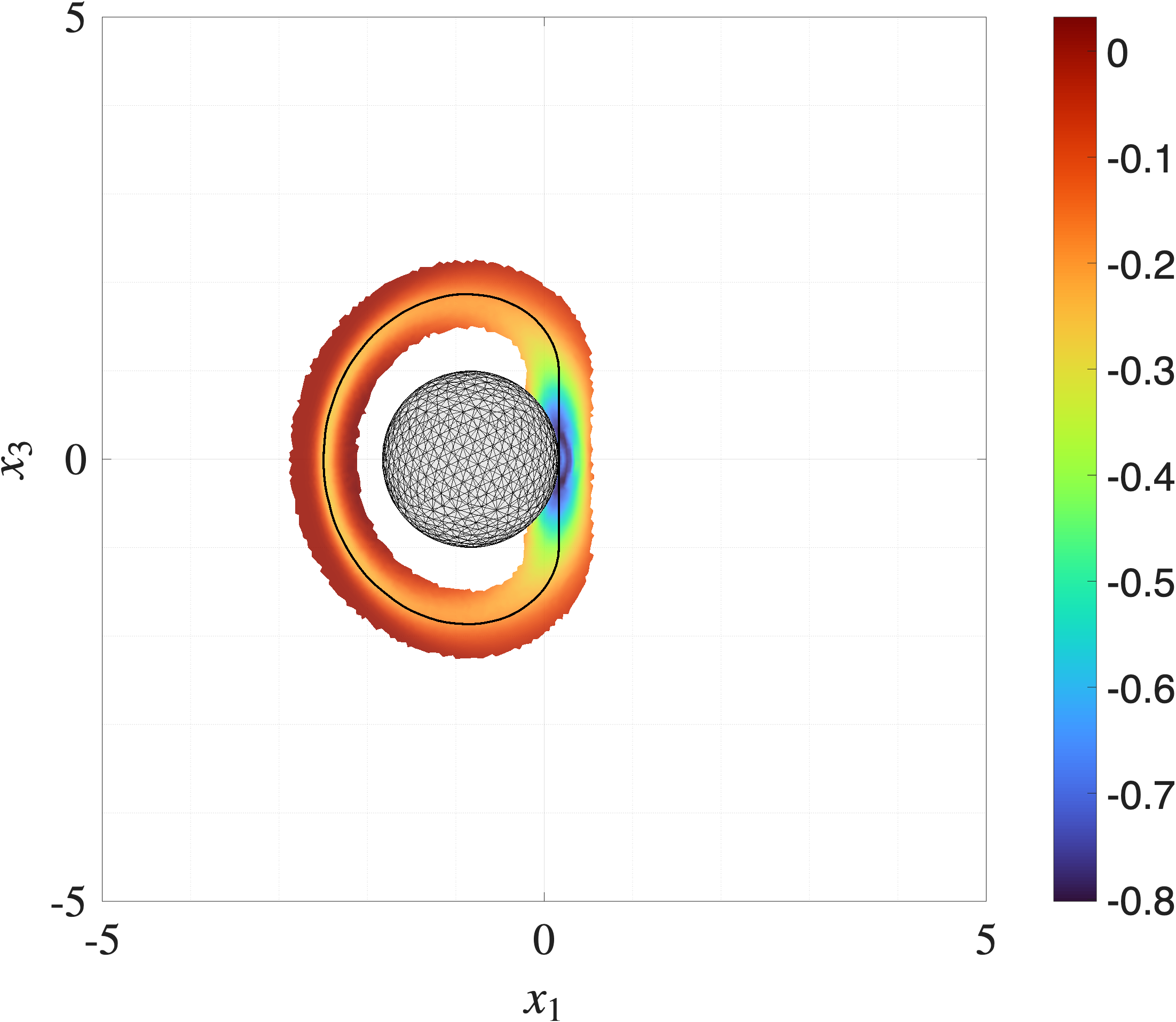}
        \caption{\emph{Left}: $t=5/3$.}
    \end{subfigure}
    \begin{subfigure}{0.24\textwidth}
        \centering
        \includegraphics[width=1.0\textwidth]{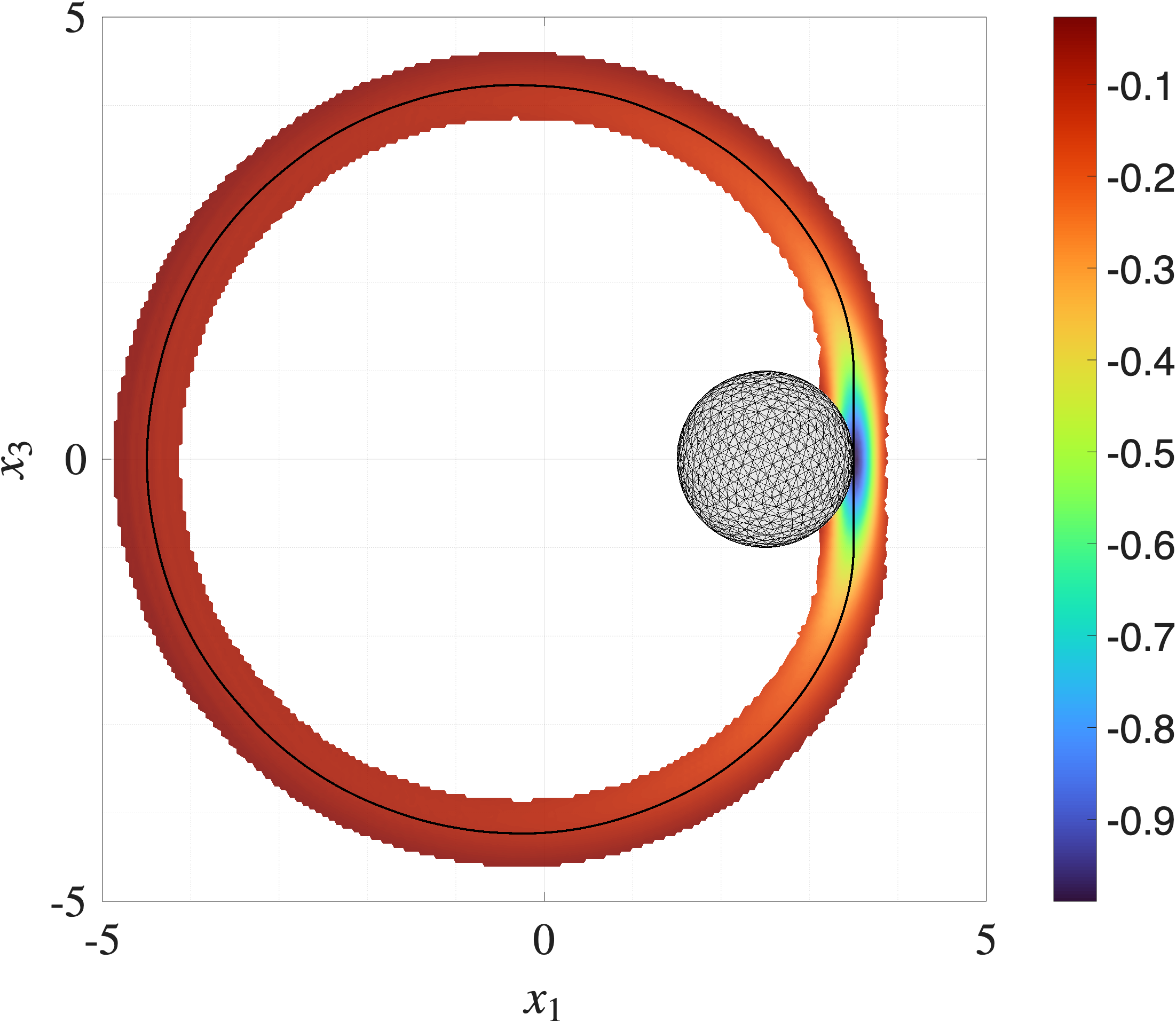}
        \caption{\emph{Right}: $t=5.0$.}
    \end{subfigure}
    \begin{subfigure}{0.24\textwidth}
        \centering
        \includegraphics[width=1.0\textwidth]{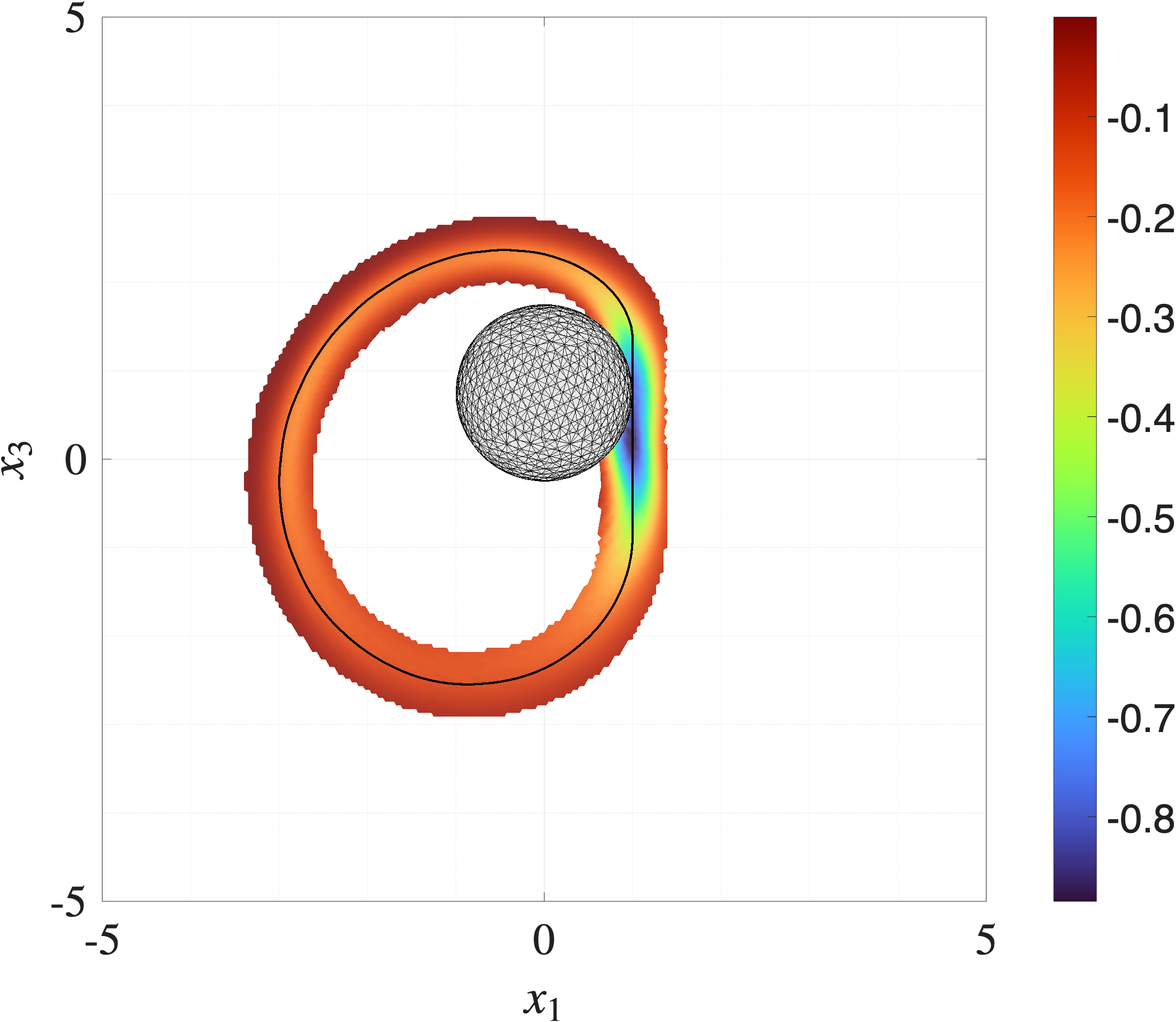}
        \caption{\emph{Rising}: $t=2.5$.}
    \end{subfigure}

    \caption{
    Numerical solutions for the four Doppler tests with Mach number $U = 0.5$.
    In each panel, the interface is shown as the gray surface and the reflected
    wavefront as the black curve. The scattered potential in the plane
    $x_2=0$, localized near the wavefront, is also displayed.
    }
    \label{fig:doppler-phi}
\end{figure}

\begin{figure}[ht]
    \centering

    \begin{subfigure}{0.45\textwidth}
        \centering
        \includegraphics[width=0.75\textwidth]{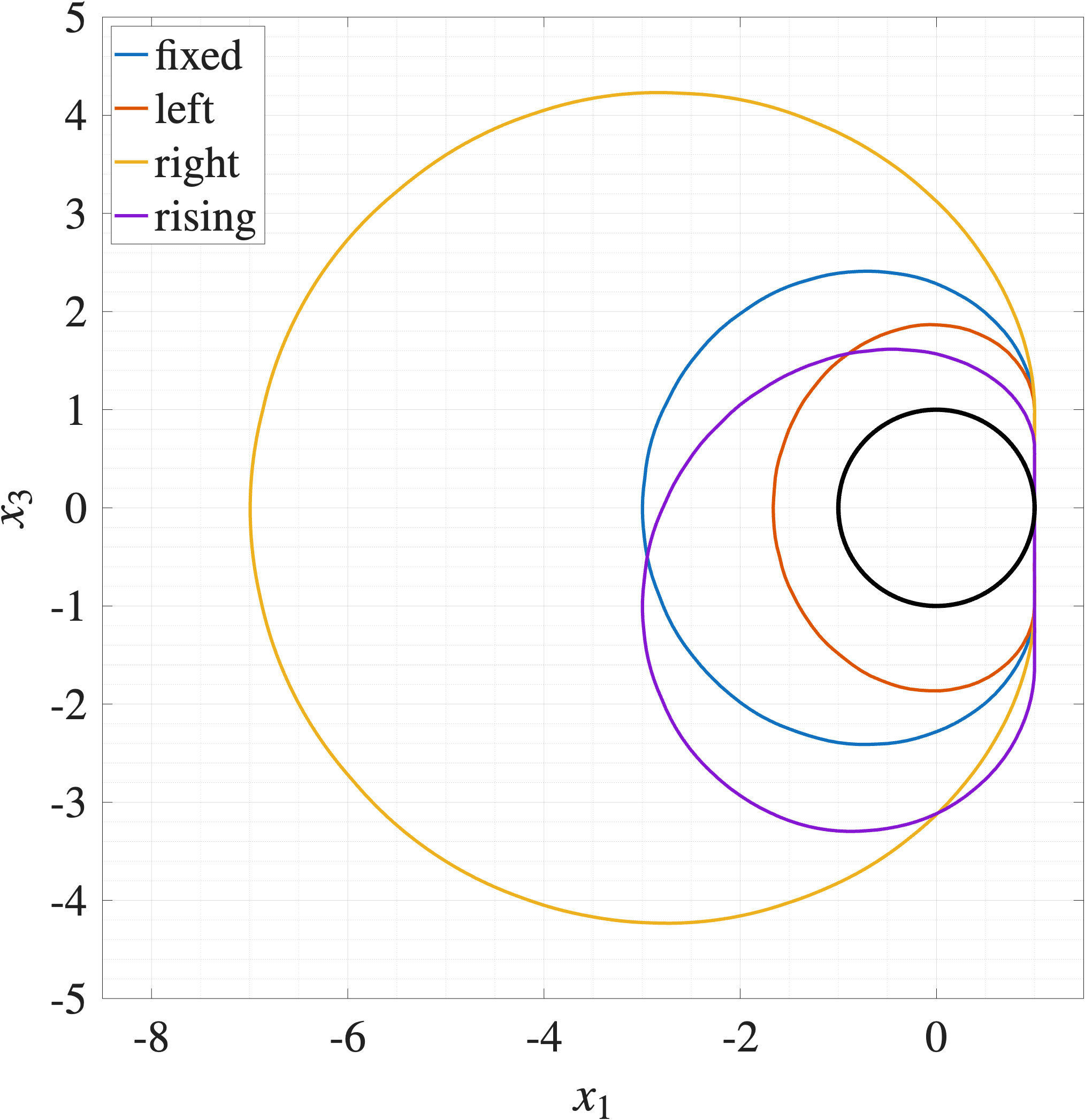}
        \caption{Reflected wavefront comparison.}
        \label{fig:doppler-comparison-wavefront}
    \end{subfigure}
    \hspace{0em}
    \begin{subfigure}{0.45\textwidth}
        \centering
        \includegraphics[width=0.75\textwidth]{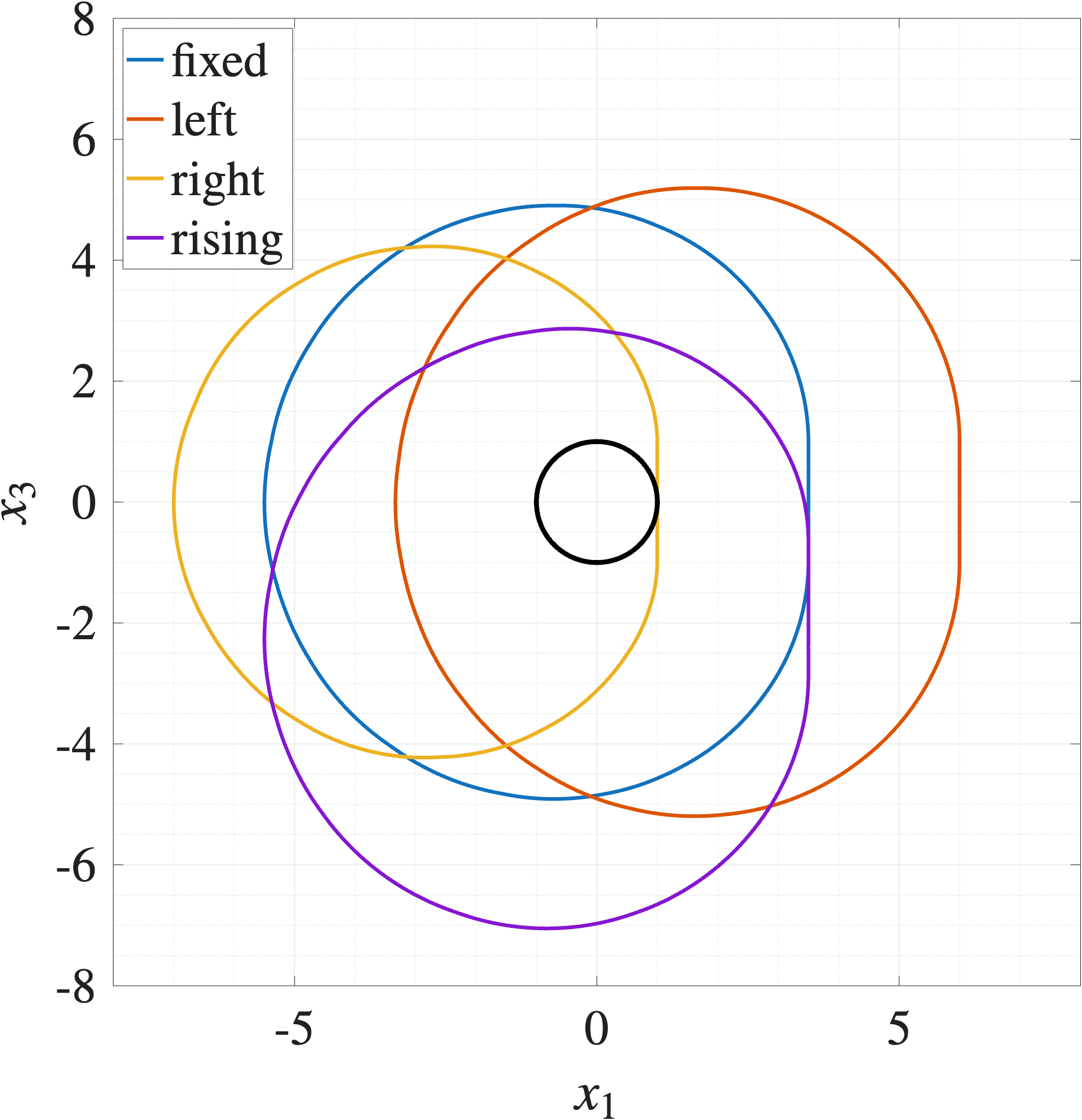}
        \caption{Reflected wavefronts at $t = 5.0$.}
        \label{fig:doppler-comparison-final}
    \end{subfigure}
    \caption{
    Comparison of the reflected wavefronts in the plane $x_2=0$ for the
    four Doppler tests at Mach number $U = 0.5$. \textbf{Left}: reflected wavefronts at the
    comparison times used in \Cref{fig:doppler-phi}. \textbf{Right}: reflected
    wavefronts at the common time $t=5.0$. In both panels, each wavefront is
    translated so that the sphere is centered at the origin.
    }
    \label{fig:doppler-comparison}
\end{figure}

These plots reveal how the geometry of the reflected wavefront depends on the
prescribed motion of the sphere.
When the sphere moves in the $x_1$-direction, the reflected wavefront is
compressed in the \emph{left} case and expanded in the \emph{right} case. 
In the \emph{rising} case, the reflected wavefront develops a
clear north--south asymmetry, with the reflected front
compressed above the sphere and expanded below it. The location of the peak
scattered potential follows this asymmetry: in the \emph{fixed}, \emph{left},
and \emph{right} cases it is centered about the equator, whereas in the
\emph{rising} case it is displaced downward.

The reflected wavefronts for the four cases are displayed at the common time
$t=5.0$ in \Cref{fig:doppler-comparison-final}.
With the sphere recentered at the origin, the geometric effect of the motion is
more clearly separated from the rigid translation of the interface. The
\emph{left} and \emph{right} cases differ primarily through a displacement in
the $x_1$-direction, while the \emph{rising} case exhibits the same
north--south asymmetry apparent in \Cref{fig:doppler-phi}.

\begin{figure}[ht]
    \centering

    \begin{subfigure}{0.45\textwidth}
        \centering
        \includegraphics[width=0.765\textwidth]{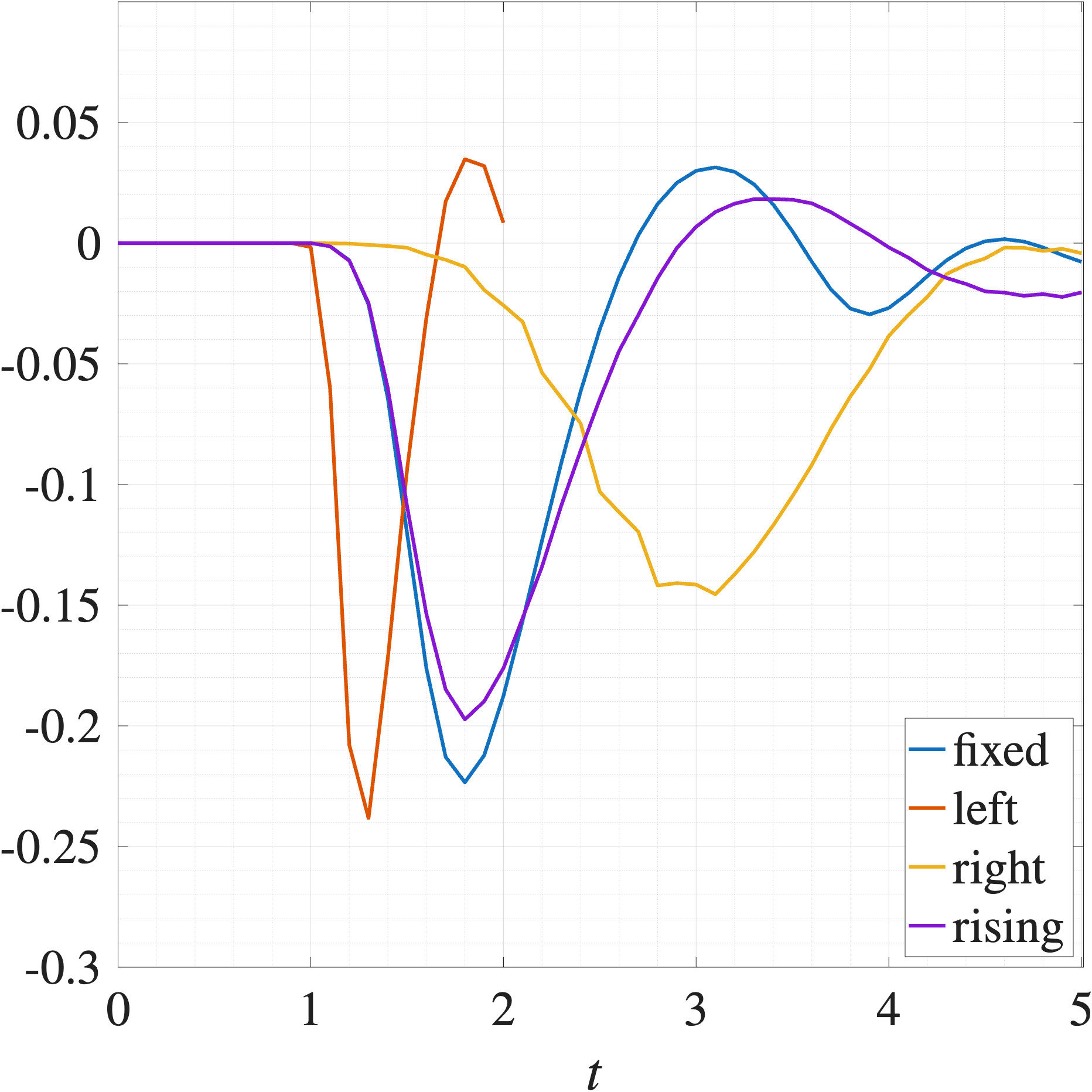}
        \caption{Fixed sensor at $(-2,0,0)$.}
        \label{fig:doppler-sensor-fixed}
    \end{subfigure}
    \begin{subfigure}{0.45\textwidth}
        \centering
        \includegraphics[width=0.75\textwidth]{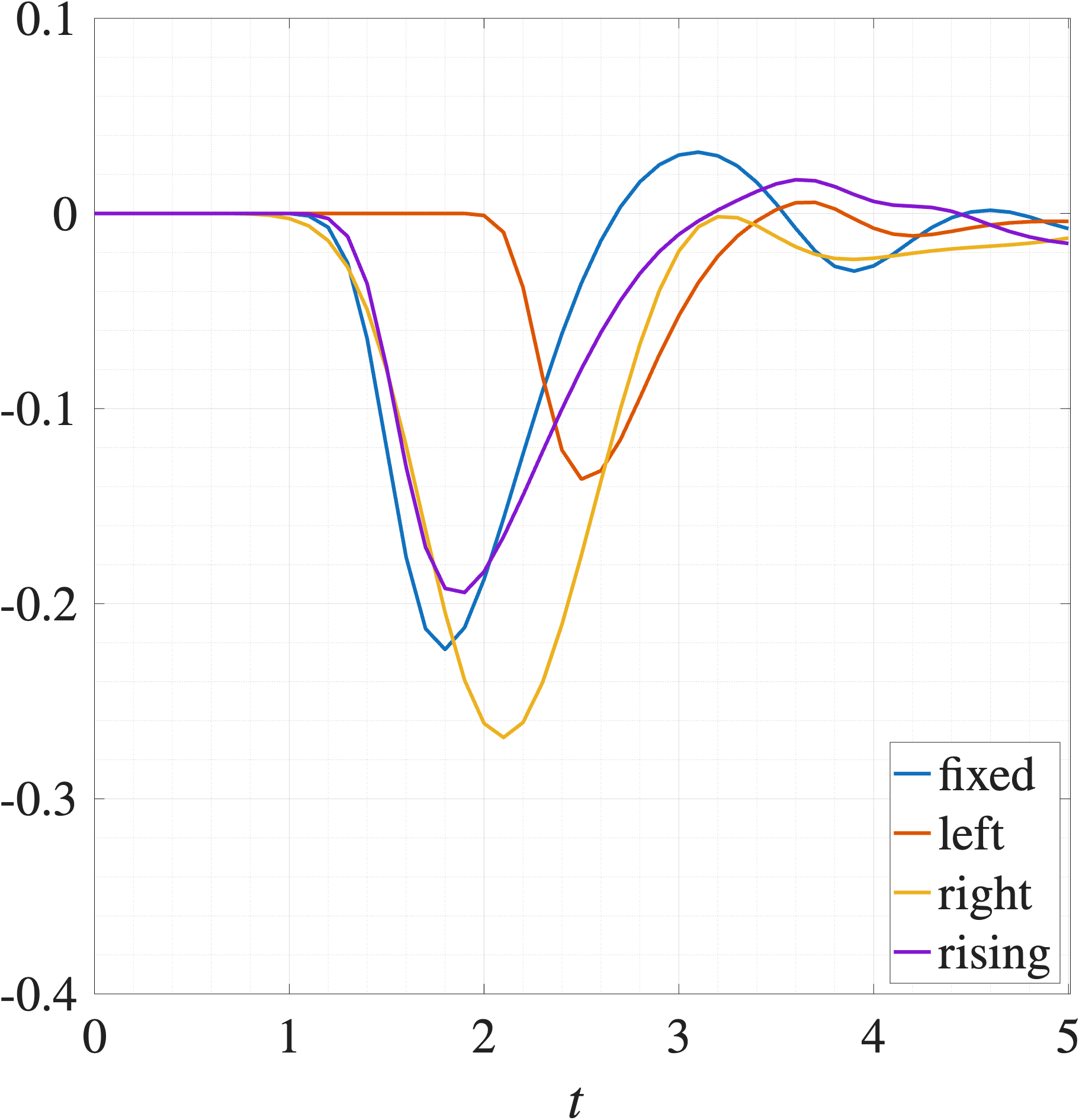}
        \caption{Co-moving sensor.}
        \label{fig:doppler-sensor-moving}
    \end{subfigure}

    \caption{
    Scattered-potential sensor histories for the four Doppler tests at Mach number $U = 0.5$.
    \textbf{Left}: history at the fixed sensor $(-2,0,0)$. \textbf{Right}: history
    at a co-moving sensor, initialized at $(-2,0,0)$ and transported with the sphere.
    }
    \label{fig:doppler-sensor}
\end{figure}

To further quantify these differences, we compare scattered-potential histories
recorded at two sensors. The first is the fixed sensor at $(-2,0,0)$; the corresponding 
sensor history is displayed in \Cref{fig:doppler-sensor-fixed}.  
As expected, the dominant scattered pulse arrives earliest in the
\emph{left} case and latest in the \emph{right} case, with the \emph{fixed}
and \emph{rising} cases lying in between. Relative to the fixed sensor, the sphere moves toward the sensor in the
\emph{left} case and away from it in the \emph{right} case, producing the
expected Doppler-type compression and stretching of the reflected waveform.

The second is a co-moving sensor initialized  at the same point $(-2,0,0)$ and
thereafter transported along with the rigid motion of the sphere.  The associated 
sensor history is displayed in \Cref{fig:doppler-sensor-moving}. 
The four peak arrival times are more closely grouped: because the sensor 
moves with the sphere, the main travel-time shift caused by
the rigid motion is removed, and what remains is the change in the reflected
pulse itself. Since the incident packet sweeps past the sphere more
slowly in the \emph{right} case and more rapidly in the \emph{left} case,  the reflected pulse 
is broader and larger in the former, and more
compressed and weaker in the latter.

\subsubsection{Mach number parameter study}

To examine the dependence on interface speed systematically, we
conduct a Mach number parameter study for the \emph{rising} configuration.
The physical setup, spatial discretization, and incident field are taken to be
the same as in the preceding Doppler tests; only the upward 
translation speed (Mach number) $U \in \{0.0, 0.1, 0.3, 0.5, 0.7, 0.9 \}$ is varied.  
The single-layer solutions computed with a time-step $\Delta t = 1/30$ are displayed 
at the final time $\tmax = 2.5$ in \Cref{fig:mach-study}. Each panel shows the 
reflected wavefront and the slice of the local scattered potential in the $x_2=0$ plane. 
As the Mach number increases, both the reflected wavefront and the scattered
potential develop an increasingly pronounced north--south asymmetry.
The reflected front is compressed above the sphere and expanded
below it, while the peak of the scattered potential shifts downward and becomes
more vertically elongated. Across the range of Mach numbers tested, the computed 
solutions remain stable, including in the largest case $U=0.9$.

\begin{figure}[ht]
    \centering

    \begin{subfigure}{0.32\textwidth}
        \centering
        \includegraphics[width=\textwidth]{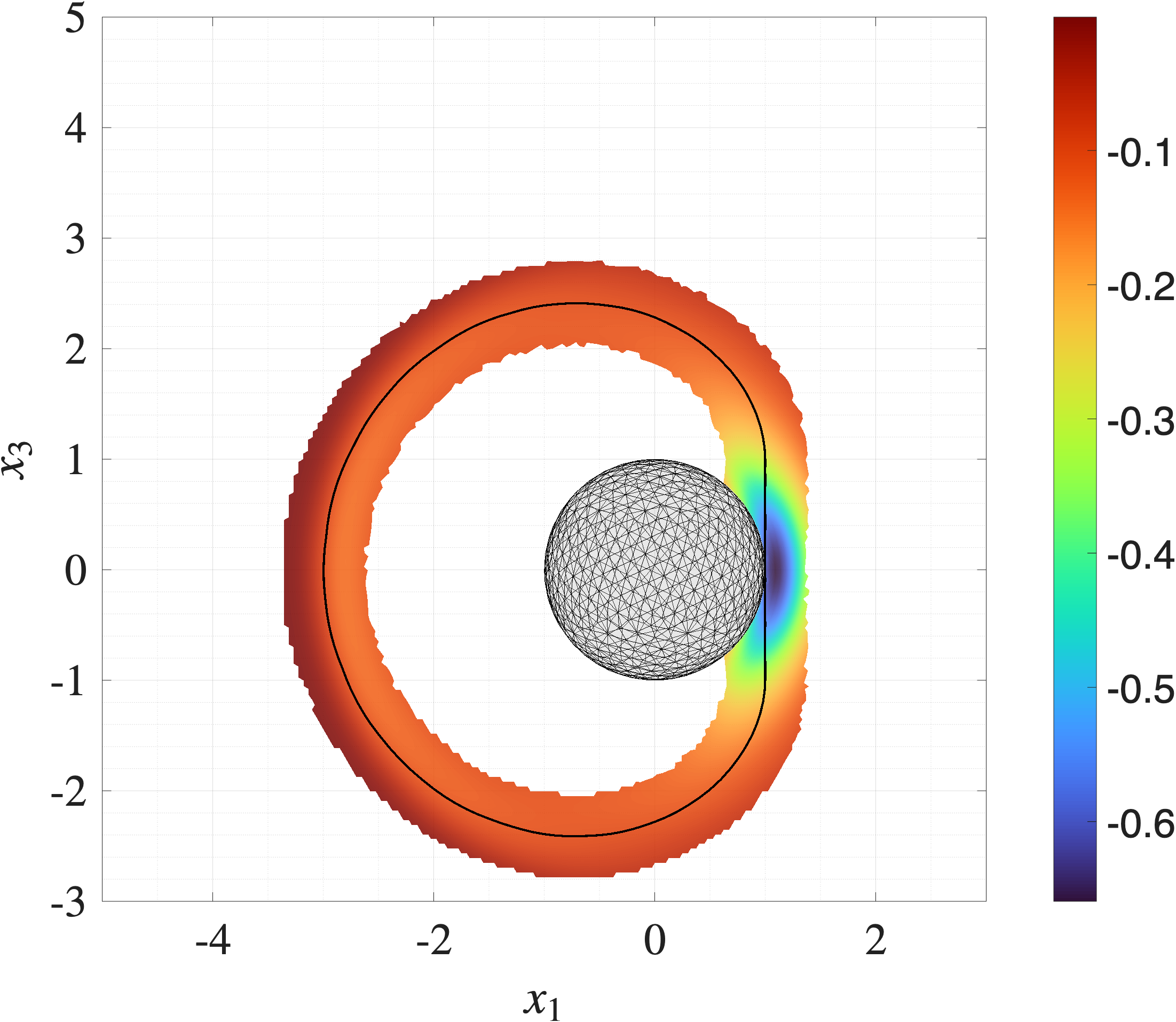}
        \caption{$U=0.0$}
    \end{subfigure}
    \hfill
    \begin{subfigure}{0.32\textwidth}
        \centering
        \includegraphics[width=\textwidth]{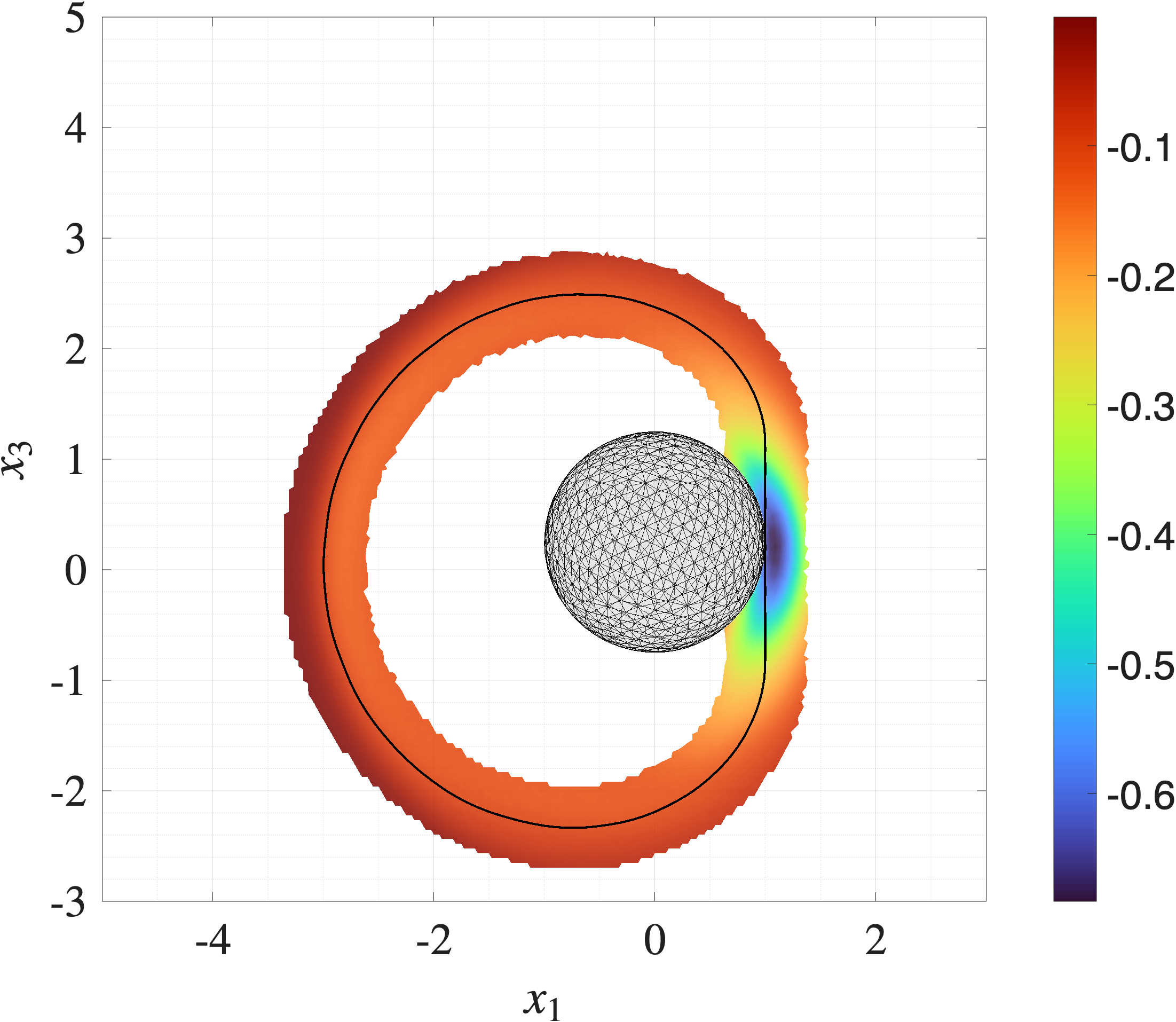}
        \caption{$U=0.1$}
    \end{subfigure}
    \hfill
    \begin{subfigure}{0.32\textwidth}
        \centering
        \includegraphics[width=\textwidth]{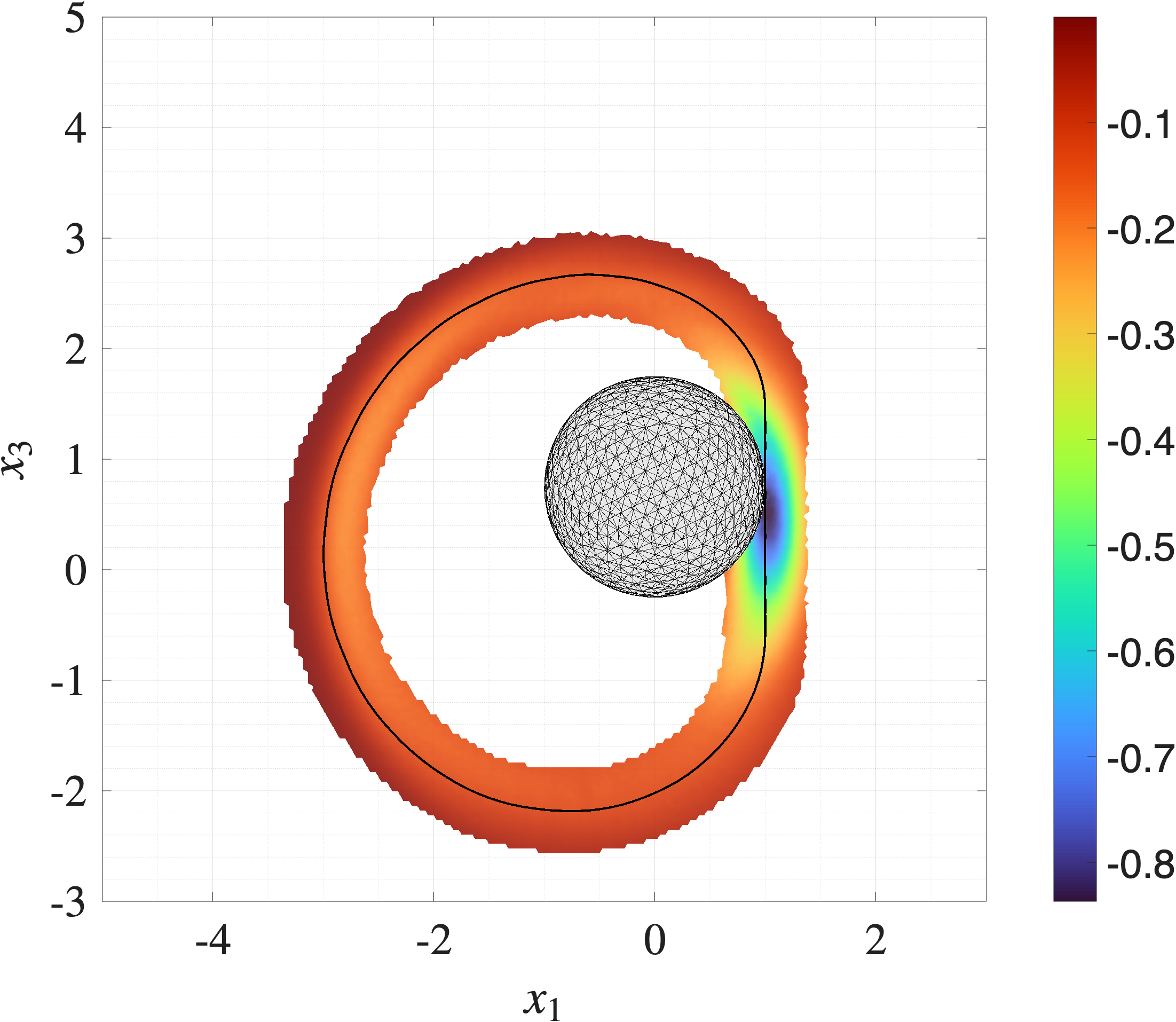}
        \caption{$U=0.3$}
    \end{subfigure}

    \vspace{0.5em}

    \begin{subfigure}{0.32\textwidth}
        \centering
        \includegraphics[width=\textwidth]{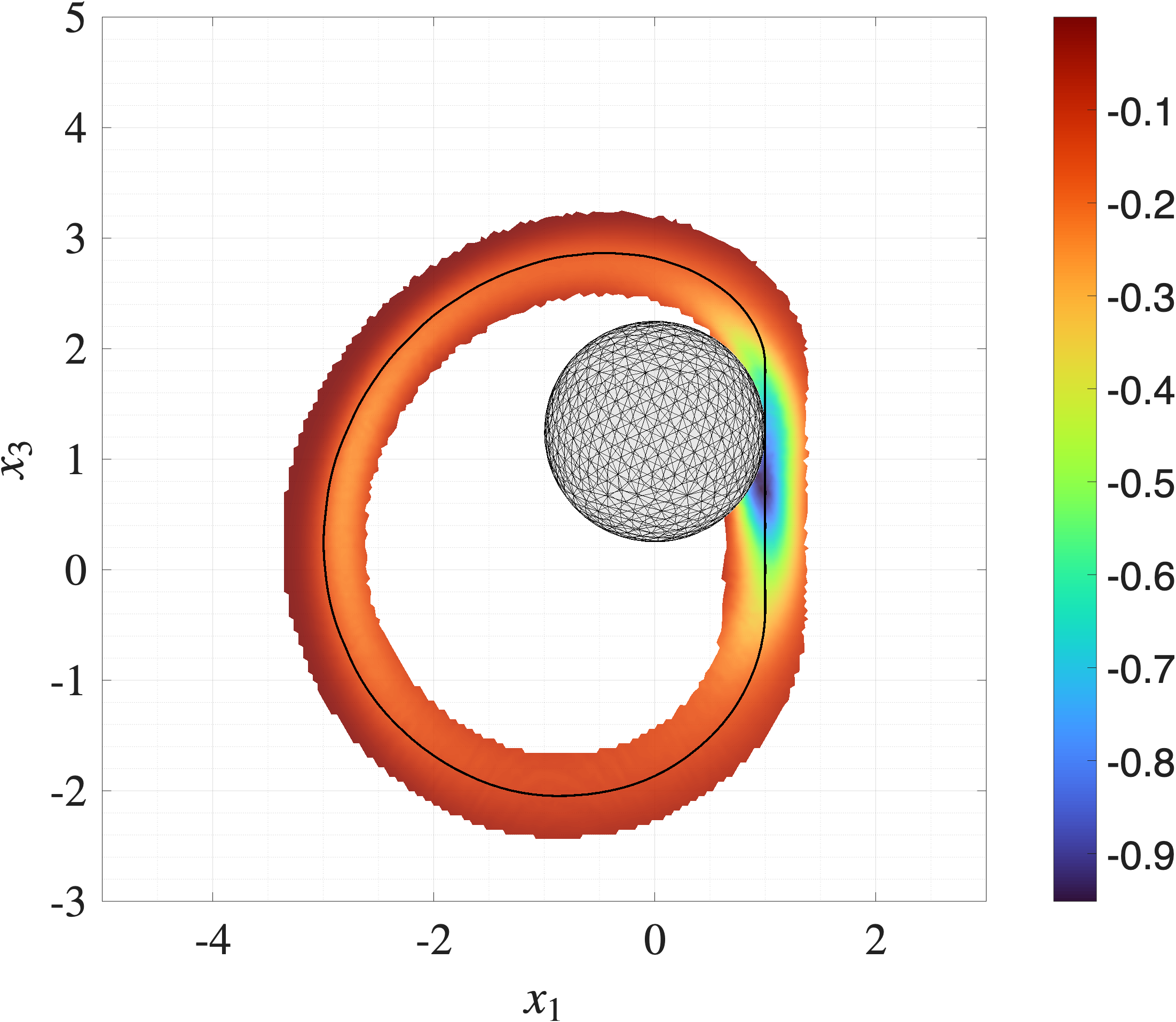}
        \caption{$U=0.5$}
    \end{subfigure}
    \hfill
    \begin{subfigure}{0.32\textwidth}
        \centering
        \includegraphics[width=\textwidth]{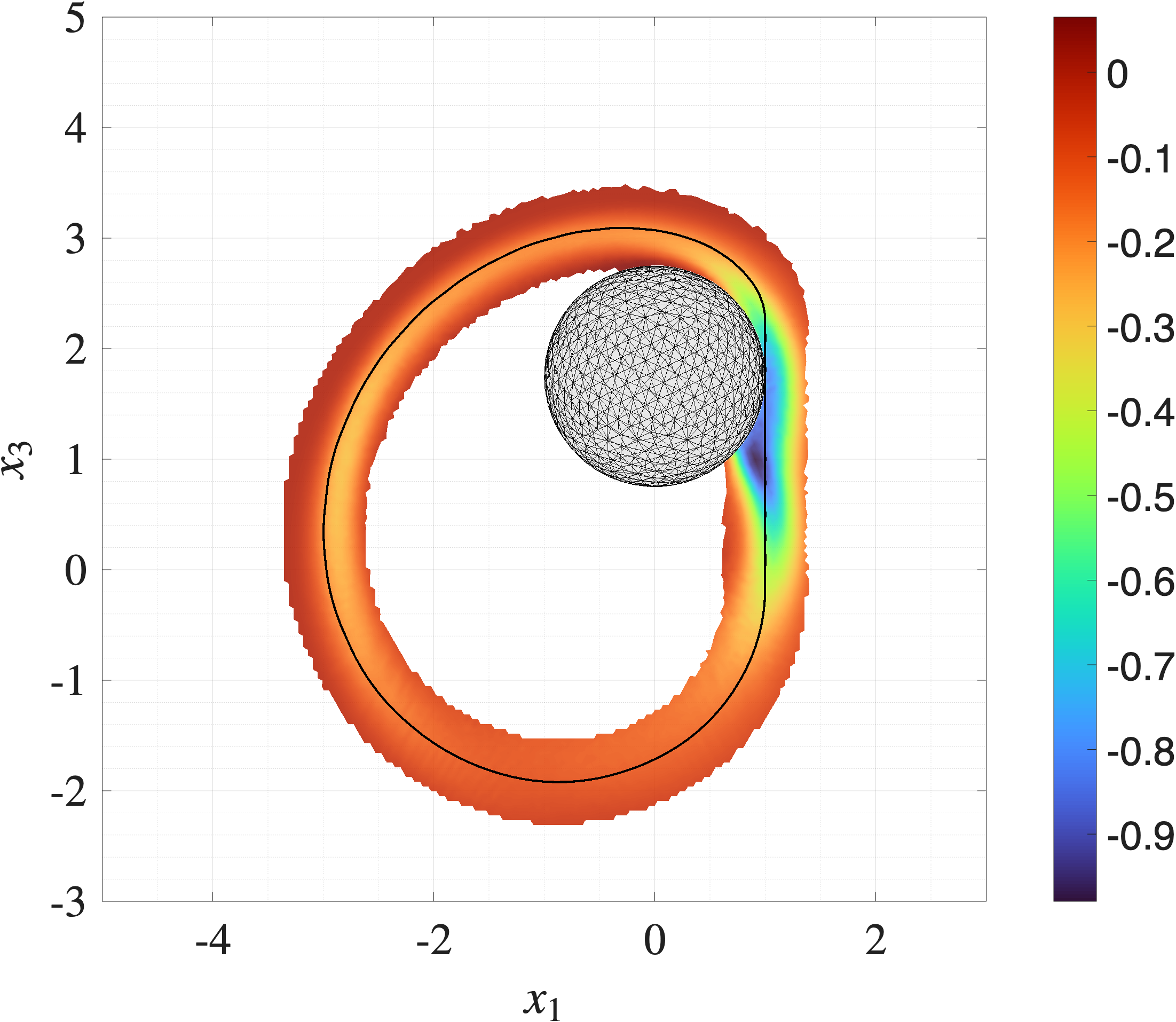}
        \caption{$U=0.7$}
    \end{subfigure}
    \hfill
    \begin{subfigure}{0.32\textwidth}
        \centering
        \includegraphics[width=\textwidth]{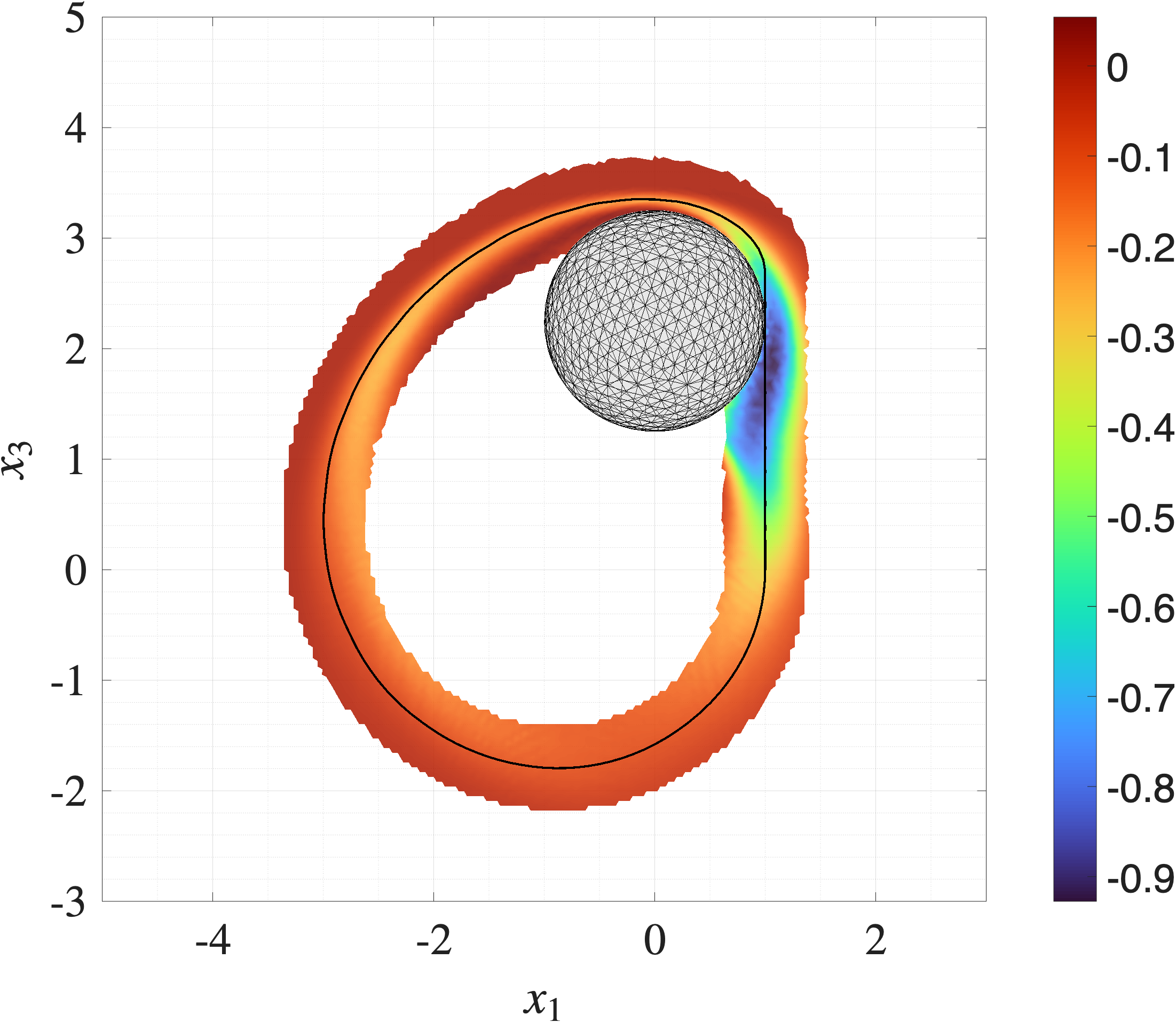}
        \caption{$U=0.9$}
    \end{subfigure}

    \caption{
    Mach number $U$ parameter study for the \emph{rising} configuration.
    In each panel, the reflected wavefront and the slice of the local scattered
    potential in the plane $x_2=0$ are displayed at the final time $\tmax = 2.5$.
    }
    \label{fig:mach-study}
\end{figure}

\subsection{Numerical example: scattering by a smooth rotating turbine}

We consider a two-dimensional example in which the sound speed is constant,
$c(x,t)\equiv 1$, and the scatterer is a smooth turbine-shaped obstacle centered
at the origin. Define the reference interface
\[
  z^{\mathsf{ref}}(\alpha)
  =
   (
    r(\alpha)\cos\alpha,\;
    r(\alpha)\sin\alpha
   ),
  \qquad
  \alpha\in[0,2\pi),
\]
with radial profile
\begin{equation}\label{fan-shape}
  r(\alpha)
  =
  \tfrac{
    1 + 0.3 \exp ( 3\sin(3\alpha) - 1  )
  }{
    1 + 0.3 e^2
  }.
\end{equation}
The time-dependent interface is then obtained by rigid counterclockwise
rotation of this reference curve:
\begin{equation}
  z(\alpha,t)=R(\omega t)\,z^{\mathsf{ref}}(\alpha), 
  \qquad \text{with} \quad 
  R(\vartheta)
  =
  \begin{pmatrix}
    \cos\vartheta & -\sin\vartheta\\
    \sin\vartheta & \cos\vartheta
  \end{pmatrix}.
\end{equation}
The angular speed $\omega$ is chosen so that the maximum normal Mach number is
equal to a prescribed value $U$. Since $c(x,t)\equiv 1$, we set
\begin{equation}\label{fan-omega}
  \omega = \tfrac{U}{S_{\max}},
  \qquad
  S_{\max}
  \coloneqq
  \max_{\alpha\in[0,2\pi]}
  \left|
     (-z^{\mathsf{ref}}_2(\alpha),\,z^{\mathsf{ref}}_1(\alpha) )
    \cdot
    n^{\mathsf{ref}}(\alpha)
  \right|,
\end{equation}
with $n^{\mathsf{ref}}$ the outward unit normal to $z^{\mathsf{ref}}$.
The prescribed incident field is a train of Gaussian plane-wave packets,
\begin{equation}\label{fan-inc}
  \phiinc(x_1,x_2,t)
  =
  \sum_{m=0}^{\infty}
  \exp\!\left(
    -\left(\tfrac{x_1 - t + x_0 + \frac12 m}{a}\right)^2
  \right),
  \qquad
  x_0=1.5,
  \qquad
  a=0.1. 
\end{equation}

We study two test cases: \emph{fixed} with $U=0.0$ and \emph{rotating} with
$U=0.5$. In both cases, we use the single-layer solver with a uniform
$100$-cell discretization of the interface. The solutions are computed up to
time $\tmax=4$ with time step $\Delta t=0.04$, and displayed at a selected 
snapshot of times in \Cref{fig:fan}.  
In the \emph{fixed} case, successive reflected wavefronts develop pronounced indentations 
near the lower blade, and the scattered potential remains concentrated in localized regions 
adjacent to that part of the obstacle. In the \emph{rotating} case, both the indentations in the 
reflected fronts and the localized scattered structures are carried upward and around the turbine as the blades turn.
The repeated interaction of the incoming wave train with the moving interface generates additional fine-scale structure 
in the scattered field, producing a visibly more intricate pattern than in the fixed configuration.

\begin{figure}[ht]
    \centering

    \begin{subfigure}{0.32\textwidth}
        \centering
        \includegraphics[width=\textwidth]{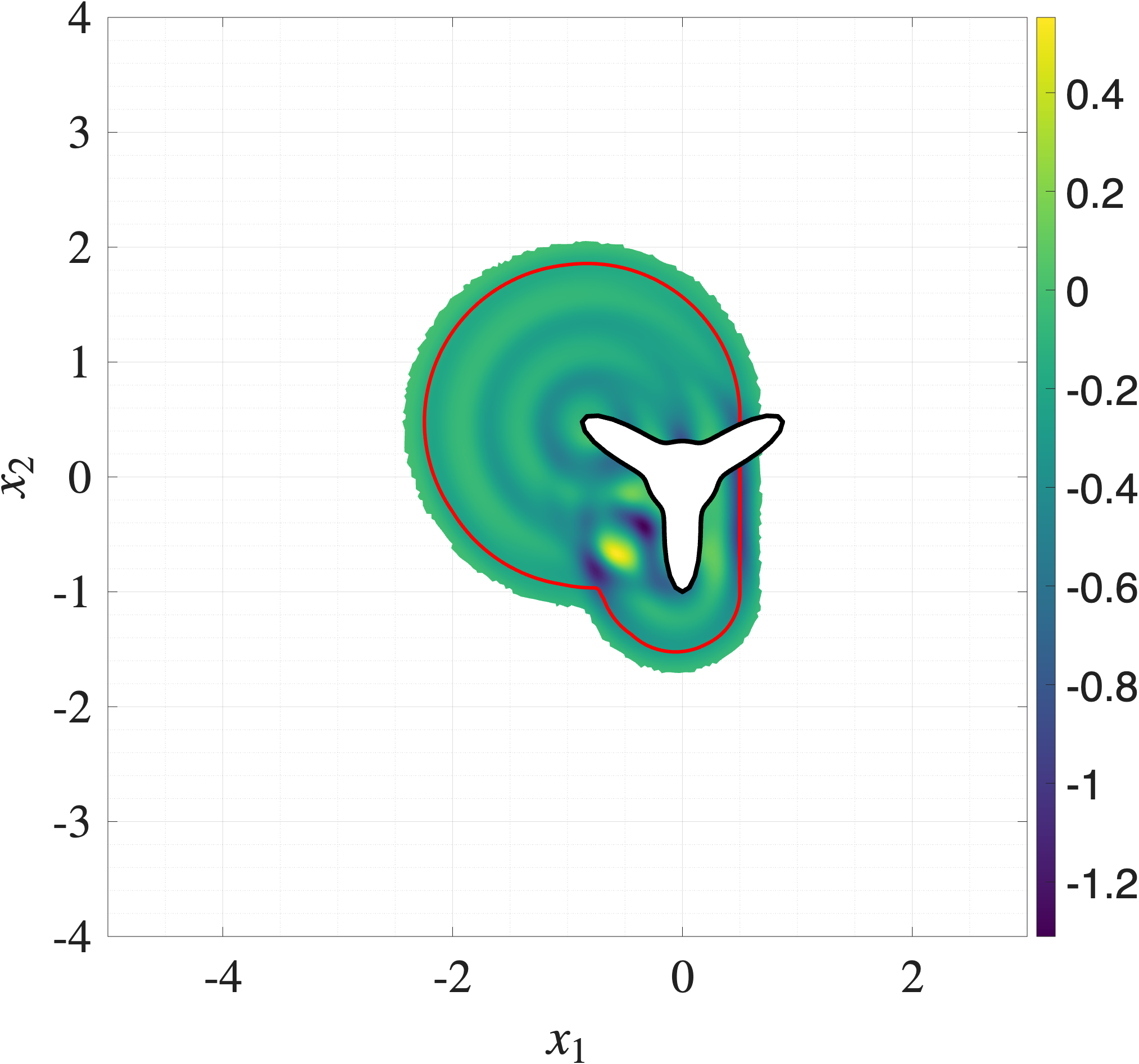}
        \caption{\emph{Fixed}: $t=2$.}
    \end{subfigure}
    \hfill
    \begin{subfigure}{0.32\textwidth}
        \centering
        \includegraphics[width=\textwidth]{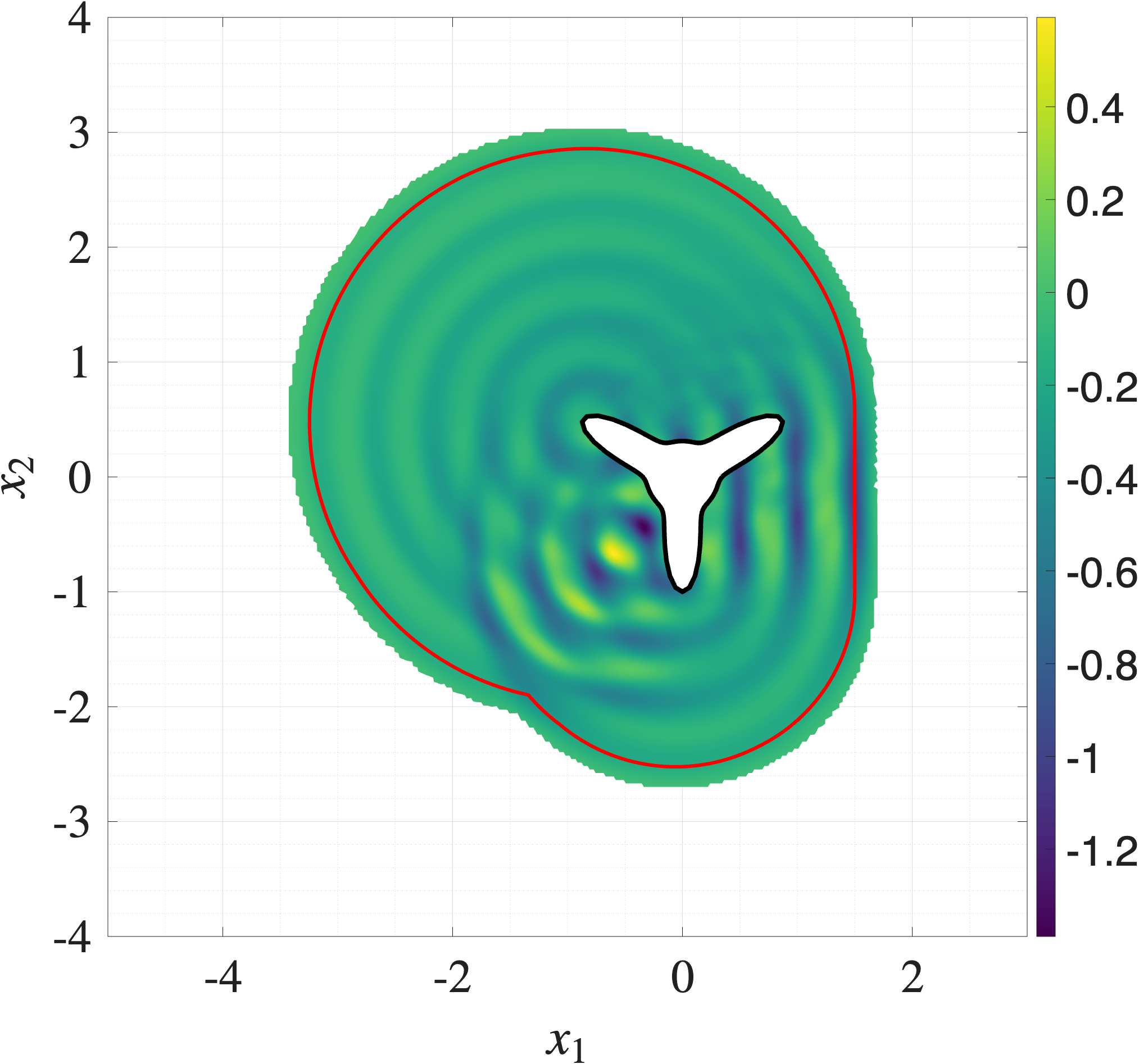}
        \caption{\emph{Fixed}: $t=3$.}
    \end{subfigure}
    \hfill
    \begin{subfigure}{0.32\textwidth}
        \centering
        \includegraphics[width=\textwidth]{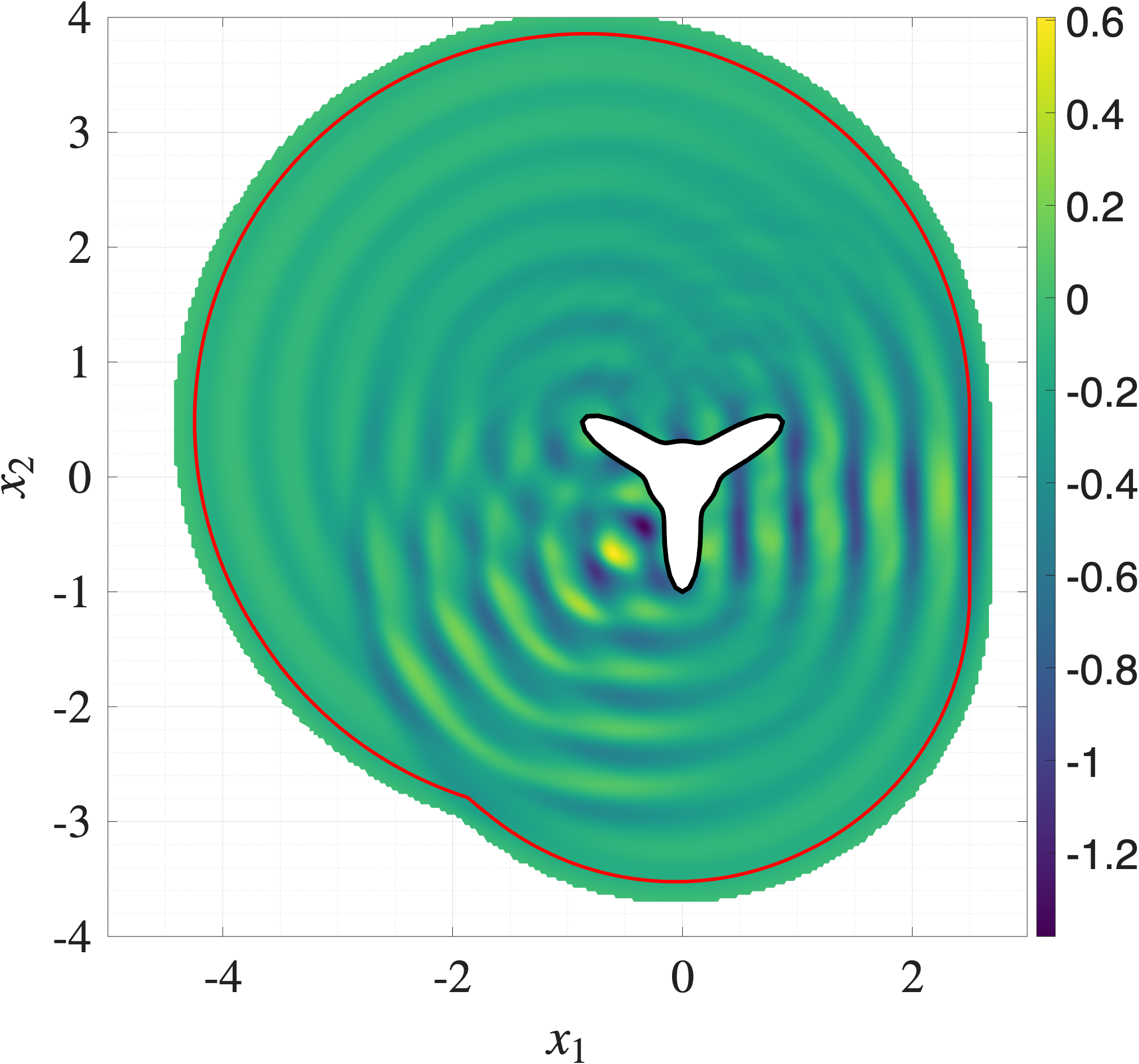}
        \caption{\emph{Fixed}: $t=4$.}
    \end{subfigure}

    \vspace{0.5em}

    \begin{subfigure}{0.32\textwidth}
        \centering
        \includegraphics[width=\textwidth]{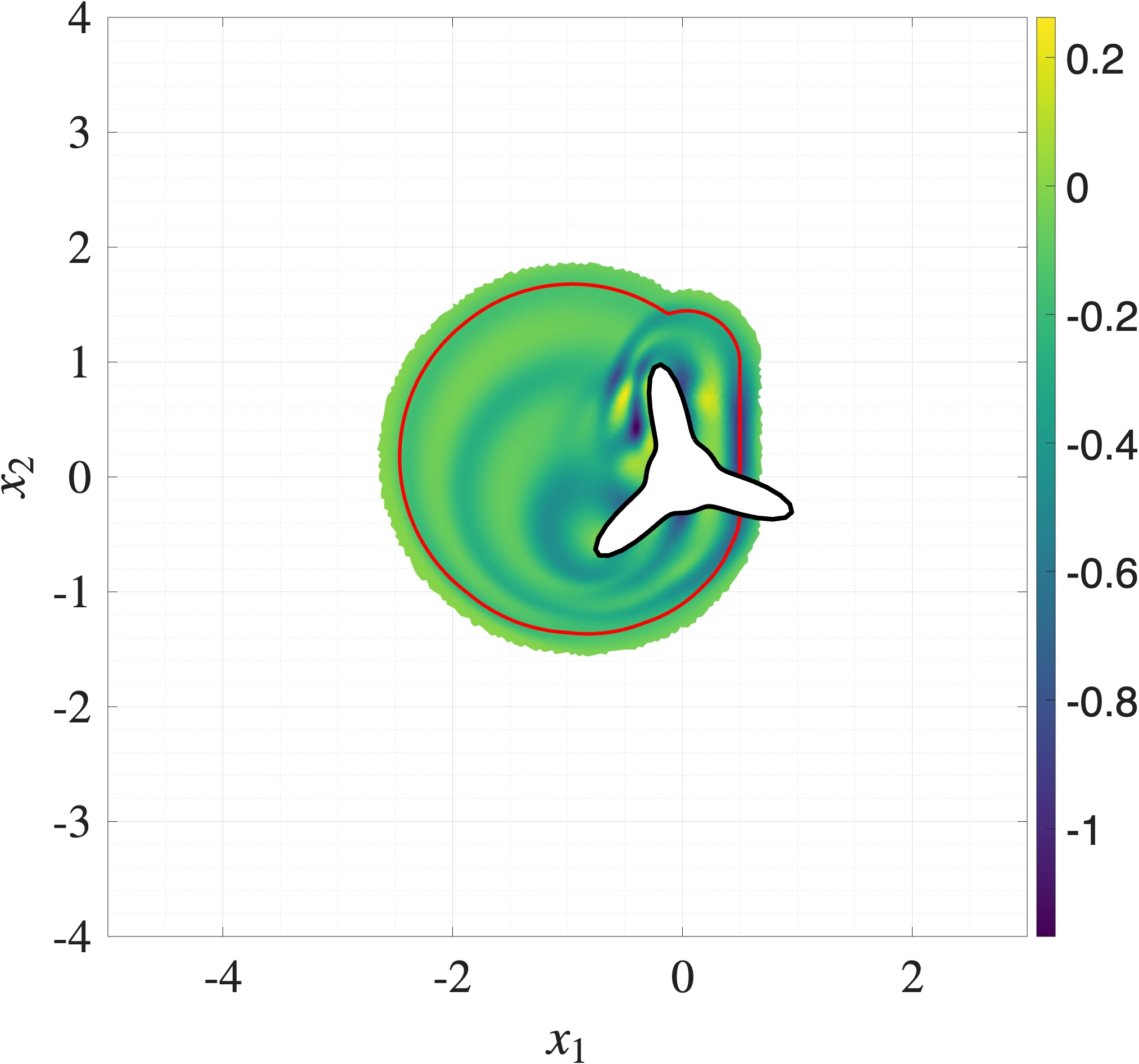}
        \caption{\emph{Rotating}: $t=2$.}
    \end{subfigure}
    \hfill
    \begin{subfigure}{0.32\textwidth}
        \centering
        \includegraphics[width=\textwidth]{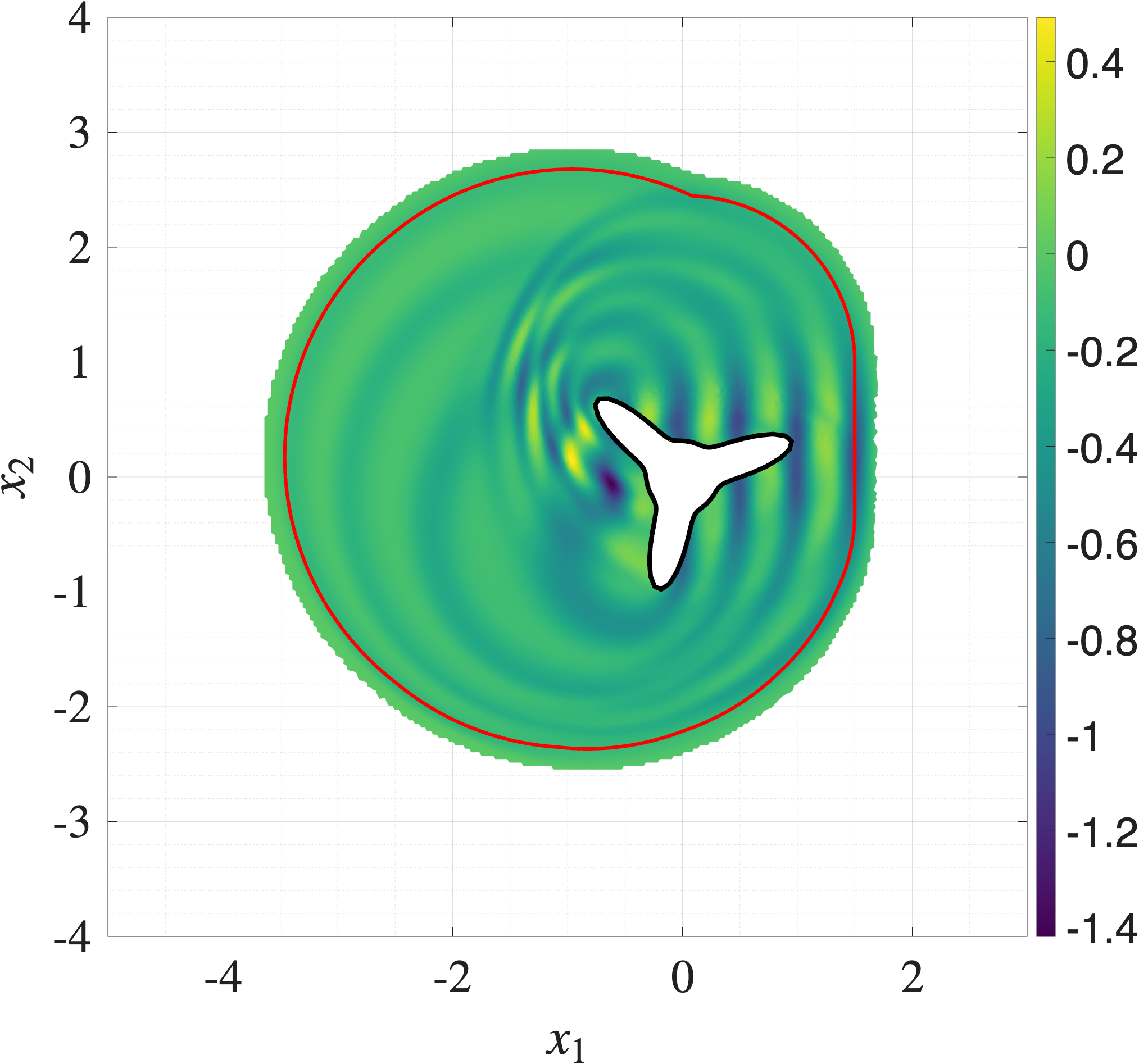}
        \caption{\emph{Rotating}: $t=3$.}
    \end{subfigure}
    \hfill
    \begin{subfigure}{0.32\textwidth}
        \centering
        \includegraphics[width=\textwidth]{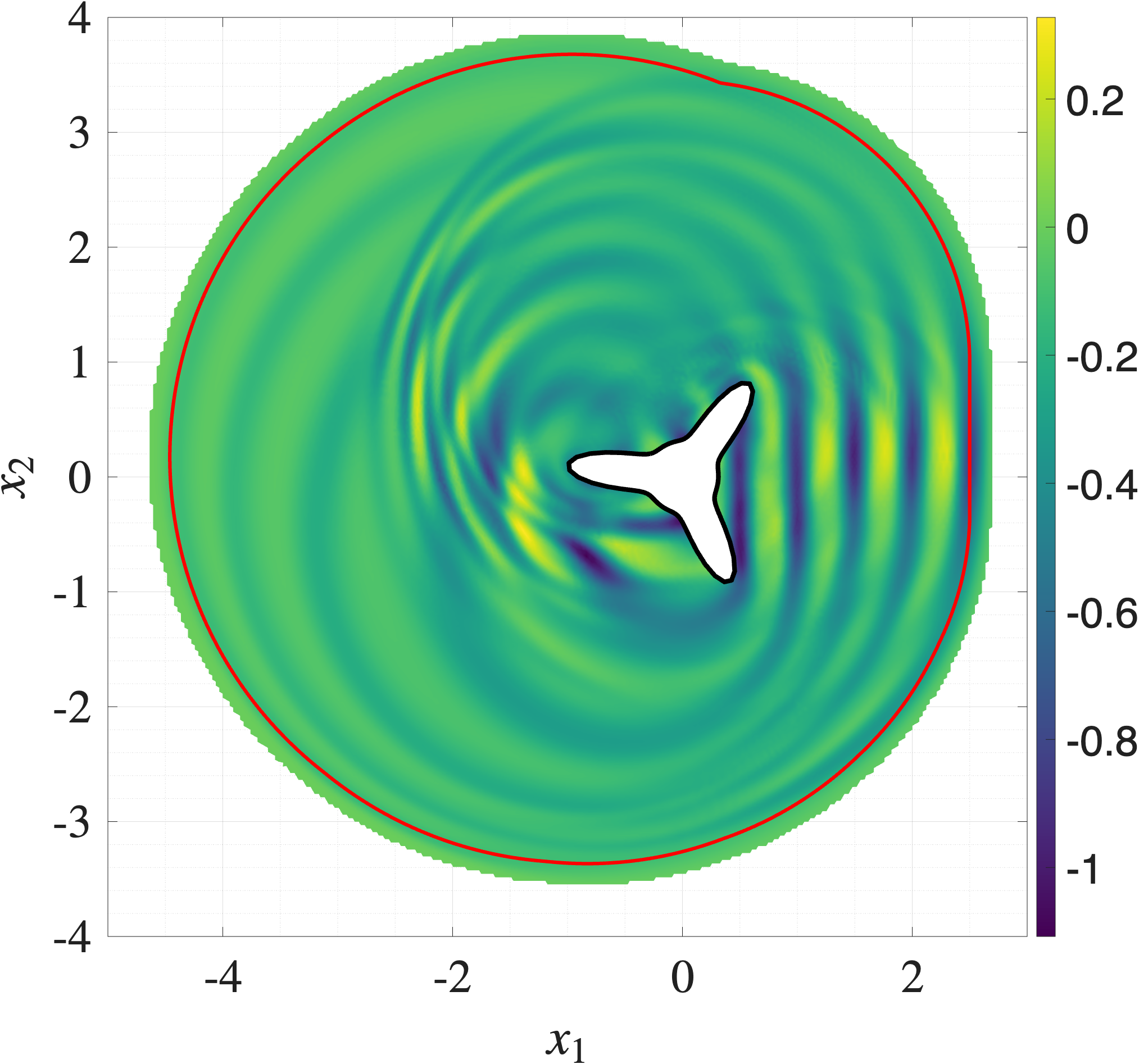}
        \caption{\emph{Rotating}: $t=4$.}
    \end{subfigure}

    \caption{
    Numerical solutions for the smooth turbine test. The top row shows the
    \emph{fixed} case with $U=0.0$, and the bottom row shows the
    \emph{rotating} case with $U=0.5$. In each panel, the reflected
    wavefront is shown as the red curve, and the scattered potential is
    shown as the heatmap. The scattered potential is computed only in the
    causal region behind the reflected front.
    }
    \label{fig:fan}
\end{figure}


\section{Scattering by spatial heterogeneities}
\label{sec:space-dependent}

We now specialize the abstract geometric--optics parametrix framework of
\Cref{sec:universal} to the purely space-dependent setting
\begin{equation} \label{c-x}
  c(x,t) \equiv c(x),
  \qquad
  \text{with }
  0 < c_{\min} \le c(x) \le c_{\max} < \infty  \ \text{for }  x \in \mathbb{R}^3, 
\end{equation}
and assume that the interface $\Gamma$ is fixed.  The sound-speed profile is 
assumed to encode the interior--exterior contrast \eqref{c-piecewise-smooth}, with 
constant sound speed profiles $c^-(x) \equiv C^-$ and $c^+(x) \equiv C^+$.

\subsection{Variational characterization of the travel-time function $\eta$}

Recall that the travel-time function
$\eta(x,t;y,\tau)$ satisfies the eikonal equation \eqref{eikonal-delay-st}.
If $c=c(x)$ and $\eta=\eta(x;y)$ are independent of $t$ and $\tau$, then 
\begin{subequations}\label{eikonal-space-only}
\begin{equation}
  |\nabla_x \eta(x;y)|
  =
  \frac{1}{c(x)},
  \qquad x\neq y,
\end{equation}
with the point-source condition
\begin{equation}\label{eikonal-space-only-ic}
  \eta(y;y)=0.
\end{equation}
\end{subequations}
Volumetric discretizations of \eqref{eikonal-space-only}, such as
fast marching or fast sweeping methods \cite{Sethian1996,Zhao2005},
are incompatible with our boundary-only framework.
Instead, we characterize $\eta(x;y)$ through 
the variational principle
\begin{subequations} \label{eikonal-var-space-only}
\begin{equation}\label{travel-time-xy}
  \eta( x ; y)
  =
  \inf_{\gamma}
  \int_0^1 \frac{| \gamma' (r)|}{c(\gamma(r))}\,\dd r,
\end{equation}
where the infimum is taken over all curves $\gamma:[0,1]\to\mathbb{R}^3$ with
\begin{equation} \label{rays-bc}
\gamma(0)=y \quad \text{and} \quad  \gamma(1)=x. 
\end{equation}
\end{subequations}

\subsubsection{Hamiltonian ray equations}

The stationary curves of the variational problem
\eqref{eikonal-var-space-only} satisfy the Euler--Lagrange equations,
which are equivalently expressed as the Hamiltonian ray system
\begin{subequations}\label{hamilton-odes}
\begin{align}
  \gamma'(r) &= \nabla_p H(\gamma(r),p(r)), \\
  p'(r) &= -\nabla_x H(\gamma(r),p(r)),
\end{align}
with Hamiltonian
\begin{equation}\label{hamiltonian}
  H(x,p)
  =
  \tfrac12\bigl(|p|^2-c(x)^{-2}\bigr).
\end{equation}
\end{subequations}
For the travel time from the source point $y$ to the observation point $x$,
the ray must satisfy the two-point boundary conditions \eqref{rays-bc}.
Along a minimizing ray, the constraint
\begin{equation}\label{hamiltonian-constraint}
  H(\gamma(r),p(r))=0
\end{equation}
holds, which is equivalent to the eikonal relation $|p(r)|=c(\gamma(r))^{-1}$.
In principle, one may therefore compute $\eta(x;y)$ by solving the
two-point boundary value problem \eqref{hamilton-odes}, \eqref{rays-bc},
and \eqref{hamiltonian-constraint}. In the boundary-integral
setting, however, this is not practical: the problem is nonlinear,
must be solved separately for each source--observation pair, and becomes
prohibitively expensive when travel times are required at many boundary
points. This motivates the use of simpler approximations to the minimizing
ray, which we develop next.

\subsection{Approximation of geometric acoustic rays}

A simple approximation to the geometric ray from $y$ to $x$ is obtained by
replacing the minimizing curve in \eqref{eikonal-var-space-only} by the straight
Euclidean chord joining the endpoints. Define the chord
\begin{subequations}\label{eta-chord}
\begin{equation}\label{rays-chord}
  \gammach(r)
  \coloneqq
  y+r(x-y),
  \qquad r\in[0,1],
\end{equation}
and approximate the travel time by
\begin{equation}\label{T-chord}
  \eta(x;y)
  =
  \int_0^1 \frac{|\gammach'(r)|}{c(\gammach(r))} \dd r
  =
  |x-y|
  \int_0^1 \frac{1}{c (y+r(x-y) )} \dd r.
\end{equation}
\end{subequations}
The chord approximation \eqref{eta-chord} provides a simple model travel-time function
that is consistent with the basic structural requirements of
Definition \ref{def:characteristic-coordinates}.

For wave scattering by a convex domain $\Omega^-$ with the interior--exterior contrast
profiles \eqref{c-piecewise-smooth}, the chord approximation is particularly
effective in the fast-inclusion case $C^- > C^+$.
In this regime, the minimizing ray typically passes through the inclusion, and since the
sound speed is constant away from the thin transition layer
$\Omega_\delta$, the corresponding ray path is straight except for its
passage through the interfacial layer.
Hence the straight chord furnishes an accurate 
approximation to the true geometric acoustic ray.

The situation changes substantially in the slow-inclusion case $C^- < C^+$.
In this regime, the minimizing ray may avoid the inclusion altogether, so the
straight chord can produce substantial travel-time error and fail to capture
the correct ray geometry. To recover the appropriate stationary path while
retaining a boundary-based formulation, we refine the chord initialization by a
Newton iteration for the ray equations (see \Cref{subsec:newton-ray-continuous}).

\subsection{Amplitude transport and a two-point model}

Since $\eta$ is independent of $t$, it follows from 
\eqref{transport-eq} that the amplitude is independent of time, and 
the transport equation \eqref{transport-eq} becomes
\begin{subequations}\label{amp-x}
\begin{equation}\label{transport-eq-space-only}
  2\,\nabla_x \amp(x;y)\cdot\nabla_x\eta(x;y)
  + \amp(x;y)\,\Delta_x\eta(x;y)
  =
  \frac{\amp(x;y)}{c(x)^2}\,\frac{2}{\eta(x;y)},
  \qquad x\neq y.
\end{equation}
The normalization condition \eqref{A-normalization-st} becomes 
\begin{equation}\label{A-normalization-space-only}
  \amp(y;y)=\kappa(y).
\end{equation}
\end{subequations}

\subsubsection{Transport ODE along travel-time rays.}

Let $\gamma(s)$ be a travel-time ray emanating from $y$, parameterized by
travel time so that $\eta(\gamma(r);y)=r$. Since
$|\nabla_x\eta(x;y)|=c(x)^{-1}$, differentiating
$\eta(\gamma(r);y)=s$ gives 
$\nabla_x\eta(\gamma(r);y)\cdot \gamma'(r)=1$,
and hence
\begin{equation}\label{ray-ode}
  \gamma'(r)
  =
  c(\gamma(r))^2\,\nabla_x\eta(\gamma(r);y),
  \qquad
  \gamma(0)=y.
\end{equation}
Define the amplitude along the ray by
\[
  \mathfrak{a}(r)\coloneqq A(\gamma(r);y).
\]
By the chain rule and \eqref{ray-ode},
\[
  \frac{\mathrm{d} \mathfrak{a}}{\mathrm{d}r}(r)
  =
  \nabla_x \amp (\gamma(r),y)\cdot\frac{\mathrm{d}\gamma}{\mathrm{d}r}(r)
  =
  c(\gamma(r))^2\,\nabla_x \amp(\gamma(r),y)\cdot
  \nabla_x\eta(\gamma(r);y).
\]
Evaluating the transport equation \eqref{amp-x}
along the ray and using $\eta(\gamma(r);y)=r$ yields the scalar ODE
\begin{subequations}\label{transport-ode}
\begin{equation}
  \frac{\mathrm{d} \mathfrak{a} }{\mathrm{d}r}(r)
  +
  \tfrac12\,c(\gamma(r))^2\,\Delta_x\eta(\gamma(r);y)\, \mathfrak{a}(r)
  =
  \tfrac{1}{r}\, \mathfrak{a} r,
  \qquad r>0.
\end{equation}
The local normalization consistent with \eqref{A-normalization-st} is
\begin{equation}\label{a0}
  \lim_{r\downarrow 0} \mathfrak{a} (r)= \amp(y,y)=\kappa(y). 
\end{equation}
\end{subequations}

\subsubsection{Two-point model.}

Solving \eqref{transport-eq-space-only} exactly requires the quantity
$\Delta_x\eta(x;y)$, which encodes the divergence of nearby rays. 
This is a second-order geometric quantity which is
singular at the source point $x=y$, and is substantially more delicate to
approximate accurately. For this reason, we do not solve
\eqref{transport-eq-space-only} exactly, and instead adopt a simpler model for
the amplitude.

A local geometric--optics argument shows that the 
prefactor $\amp(x;y)$ scales like $c(x)^{-1/2}$ along the
ray.\footnote{%
A heuristic justification of \eqref{A-sqrt-model} follows from
energy-flux conservation in geometric optics.
For the variable-speed wave equation
$\frac{1}{c(x)^2}\partial_t^2\phi - \Delta \phi = 0$,
the conserved energy density and flux are
$e=\tfrac12\!\left(\tfrac{1}{c(x)^2}|\partial_t\phi|^2 +|\nabla\phi|^2\right)$ 
and  $\mathbf{S}=-\partial_t\phi\,\nabla\phi$.
Inserting a high-frequency ansatz $\phi(x,t)=a(x)e^{i\omega(\eta(x)-t)}$,
using the eikonal equation \eqref{eikonal-space-only},
and retaining leading-order terms gives
$e \sim \frac{\omega^2 a(x)^2}{c(x)^2}$ and 
$|\mathbf{S}| \sim \frac{\omega^2 a(x)^2}{c(x)}$.
Energy conservation along a ray tube implies $|\mathbf{S}|\,J(x)=\text{const}$,
where $J(x)$ is the cross-sectional area of the tube. Hence
$a(x)^2 \propto \frac{c(x)}{J(x)}$.
A wavefront at travel-time distance $\eta(x)$ has local
physical radius of order $c(x)\eta(x)$, so $J(x)\propto c(x)^2\eta(x)^2$.
Therefore $a(x)\propto c(x)^{-1/2}\eta(x)^{-1}$.
Since the reference kernel $\Gopo$ 
already contains the universal free-space spreading factor $\eta^{-1}$, the
remaining prefactor scales like $\amp(x;y) \propto c(x)^{-1/2}$.}
The normalization condition \eqref{A-normalization-space-only} then implies the two-point model
\begin{equation}\label{A-sqrt-model}
  \amp(x;y)
  =
  \tfrac{1}{\sqrt{c(x)c(y)}}.
\end{equation}
The model \eqref{A-sqrt-model} does not, in general, satisfy the
exact transport equation \eqref{amp-x}. Substituting \eqref{A-sqrt-model} into
\eqref{amp-x} shows that the transport equation is satisfied exactly if and only if
\begin{equation}\label{geom-spreading-balance}
  \tfrac{2}{c(x)^2\,\eta(x;y)}-\Delta_x\eta(x;y)
  =
  -\nabla_x\log c(x)\cdot\nabla_x\eta(x;y).
\end{equation}
When \eqref{geom-spreading-balance} fails, the discrepancy measures the 
contribution of ray-tube divergence not captured by the two-point model.
 Thus \eqref{A-sqrt-model} should be viewed as a
simplified closure that retains the leading local sound-speed dependence of the
amplitude while neglecting the remaining correction due to geometric spreading.
The model is symmetric in $(x,y)$, satisfies the diagonal normalization
\eqref{A-normalization-space-only}, and reduces to the constant-speed
amplitude $\amp\equiv c_0^{-1}$ when $c(x)\equiv c_0$.

\subsection{Summary of slabwise formulas for evaluation of temporal weights}

We summarize the quantities required for the evaluation of the 
weights \eqref{sltw-theta-frozen-closed}--\eqref{dltw-theta-frozen-closed}. 
\begin{itemize}
\item The travel-time function $\upeta(\alpha;\beta)$ is obtained from the variational
problem \eqref{eikonal-var-space-only}. In the fast-inclusion regime, a useful
first approximation is the chord formula \eqref{eta-chord}; in the slow-inclusion
regime, this approximation may be refined by the Newton iteration described in
Appendix~\ref{subsec:newton-ray-continuous}. 
\item The amplitude is approximated by the two-point model
\eqref{A-sqrt-model}. Restricting to boundary points
$x=z(\alpha)$ and $y=z(\beta)$ gives
\begin{equation}
  A(\alpha;\beta)
  =
  \tfrac{1}{\sqrt{C(\alpha)\,C(\beta)}}.
\end{equation}
\end{itemize}

\subsection{Numerical example: a spherical gas bubble in water}

We consider a model of a spherical gas bubble in water.
The wave equation \eqref{L-general-pde} is interpreted 
in nondimensional variables obtained from the dimensional acoustic equation
through the scaling
\begin{equation}\label{slow-inclusion-scales}
  x_{\mathsf{phys}} = L_0\,x,
  \qquad
  t_{\mathsf{phys}} = T_0\,t,
  \qquad
  c_{\mathsf{phys}}(x_{\mathsf{phys}})
  = c_0\,c(x),
  \qquad
  T_0 = \tfrac{L_0}{c_0}.
\end{equation}
The ambient medium is water and the
inclusion represents an air-filled bubble or cavity. This provides a
physically motivated model of a strong slow-inclusion regime.
The reference scales are chosen as
\begin{equation}\label{slow-inclusion-reference-scales}
  L_0 = 10^{-3} \, \unit{\meter},
  \qquad
  c_0 = 1500 \, \unit{\meter\per\second},
  \qquad
  T_0 = \tfrac{L_0}{c_0} = 6.67\times 10^{-7}\, \unit{\second}. 
\end{equation}
With water as the reference medium, the dimensionless sound speeds are
\begin{equation}\label{slow-inclusion-contrast}
  C^+ = 1
  \qquad \text{and} \qquad
  C^- = \tfrac{c_{\mathsf{air}}}{c_{\mathsf{water}}}
  \approx 0.227.
\end{equation}
The interface is the fixed unit sphere
\begin{equation}\label{slow-inclusion-interface}
  z(\alpha)=\alpha,
  \quad
  \alpha\in\mathbb{S}^2,
\end{equation}
which corresponds in physical units to a bubble of radius $1 \, \unit{\milli\meter}$. 
Let $d(x)$ denote the signed distance to the interface. 
We prescribe a smooth sound-speed profile by
\begin{equation}\label{slow-inclusion-c}
  c(x)
  =
  C^+
  +
  \tfrac{C^- - C^+}{2}
  \left(
    1+\tanh\!\left(-\tfrac{d(x)}{\delta}\right)
  \right),
\end{equation}
where $\delta = 0.2$ is the transition layer width.
The incident field is given by a Gaussian plane-wave packet propagating in
the positive $x_1$-direction,
\begin{equation}\label{slow-inclusion-inc}
  \phiinc(x_1,x_2,x_3,t)
  =
  \exp\left(
    -\left(\tfrac{x_1-t+x_0}{a}\right)^2
  \right),
  \qquad
  x_0=1.5,
  \qquad
  a=0.2.
\end{equation}

\subsubsection{Numerical simulation using the single-layer solver}

We simulate the scattering of the incident wave \eqref{slow-inclusion-inc}
from the gas bubble using the single-layer solver on a mesh
with 713 triangles. The computation is run with time-step $\Delta t = 0.1$ until the final nondimensional time
$\tmax=3$, and the travel-time function $\eta$ is approximated by the Newton ray
construction described in Appendix~\ref{subsec:newton-ray-continuous}.

\Cref{fig:slow-inclusion-newton} displays the computed solution at the final
time $\tmax=3$. The left panel shows the $x_2=0$ slice of the sound-speed field,
together with a representative family of three-dimensional rays emanating from
the source point $(-1,0,0)$ and terminating at a selection of target points on
the interface. The right panel shows the corresponding approximate scattered
field $\tilde{\phi}(x,t)$ defined by \eqref{layer-ansatze-a}, restricted to the
same plane.

The rays avoid the slow interior and instead propagate through the exterior
domain around the boundary of the bubble. The scattered field in the plane
$x_2=0$ therefore develops an asymmetric reflected pattern. A pronounced
positive crescent appears along the right-hand side of the bubble, while a
broader and weaker positive band extends around the left-hand side of the
pattern. Outside the bright downstream crescent, the field transitions into a
narrow negative trough across the reflected wavefront. 
The boundary response generated on the illuminated side of the bubble is carried
around the obstacle by the bypassing ray family and then concentrated most
strongly on the downstream side of the slow inclusion.

\begin{figure}[ht]
    \centering

    \begin{subfigure}{0.45\textwidth}
        \centering
        \includegraphics[width=1.0\textwidth]{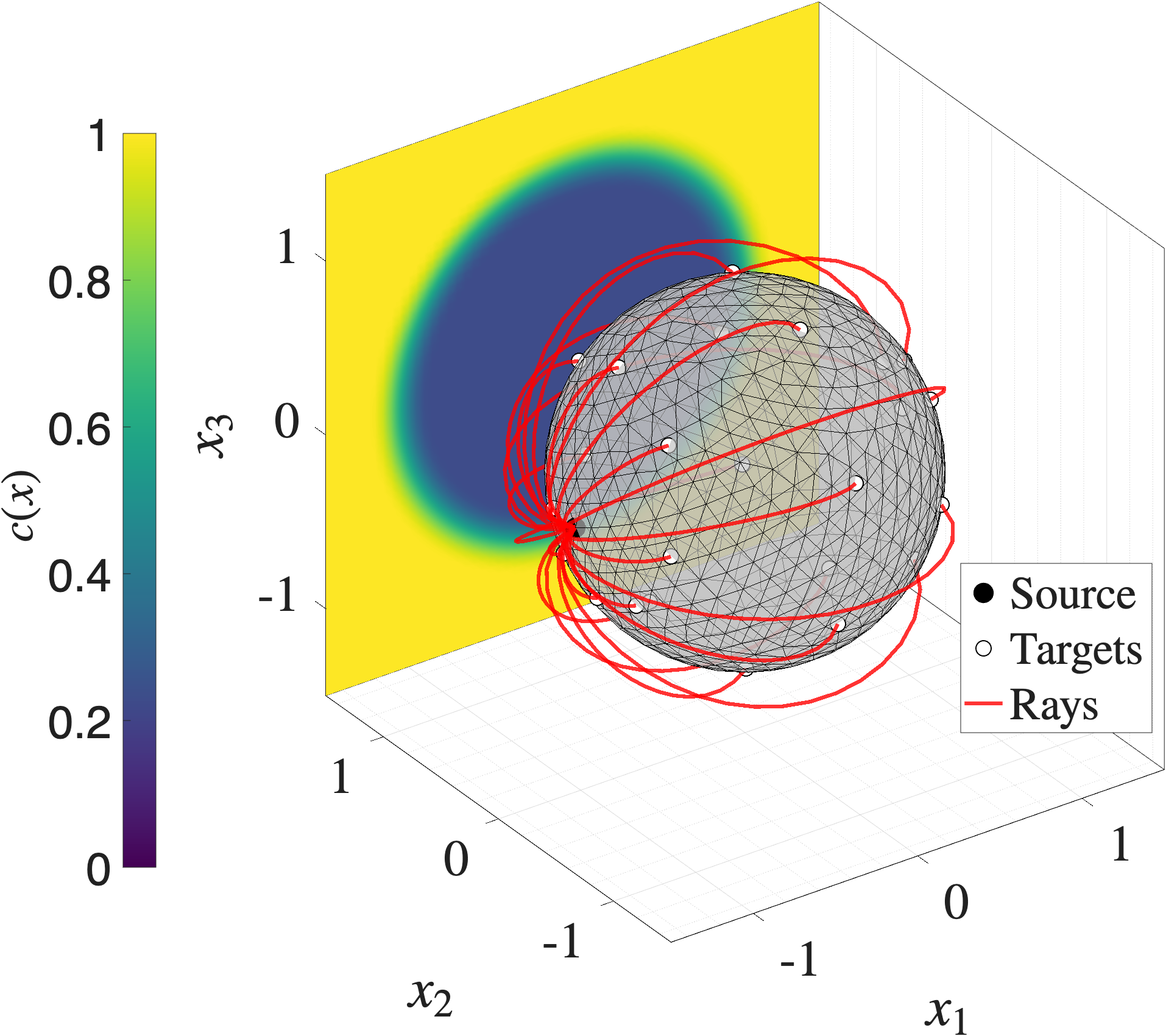}
        \caption{Sound-speed field with representative rays.}
    \end{subfigure}
    \hfill
    \begin{subfigure}{0.45\textwidth}
        \centering
        \includegraphics[width=0.95\textwidth]{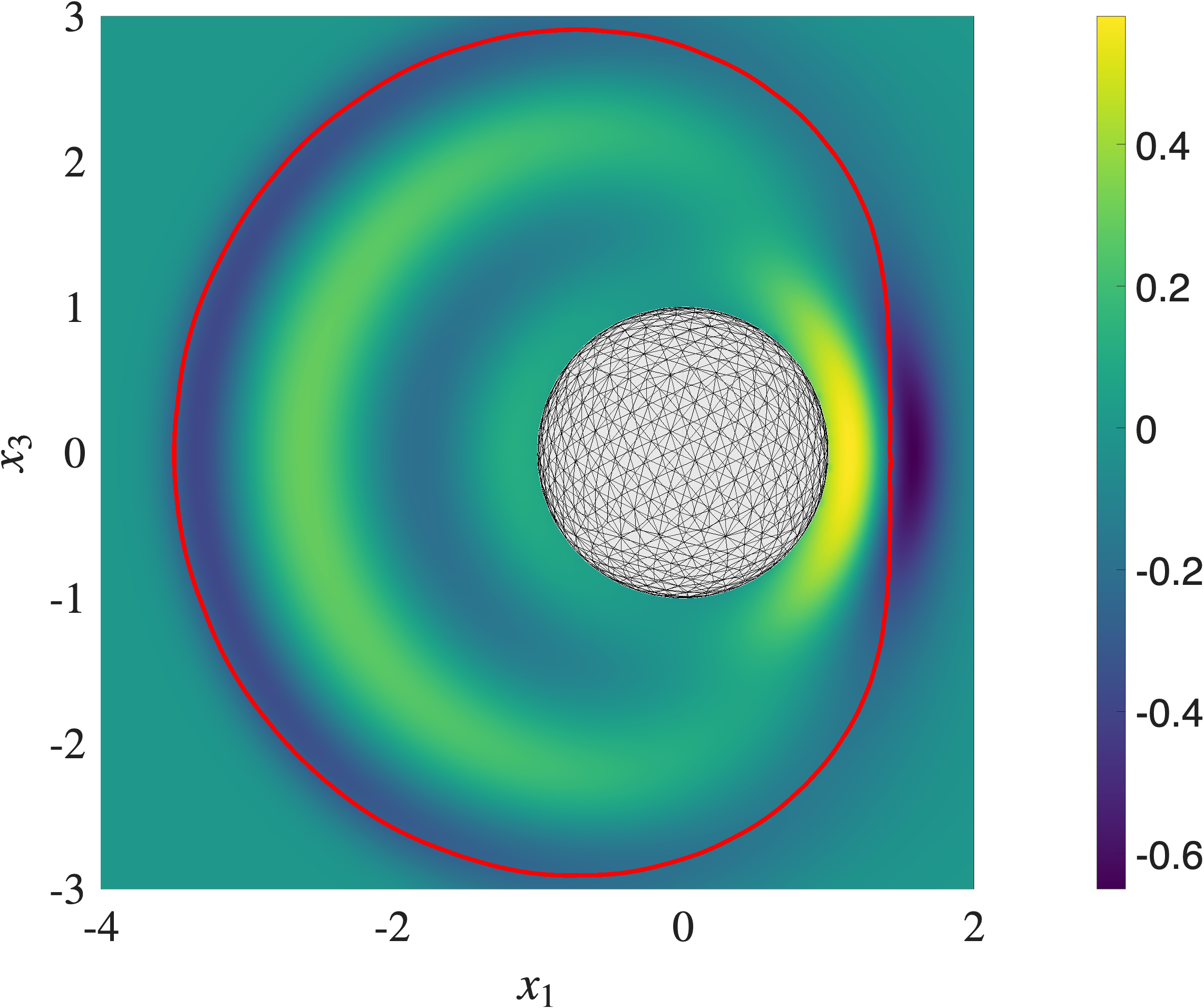}
        \caption{Approximate scattered field $\tilde{\phi}(x,t)$ at $t=3$.}
    \end{subfigure}

   \caption{
    Numerical solution for the spherical gas bubble in water at the final
    nondimensional time $t=3$. \textbf{Left:} $x_2=0$ slice of the sound-speed field,
    together with a representative family of three-dimensional rays from the source
    point $(-1,0,0)$ to selected target points on the interface. \textbf{Right:}
    corresponding slice of the approximate scattered field $\tilde{\phi}(x,t)$ in
    the plane $x_2=0$, with the reflected wavefront in that plane overlaid in red.
    }
    \label{fig:slow-inclusion-newton}
\end{figure}

For comparison, we also compute the corresponding solution in the ambient case,
obtained by setting the sound speed equal to the exterior value $C^+=1$
throughout the domain. The left panel of
\Cref{fig:slow-inclusion-comparison} shows the corresponding slice of the
scattered potential $\tilde{\phi}(x,t)$ in the plane $x_2=0$. The reflected
wavefront itself is nearly unchanged from the gas-bubble case, since in the
latter the sound speed differs from $C^+=1$ only within a thin tubular
neighborhood of the interface. As a result, the travel time of the principal
reflected arrival is only weakly perturbed by the bubble. In contrast to the
gas-bubble case, however, the ambient solution retains a much simpler reflected
pattern, lacking the strong downstream positive crescent and the broader
asymmetric structure produced by the slow inclusion.

\begin{figure}[ht]
    \centering

    \begin{subfigure}{0.45\textwidth}
        \centering
        \includegraphics[width=0.95\textwidth]{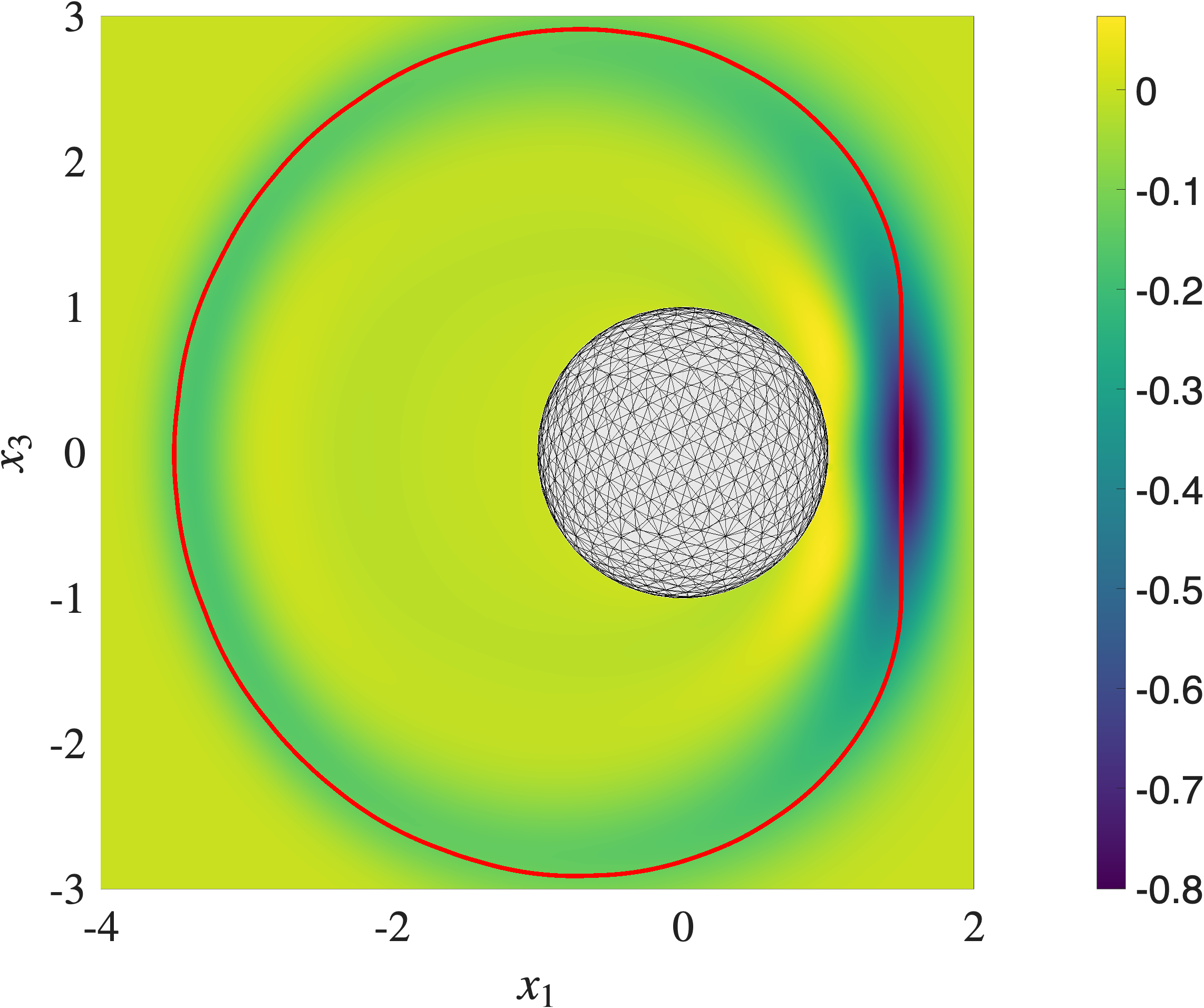}
        \caption{Ambient case.}
    \end{subfigure}
    \hspace{2em}
    \begin{subfigure}{0.45\textwidth}
        \centering
        \includegraphics[width=0.95\textwidth]{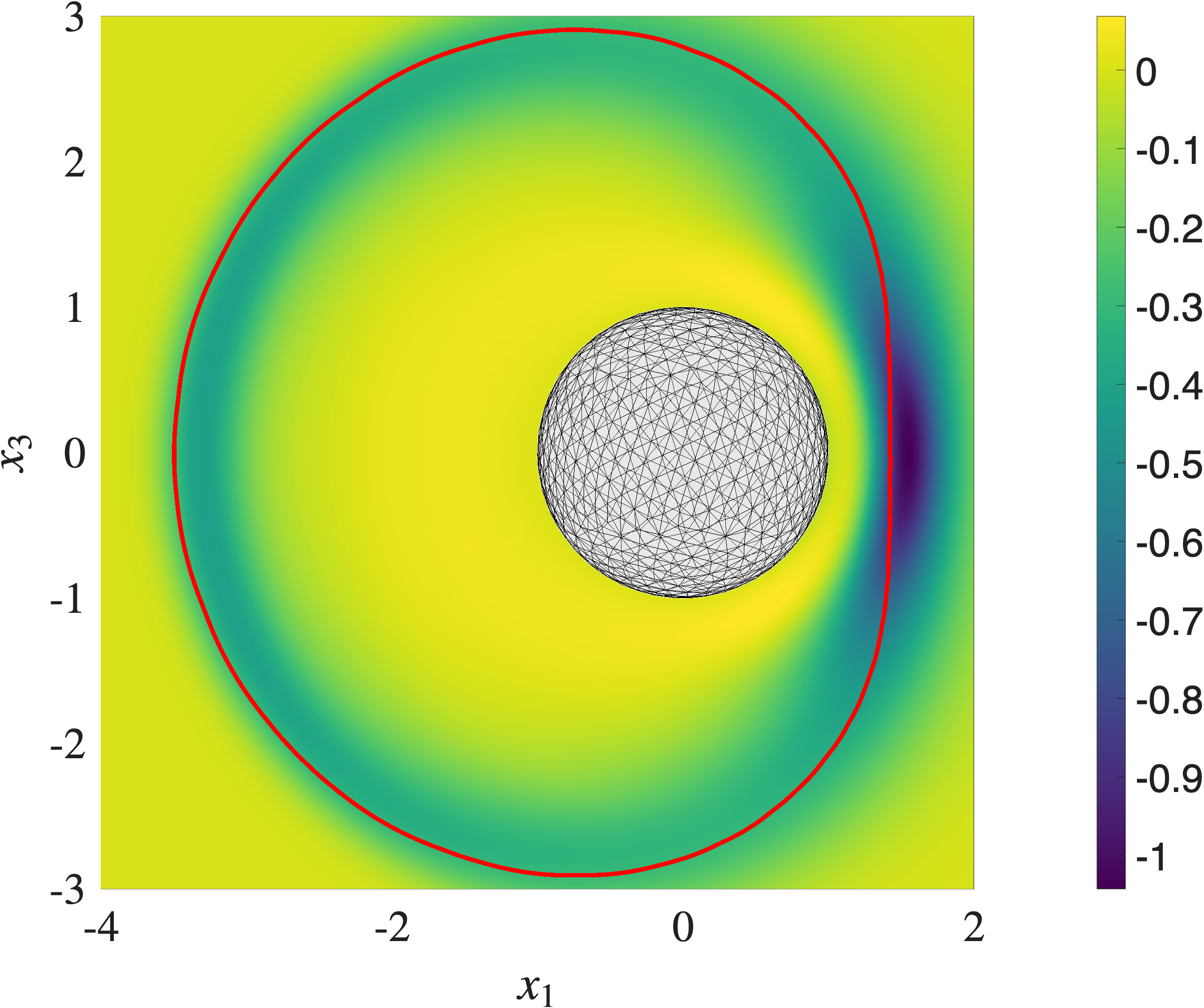}
        \caption{Gas-bubble case with chord approximation.}
    \end{subfigure}

    \caption{
    Reference comparison for the spherical gas-bubble test at the final
    nondimensional time $t=3$. The plots show the slice of the scattered field
    $\tilde{\phi}(x,t)$ in the plane $x_2=0$, with the reflected wavefront in that
    plane overlaid in red. \textbf{Left}: ambient reference case, obtained by setting
    $c(x)\equiv C^+=1$. \textbf{Right}: gas-bubble case with the travel-time function
    approximated by the chord model \eqref{eta-chord}.
    }
    \label{fig:slow-inclusion-comparison}
\end{figure}

We also repeat the gas-bubble computation using the chord approximation
\eqref{eta-chord} in place of the Newton ray construction for the travel-time
function $\eta$. The resulting slice of the scattered potential
$\tilde{\phi}(x,t)$ in the plane $x_2=0$ is shown in the right panel of
\Cref{fig:slow-inclusion-comparison}. Although the chord-based simulation
captures a perturbation of the reflected wave due to the bubble, it does not
reproduce the same scattered-field geometry as the Newton-based simulation. In
particular, the pronounced downstream positive crescent and the broader
asymmetric reflected pattern seen in \Cref{fig:slow-inclusion-newton} are not
recovered. Taken together,
\Cref{fig:slow-inclusion-newton,fig:slow-inclusion-comparison} show that the
asymmetric reflected structure is a genuine signature of the slow inclusion and
is recovered only when the travel-time function reflects the correct bypass ray
geometry.


\section{Scattering by spatio-temporal heterogeneities}
\label{sec:space-time-dependent}

We now return to the general space-time--dependent sound speed 
operator \eqref{L-general-pde} with sound speed $c = c(x,t)$
satisfying \eqref{c-bounds}.  

\subsection{A chord model for the travel-time function}
\label{subsec:eikonal-st}

Recall that the exact travel-time
function $\eta(x,t;y,\tau)$ is governed by the space-time eikonal equation
\eqref{eikonal-delay-st}. In general, solving \eqref{eikonal-delay-st}
amounts to computing genuinely space-time characteristics, which is
considerably more complicated than in the space-dependent setting.

As in \Cref{sec:space-dependent}, we therefore introduce a simple closure
model for $\eta$. The key point is that when $c=c(x,t)$, a propagation path
must specify not only a spatial curve joining $y$ to $x$, but also the
physical time at which that curve is traversed. 
Equivalently, if $\gamma:[0,1]\to\mathbb{R}^3$
is a spatial path from $y$ to $x$, then one must also specify a monotone
\emph{time-lift}
\[
  s:[0,1]\to[\tau,t],
  \qquad s(0)=\tau, \qquad s(1)=t,
\]
which records the time at which the path passes through the point
$\gamma(r)$. The corresponding travel-time functional is
\begin{equation}\label{eta-functional-st}
  \int_0^1 \frac{|\gamma'(r)|}{c(\gamma(r),s(r))}\,\dd r.
\end{equation}

Our closure model is obtained by making the simplest possible choices for
these two ingredients. We replace the unknown ray geometry by the Euclidean
chord
\begin{subequations}\label{eta-chord-st}
\begin{equation}\label{chord-geometry-st}
  \gammach(r) = y + r(x-y),
  \qquad 0\le r\le 1,
\end{equation}
and we replace the unknown time history along the ray by the linear
time-lift
\begin{equation}\label{time-lift-linear-st}
  s_{\mathsf{lin}}(r) = \tau + r(t-\tau).
\end{equation}
Substituting \eqref{chord-geometry-st} and \eqref{time-lift-linear-st}
into \eqref{eta-functional-st} yields the space-time chord model
\begin{equation}\label{eta-chord-functional}
  \eta(x,t;y,\tau)
  =
  |x-y|
  \int_0^1
  \frac{1}{c(\gamma(r),\,\tau+r(t-\tau))}\,\dd r.
\end{equation}
\end{subequations}

The approximation \eqref{eta-chord-st} is the natural generalization 
of the space-only chord approximation \eqref{eta-chord} 
used in the space-dependent case. 
Under mild regularity and temporal-growth assumptions on the sound speed, it can be 
shown that the chord model \eqref{eta-chord-st} defines a travel-time 
function in the sense of Definition \ref{def:characteristic-coordinates}. 
It becomes exact when $c$ is constant, and reduces to
the earlier space-only chord model \eqref{eta-chord} when $c=c(x)$ is independent of time.

\subsection{A two-point model for the amplitude}

The amplitude $\amp(x,t;y,\tau)$ is determined, in principle, by the
transport equation \eqref{transport-st}. As in the space-dependent case,
we do not attempt to solve this equation exactly. 
Instead, we adopt the natural space-time analogue of the two-point
model \eqref{A-sqrt-model} used in \Cref{sec:space-dependent}, namely
\begin{equation}\label{A-impedance-st}
  \amp(x,t;y,\tau)
  =
  \tfrac{1}{\sqrt{c(x,t)c(y,\tau)}}.
\end{equation}
This model retains the leading endpoint dependence of the amplitude on the
local sound speed while neglecting the remaining correction due to
geometric spreading. It is consistent with the diagonal normalization
\eqref{A-normalization-st}, reduces to the space-dependent two-point model
when $c=c(x)$ is independent of time, and reduces to the 
constant-speed amplitude in the homogeneous case.

\subsection{Benchmark example: scattering with time-dependent sound speed}

We consider a planar example in which the sound speed depends only on time,
$c=c(t)$, and the scattering interface is the fixed unit circle
\begin{equation}\label{unit-circle}
  z(\alpha)=\alpha,
  \quad
  \alpha\in\mathbb{S}^1.
\end{equation}
This provides a convenient benchmark problem, since
the corresponding solution can also be computed using a simple two-dimensional finite element
method\footnote{
For the exterior Dirichlet problem
\eqref{L-general-pde}--\eqref{L-bcs} with time-dependent sound 
speed $c(t)$, a finite element reference
solution may be computed by solving the wave equation only in
$\Omega^+(t)$ with Dirichlet data prescribed on $\Gamma(t)$.
If the sound speed varies spatially across the interface and
transmission effects are to be retained, an exterior solve alone
would not capture the influence of the interior medium, and a
coupled two-phase formulation on $\Omega^\pm(t)$ would be required.
Our finite-element discretization utilizes an unstructured triangular mesh with piecewise 
quadratic basis functions in space and a second-order explicit central difference (leapfrog) 
scheme in time, with the time step chosen to satisfy the CFL condition $\Delta t \lesssim h/c_{\max}$.
} for the exterior scattering problem
\eqref{L-general-pde}--\eqref{L-bcs}.

The sound speed is prescribed by the
smooth time-dependent profile
\begin{equation}\label{time-only-c}
  c(x,t)=c(t)
  =
  1 + \tfrac{c_{\mathsf{fin}}-1}{2}\left(1+\tanh(5(t-1.5))\right), \qquad c_{\mathsf{fin}} = \tfrac12.
\end{equation}
Thus $c(t)$ decreases smoothly from $1$ toward $1/2$ as $t$ passes
through $1.5$.
The incident field is taken to be the Gaussian plane-wave packet
\begin{equation}\label{time-only-inc}
  \phiinc(x_1,x_2,t)
  =
  \exp\!\left(
    -\left(\tfrac{x_1-t+x_0}{a}\right)^2
  \right),
  \qquad
  x_0=1.5,
  \qquad
  a=\tfrac15.
\end{equation}

We simulate scattering of the incident wave \eqref{time-only-inc} 
with the time-dependent sound speed \eqref{time-only-c} using the two-dimensional 
single-layer and double-layer solvers.  
The numerical solutions are computed on an 80-cell mesh with time step $\Delta t = 0.05$, and the associated scattered potentials are then 
evaluated on an unstructured exterior mesh 
with 3823 vertices,  7343 triangles, and maximum edge length $h = 0.2$.
For comparison, we also compute the finite element solution on the same exterior mesh.  
\Cref{fig:time-only-comparison} shows the three solutions at the final time $\tmax = 3.0$.

\begin{figure}[ht]
    \centering

    \begin{subfigure}{0.32\textwidth}
        \centering
        \includegraphics[width=\textwidth]{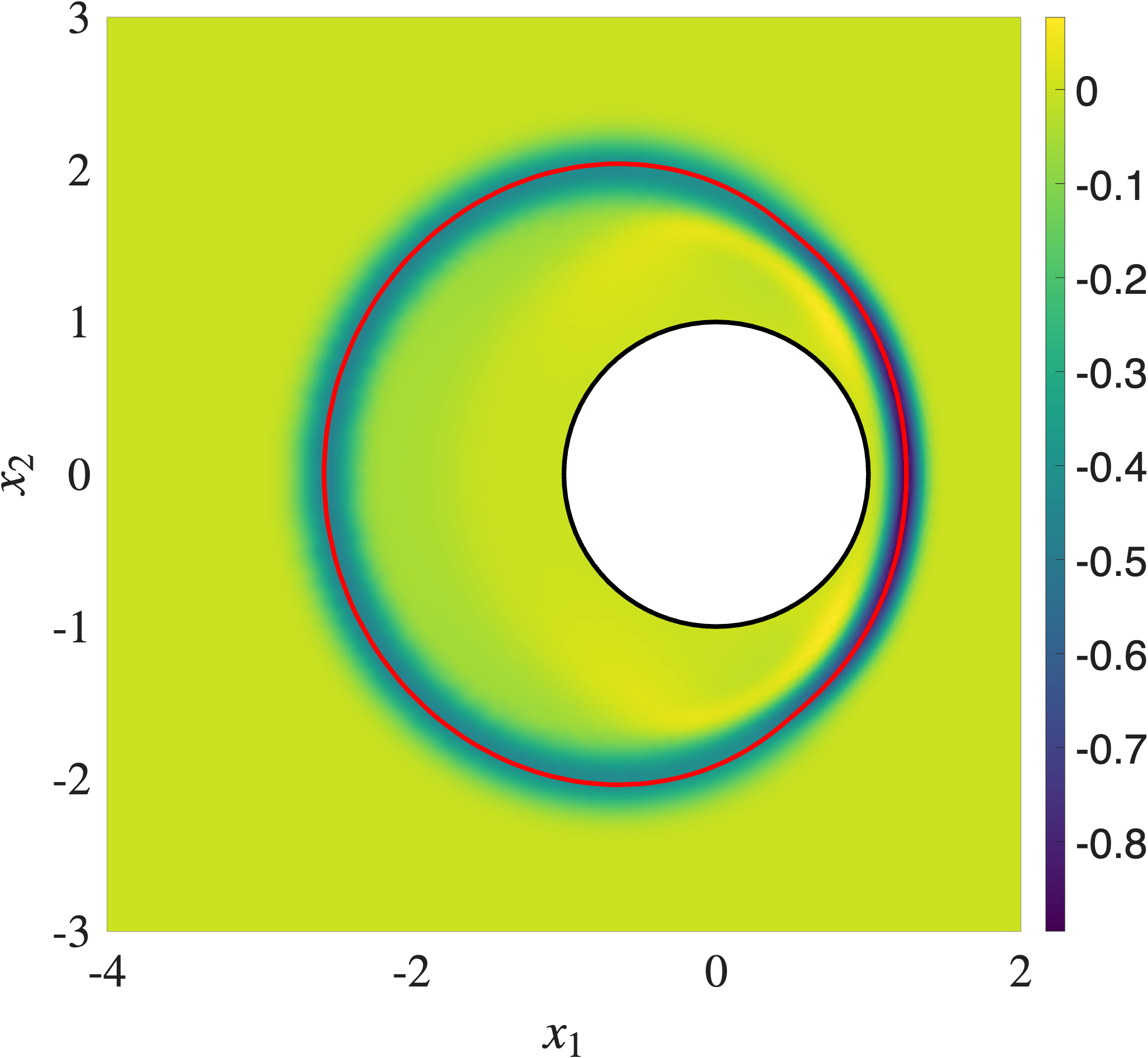}
        \caption{Single-layer solution.}
    \end{subfigure}
    \hfill
    \begin{subfigure}{0.32\textwidth}
        \centering
        \includegraphics[width=\textwidth]{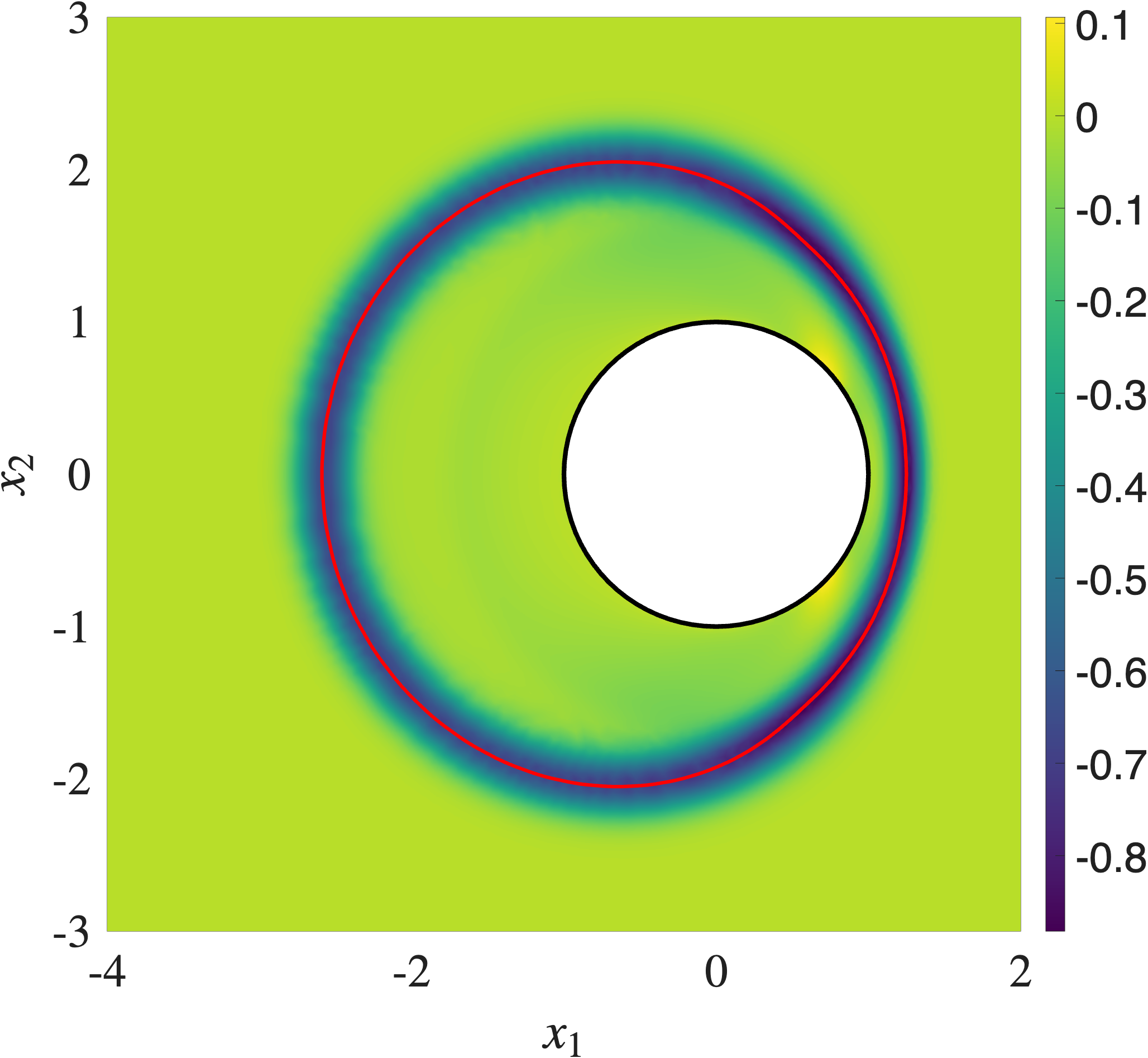}
        \caption{Double-layer solution.}
    \end{subfigure}
    \hfill
     \begin{subfigure}{0.32\textwidth}
        \centering
        \includegraphics[width=\textwidth]{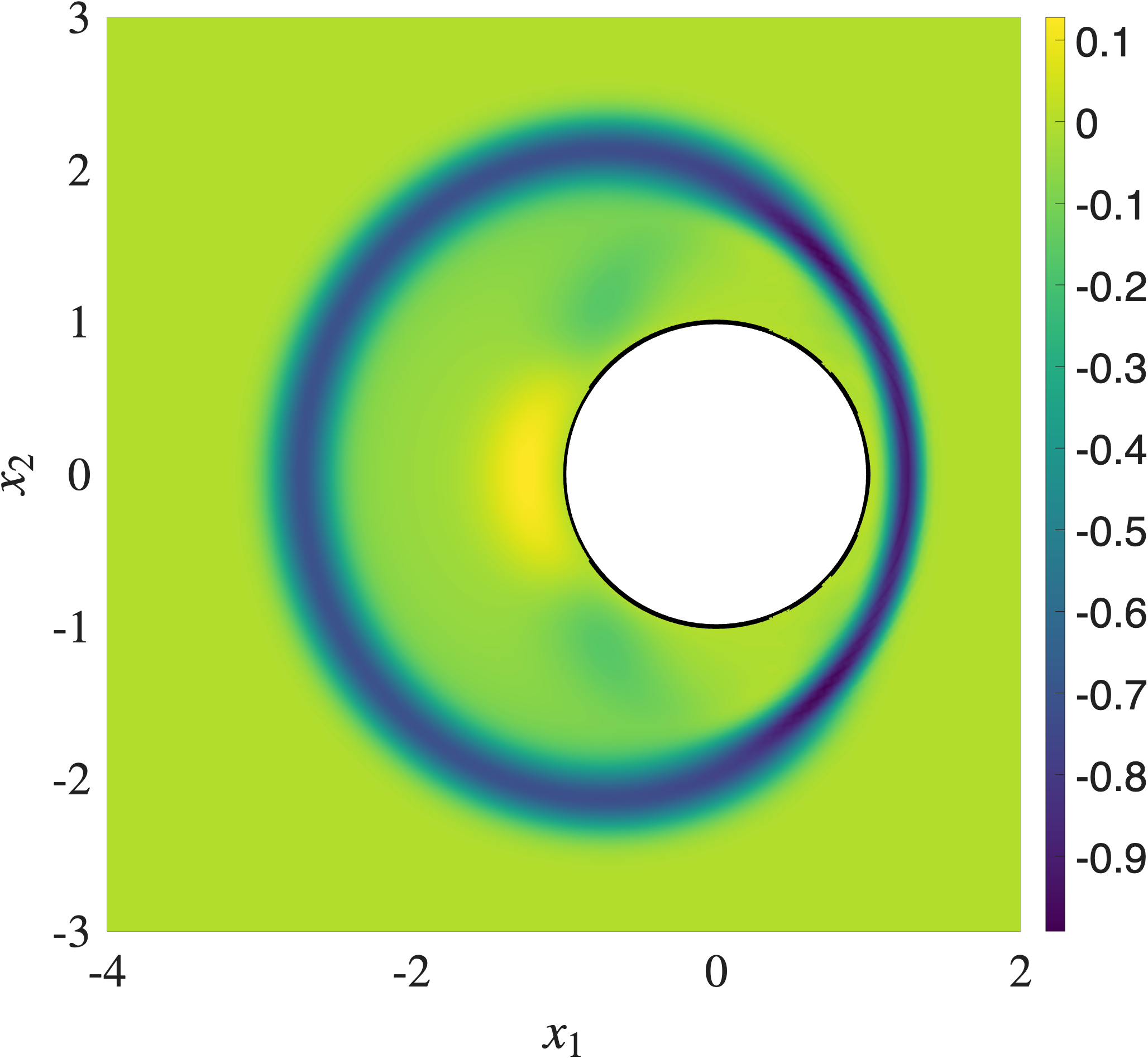}
        \caption{Finite element solution.}
    \end{subfigure}
    \caption{
    Comparison at $\tmax = 3.0$ for the benchmark problem with time-dependent sound speed \eqref{unit-circle}--\eqref{time-only-inc}.
    \textbf{Left}: two-dimensional single-layer solution with the reflected wavefront (red curve) overlaid.
    \textbf{Middle}: two-dimensional double-layer solution with the reflected wavefront (red curve) overlaid. 
    \textbf{Right}: finite element solution of the same exterior scattering problem. 
    In each panel, the obstacle boundary is shown in black.
    }
    \label{fig:time-only-comparison}
\end{figure}

The single-layer, double-layer, and finite element solutions shown in
\Cref{fig:time-only-comparison} are in good qualitative agreement. In
particular, all three computations recover the same overall reflected-wave
structure and place the principal reflected front in nearly the same location,
although differences remain in the detailed field amplitude behind the wavefront.
Such discrepancies are expected, since the single-layer and double-layer
parametrix formulations carry an intrinsic
approximation error. Even so, the comparison shows that both
layer-potential formulations capture the main scattering behavior in this
time-dependent benchmark and agree well with the independent finite element
solution.

\subsection{Numerical example: wave refraction by a rising fireball}

We interpret the wave equation \eqref{L-general-pde}
in nondimensional variables obtained from the dimensional acoustic equation
by the scaling
\begin{equation}\label{explosion-scales}
  x_{\mathsf{phys}} = L_0\,x,
  \qquad
  t_{\mathsf{phys}} = T_0\,t,
  \qquad
  c_{\mathsf{phys}}(x_{\mathsf{phys}},t_{\mathsf{phys}})
  = c_0\,c(x,t),
  \qquad
  T_0 = \tfrac{L_0}{c_0}.
\end{equation}
We consider an idealized model of an explosion-driven hot
bubble rising through the troposphere. The purpose is not to reproduce the
full compressible dynamics of a nuclear cloud, but rather to prescribe a
moving refractive inclusion with the interior--exterior contrast profile \eqref{c-piecewise-smooth} for the sound speed.

Our parameter choices are guided by the \emph{Teapot--Wasp} test case in \citet{Arthur2021}. 
In that study, a $1$ kT detonation is initiated at a burst height of
$232.3\, \unit{\meter}$ above ground level, producing an initially spherical hot fireball
with radius $R=155\, \unit{\meter}$ and temperature perturbation
$\Delta\theta=1000\, \unit{\kelvin}$.
For this example we set
\begin{equation}\label{explosion-parameter-values}
  L_0 = 100 \, \unit{\meter},
  \qquad
  c_0 = 340 \, \unit{\meter\per\second},
  \qquad
  T_0 =  0.294 \, \unit{\second},
  \qquad
  \Delta \theta_{\mathsf{bub}} = 1000 \, \unit{\kelvin},
  \qquad
  \varepsilon_0 = 50 \, \unit{\meter}.
\end{equation}
In the troposphere, the background temperature and sound-speed profiles are
\begin{equation}\label{explosion-theta-atm}
  \theta_{\mathsf{atm}}(x)
  =
  290 - \tfrac{70}{15000}\,x_3,
  \qquad 0 \le x_3 \le 15000 \, \unit{\meter},
\end{equation}
and
\begin{equation}\label{explosion-c-atm}
  c_{\mathsf{atm}}(x)
  =
  \sqrt{\gamma R_g \theta_{\mathsf{atm}}(x)}
  =
  \sqrt{\gamma R_g \left(290 - \tfrac{70}{15000}\,x_3\right)},
  \qquad 0 \le x_3 \le 15000 \, \unit{\meter},
\end{equation}
where $\gamma = 1.4$ is the adiabatic index and
$R_g = 287 \, \unit{\joule\per\kilogram\per\kelvin}$ is the specific gas
constant for dry air.

We model the hot bubble as the moving interior domain $\Omega^-(t)$.
Let $d(x,t)$ denote the signed distance to the moving interface, and
define the (smooth) temperature difference
\begin{equation}\label{explosion-theta-bub}
  \theta_{\mathsf{bub}}(x,t)
  =
  \tfrac{\Delta \theta_{\mathsf{bub}}}{2}
  \left(
    1+\tanh\!\left(-\tfrac{d(x,t)}{\varepsilon_0}\right), 
  \right).
\end{equation}
and the total temperature field
\begin{equation}\label{explosion-theta-total}
  \theta_{\mathsf{phys}}(x,t)
  =
  \theta_{\mathsf{atm}}(x)
  +
  \theta_{\mathsf{bub}}(x,t). 
\end{equation}
The corresponding physical sound speed is modeled by
\begin{equation}\label{explosion-c-phys}
  c_{\mathsf{phys}}(x,t)
  =
  \sqrt{\gamma R_g\,\theta_{\mathsf{phys}}(x,t)}.
\end{equation}
The interface is the rising sphere
\begin{equation}\label{fireball-param-nd}
  z(\alpha,t)
  =
  (0, 0, z_c(t) )^{\top}
  +
  R\,\alpha,
  \qquad
  \alpha\in\mathbb{S}^2,
\end{equation}
with fixed radius $R=1.55$ and initial center height $z_c(0)=2.323$. 
As a simplified early-time model, we prescribe
\begin{equation}\label{fireball-center-nd}
  z_c(t)=2.323+7.69\times 10^{-2}\,t,
  \qquad
  0\le t\le 110,
\end{equation}
which corresponds to a rise from about $232\, \unit{\meter}$ to about
$1.08\, \unit{\kilo\meter}$ over approximately $32$ seconds. 
At the initial burst height, the maximum sound speed 
contrast is $\frac{c_{\max}}{c_{\min}} \approx 2.1$, while 
the characteristic early-time Mach number is $U \approx 0.08$.

To probe the moving refractive inclusion, we prescribe an upward-propagating
Gaussian-modulated plane-wave pulse train,
\begin{equation}\label{fireball-incident-wave}
  \phiinc(x,t)
  =
  \sum_{m=0}^{N_{\mathsf{pulse}}-1}
  \exp\!\left(
    -\tfrac{\bigl(a\cdot x-(t-T_{\mathsf{rep}} m )+x_0\bigr)^2}{\sigma^2}
  \right)
  \cos\!\left(
    2\pi f_0 \bigl(a\cdot x-(t-T_{\mathsf{rep}} m)+x_0\bigr)
  \right),
\end{equation}
where
\[
  a=(0,0,1)^\top,
  \qquad
  x_0=5,
  \qquad
  \sigma=3,
  \qquad
  f_0 = 0.05,
  \qquad
  T_{\mathsf{rep}}=15,
  \qquad
  N_{\mathsf{pulse}}=20.
\]
Thus the incident field is a train of twenty identical planar pulses propagating upward 
from the ground, with a repetition period of approximately 4.4 seconds in dimensional units.

\subsubsection{Numerical simulation using the single-layer solver}

We simulate the rising-fireball problem using the single-layer solver on a 
triangulated spherical interface with  713 triangles. The solution is computed using a time step
$\Delta t = 1$ up to the final nondimensional time $\tmax=110$.
\Cref{fig:fireball-sigma-phi} shows representative snapshots of the computation 
at the two times $t_{\mathsf{phys}}=17.4 \, \unit{\second}$ and $t_{\mathsf{phys}}=32.1 \, \unit{\second}$. 
The scattered wave is visibly refracted by the hot fireball, which acts as a localized 
high-speed spherical inclusion and, in the paraxial geometric--optics approximation, behaves like a defocusing lens.

\begin{figure}[ht]
    \centering

    \begin{subfigure}{0.5\textwidth}
        \centering
        \includegraphics[width=\textwidth]{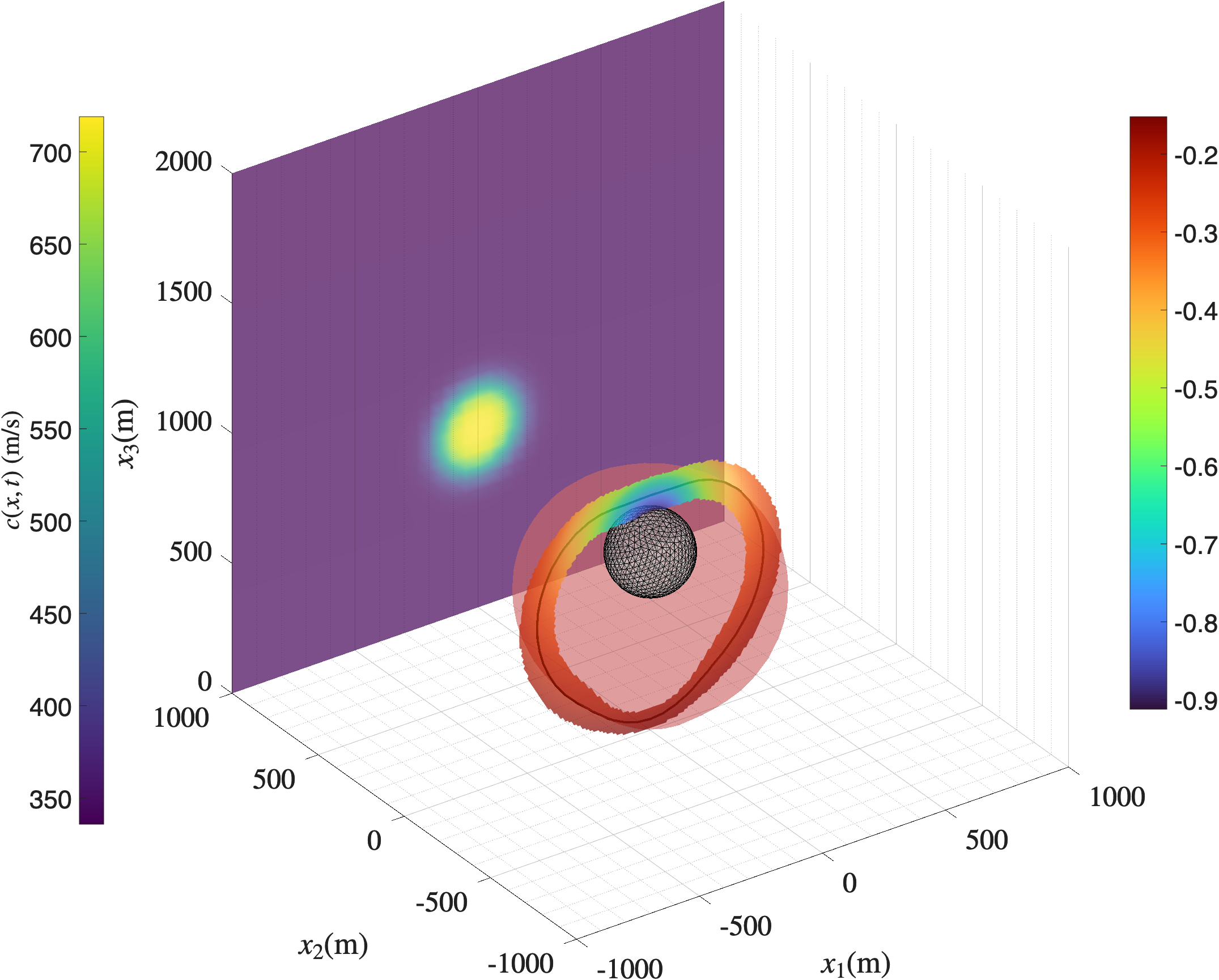}
        \caption{$t_{\mathsf{phys}}=17.4 \  \unit{\second}$}
    \end{subfigure}
    \hfill
    \begin{subfigure}{0.46\textwidth}
        \centering
        \includegraphics[width=\textwidth]{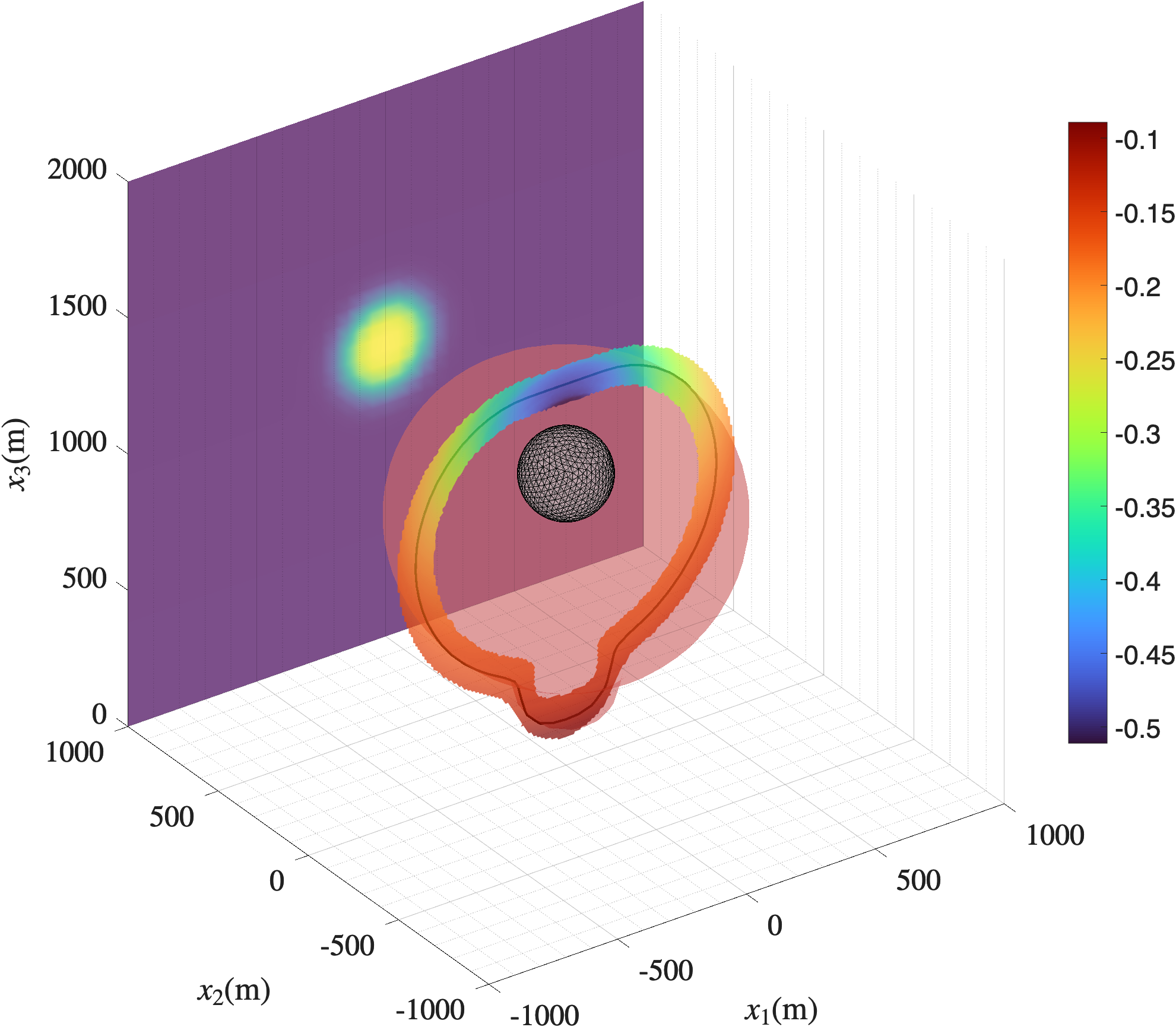}
        \caption{$t_{\mathsf{phys}}=32.1 \  \unit{\second}$}
    \end{subfigure}

\caption{
Representative snapshots of the rising-fireball simulation at two selected times. 
In each panel, the fireball interface $\Gamma(t)$ is shown as the triangulated gray surface. 
The computed reflected wavefront is shown as the translucent red surface, and 
its intersection with the plane $x_2=0$ is shown as the black curve. 
Also displayed is the approximate scattered potential $\tilde{\phi}(x,t)$, defined in \eqref{layer-ansatze-a}, restricted 
to the plane $x_2=0$, plotted only in a local neighborhood of the reflected wavefront. 
For reference, the $x_2=0$ slice of the sound-speed field is rendered on the back face of each panel. 
}
    \label{fig:fireball-sigma-phi}
\end{figure}

To assess this interpretation more quantitatively, we compare the hot-fireball
simulation against a reference ambient-air case in which the same moving spherical
interface is retained but the large temperature 
difference is removed, $\Delta \theta_{\mathsf{bub}} = 0$.
\Cref{fig:fireball-sensor-history} summarizes this comparison in two complementary ways.
The left panel shows the ambient-air analog of the late-time 
snapshot in \Cref{fig:fireball-sigma-phi}, where the reflected front remains 
essentially undistorted in the absence of the hot refractive inclusion.
Meanwhile, the sensor trace in the right panel of \Cref{fig:fireball-sensor-history}
shows both a time shift and a reduction in peak amplitude for the hot-fireball case.
The reflected arrival occurs earlier than in the ambient-air reference case due to
the reduced travel time through the hotter (faster) fireball. At the same time, the
peak signal is smaller due to the spreading of the reflected wave energy.

\begin{figure}[ht]
    \centering

    \begin{subfigure}{0.48\textwidth}
        \centering
        \includegraphics[width=\textwidth]{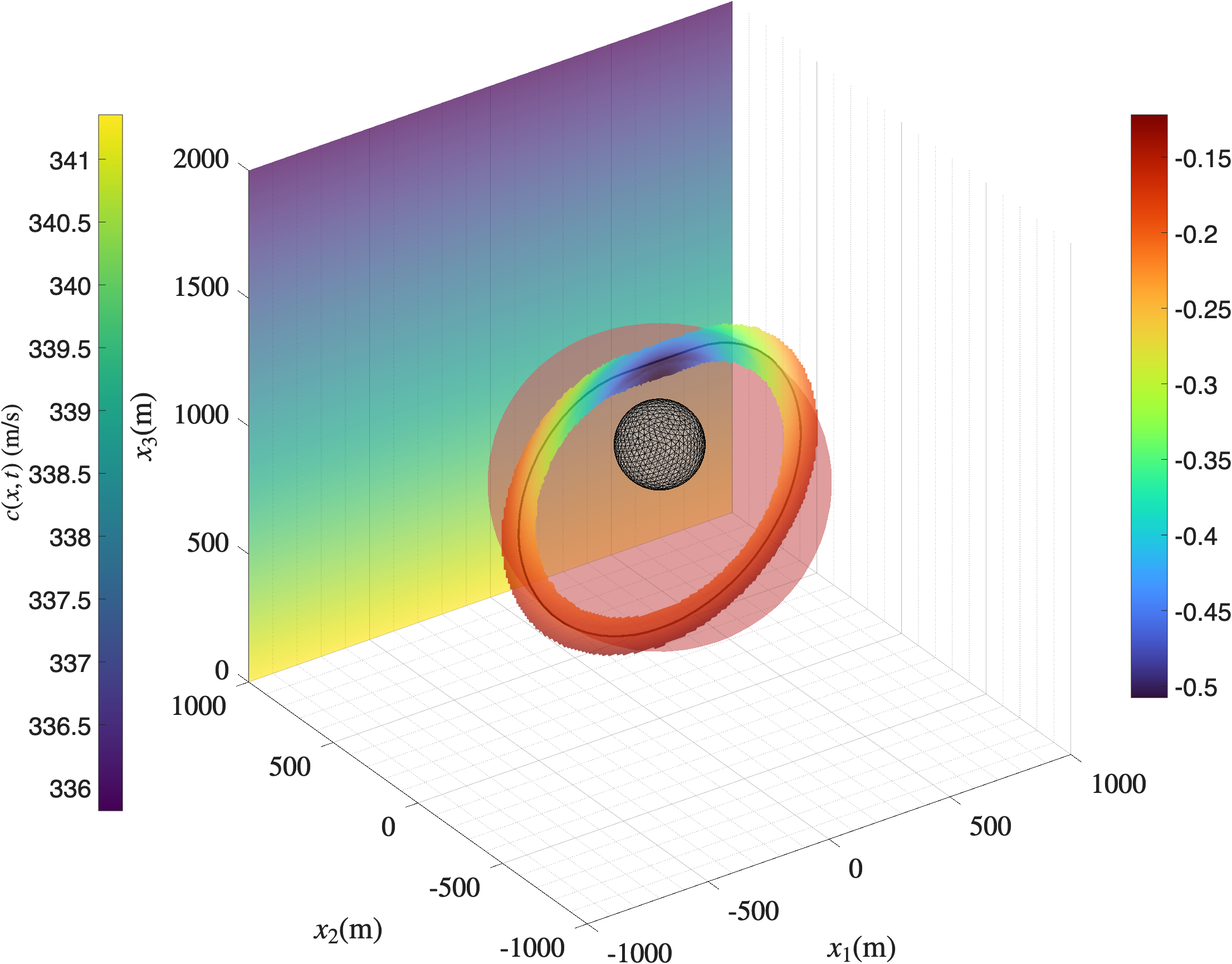}
        \caption{Ambient-air reference case at $t_{\mathsf{phys}}=32.1\, \unit{\second}$.}
    \end{subfigure}
    \hfill
    \begin{subfigure}{0.48\textwidth}
        \centering
        \includegraphics[width=0.7\textwidth]{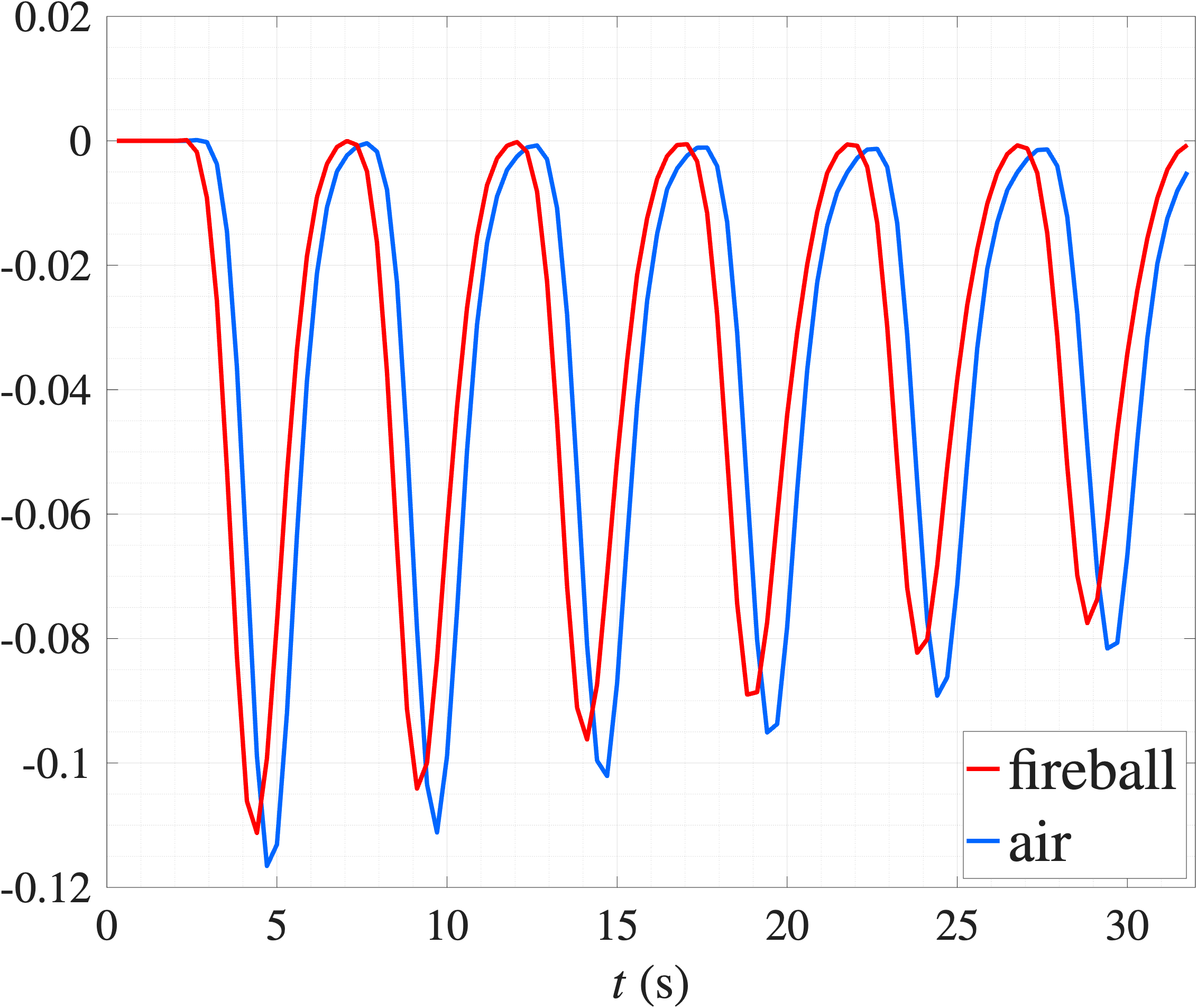}
        \caption{Sensor history at $x_{1,\mathsf{phys}}=1000\, \unit{\meter}$.}
    \end{subfigure}

    \caption{
    Comparison between the hot-fireball simulation and the ambient-air reference case.
    \textbf{Left}: late-time snapshot of the ambient-air reference case, shown in the same
    style as \Cref{fig:fireball-sigma-phi}.    
    \textbf{Right}: sensor history at $x_{1,\mathsf{phys}}=1000\, \unit{\meter}$ for the hot-fireball and ambient-air
    cases. 
    }
    \label{fig:fireball-sensor-history}
\end{figure}

\section{Conclusion}\label{sec:conclusion}

We have developed a boundary integral framework for 
linear wave propagation in heterogeneous media with moving obstacles.
The two main mathematical tools employed---namely, the geometric--optics 
parametrix and 
the variational characterization of the travel-time function---are classical, but their 
combined use here yields, to our knowledge, a novel boundary 
integral method that \emph{avoids volumetric discretization}. 
The parametrix supplies an explicit kernel underlying 
the construction of layer potentials and the derivation 
of the associated boundary integral equations.
Meanwhile, the travel-time function encodes the geometry 
of wave propagation, capturing both refraction due to heterogeneities 
in the medium as well as Doppler effects due to interface motion.

At the discrete level, we construct slab-frozen approximations of the 
backwards light cone and the worldsheet of the interface, and use their 
intersection to approximate the causal interaction set.
This slab-frozen approximation reduces the problem to one with the same 
structural form as the homogeneous fixed-interface case, enabling the use 
of standard time-domain boundary integral methods.
Numerical experiments demonstrate that the method 
captures motion-induced Doppler effects and remains stable up to 
Mach 0.9, as well as refraction in heterogeneous media, including 
propagation around a slow inclusion (gas bubble) and refractive 
defocusing due to a rising hot fireball.

The natural next step is to further exploit the extensive geometric--optics
and boundary integral machinery in the existing literature, including: 
provably stable Galerkin discretizations \cite{Ha2003,Sayas2013}, 
high-order quadrature schemes \cite{SaSc2010,KlBaGr2013,GaGrSi1998,GaGr2004,GiVe2013}, 
efficient algorithms for large-scale computations \cite{BaSa2009,ScLoLu2006,BrKu2001,Rokhlin1990,ErShMi1998}, and
improved computation of travel times and amplitudes beyond caustics \cite{UmTh1987,SaKe1990,ViIvGj1993}. 
Given the broad range of applications in which wave propagation models arise, it is 
essential to develop accurate, robust, and efficient implementations for practical use.

From an application standpoint, one possibility is the
implementation of the proposed method within a more general multiscale
simulation framework.  In this framework, a hybrid Eulerian--Lagrangian scheme may be employed, in which 
coarse scales are treated using Eulerian discretizations while 
fine scales are handled by a Lagrangian component. In \cite{RaSh2020}, a two-scale 
solution strategy of this type was proposed for the compressible Euler equations. 
The fine-scale boundary integral solver in \cite{RaSh2020} relied upon an 
incompressible approximation to the dynamics, producing a model with 
inherent non-locality. 
To maintain consistency with the coarse-scale model, the incompressible 
approximation could be replaced by an adaptation of the present method, 
producing a coupled algorithm that respects finite-speed 
propagation of waves across scales.

\section*{Acknowledgements}
This work was supported by the Mark Kac 
Applied Mathematics Postdoctoral Fellowship at the 
Center for Nonlinear Studies at Los Alamos National Laboratory.


\appendix

\section{Trace of the double-layer potential} \label{app:traceDL}

We compute the one-sided Dirichlet traces of the double-layer potential.
The key issue is the existence and uniqueness of the retarded root
near the endpoint $\tau=t$ when the observation point approaches the
moving interface from either side.
We first record this preliminary fact, and then state the trace formula.

\begin{lemma}[Existence and uniqueness of the near-endpoint retarded root]
\label{lem:near-endpoint-root}

Let $(\theta,\eta)$ be travel-time coordinates in the sense of
Definition \ref{def:characteristic-coordinates}, let
$z\in C^2(\mathbb{S}^2\times[0,\tmax))$, and fix $(\alpha,t)\in\mathbb{S}^2\times(0,\tmax)$.
Assume that the subsonic condition
\eqref{subsonic-condition} holds.

For $\varepsilon>0$, define the exterior offset point
\[
  x_\varepsilon \coloneqq z(\alpha,t)+\varepsilon \NN(\alpha,t),
\]
and the phase function
\begin{equation}\label{g-def-root}
  \Phi_\varepsilon(\tau;\alpha,t)
  \coloneqq
  \theta(t;\tau)-\eta(x_\varepsilon,t;z(\alpha,\tau),\tau).
\end{equation}
Then, for all sufficiently small $\varepsilon>0$, the equation
\[
  \Phi_\varepsilon(\tau;\alpha,t)=0
\]
has a unique simple solution
\[
  \tau_*(\varepsilon;\alpha,t)\in(0,t)
\]
with 
\[
  \tau_*(\varepsilon;\alpha,t)\to t
  \qquad\text{as}\qquad
  \varepsilon\downarrow0.
\]
\end{lemma}

\begin{proof}
\hfill

\vspace{1em}
\noindent
\textbf{Step 1: local expansion near the endpoint.}
Write $\theta = \theta(t;\tau)=t-\tau$ 
and define $K(\alpha,\tau)\coloneqq \kappa(z(\alpha,\tau),\tau)$.
Since $z\in C^2(\mathbb{S}^2\times[0,\tmax))$, Taylor expansion in time about
$(\alpha,t)$ gives
\begin{subequations}\label{z-taylor-group}
\begin{equation}\label{z-taylor-endpoint}
  z(\alpha,\tau)
  =
  z(\alpha,t)-\partial_t z(\alpha,t)\, \theta +\mathcal O(\theta^2),
  \qquad \theta\downarrow0.
\end{equation}
Hence
\begin{equation}\label{xeps-minus-z}
  x_\varepsilon-z(\alpha,\tau)
  =
  \varepsilon \NN(\alpha,t)
  +
  \partial_t z(\alpha,t)\,\theta
  +
  \mathcal O(\theta^2).
\end{equation}
\end{subequations}
The tangential component of $\partial_t z$ changes only the labeling of points
along the surface and does not affect the geometric motion of the interface.
We therefore choose the boundary parameter locally so that, near $(\alpha,t)$,
the trajectory $\tau\mapsto z(\alpha,\tau)$ has no tangential component. Thus
\begin{equation}\label{velocity-normal}
  \partial_t z(\alpha,t)
  =
  V_\NN(\alpha,t)\,\NN(\alpha,t),
  \qquad
  V_\NN(\alpha,t)\coloneqq \partial_t z(\alpha,t)\cdot \NN(\alpha,t).
\end{equation}
Substituting \eqref{velocity-normal} into \eqref{xeps-minus-z}, we obtain
\begin{subequations}\label{xeps-z-expanded-group}
\begin{equation}\label{xeps-minus-z-expanded}
  x_\varepsilon-z(\alpha,\tau)
  =
  \bigl(\varepsilon+V_\NN(\alpha,t)\,\theta\bigr)\NN(\alpha,t)
  +
  \mathcal O(\theta^2).
\end{equation}
Therefore,
\begin{equation}\label{norm-xeps-minus-z}
  |x_\varepsilon-z(\alpha,\tau)|
  =
  \bigl|\varepsilon+V_\NN(\alpha,t)\,\theta\bigr|
  +
  \mathcal O(\theta^2).
\end{equation}
\end{subequations}

In particular, if we restrict to the near-endpoint regime 
$0\le \theta \le M\varepsilon$, with $M>0$ chosen so that
$1+M\,V_\NN(\alpha,t)>0$, then $\varepsilon+V_\NN(\alpha,t)\,\theta>0$, and hence
\begin{equation}\label{xeps-exp}
  |x_\varepsilon-z(\alpha,\tau)|
  =
  \varepsilon+V_\NN(\alpha,t)\,\theta
  +
  \mathcal O(\theta^2).
\end{equation}
Using the local diagonal expansion \eqref{eta-local-diag-expansion}, we obtain
\begin{equation}\label{eta-xeps-local}
  \eta(x_\varepsilon,t;z(\alpha,\tau),\tau)
  =
  K(\alpha,\tau)\,|x_\varepsilon-z(\alpha,\tau)|
  +
  \mathcal O\!\left(\varepsilon^2+\theta^2+\varepsilon\theta\right).
\end{equation}
Substituting \eqref{eta-xeps-local} into \eqref{xeps-exp} yields
\begin{equation}\label{eta-eps-expansion}
  \eta(x_\varepsilon,t;z(\alpha,\tau),\tau)
  =
  K(\alpha,\tau)\,\varepsilon
  +
  K(\alpha,\tau)\,V_\NN(\alpha,t)\,\theta
  +
  \mathcal O\!\left(\varepsilon^2+\theta^2+\varepsilon\theta\right).
\end{equation}

Combining \eqref{g-def-root} and \eqref{eta-eps-expansion}, we obtain
\begin{equation}\label{phi-eps-endpoint-expansion}
  \Phi_\varepsilon(\tau;\alpha,t)
  =
  -K(\alpha,\tau)\,\varepsilon
  +
  \left(
    1
    -
    K(\alpha,\tau)\,V_\NN(\alpha,t)
  \right)\theta
  +
  \mathcal O\!\left(\varepsilon^2+\theta^2+\varepsilon\theta\right).
\end{equation}
Since $z$ and $\kappa$ are smooth, we also have
$K(\alpha,\tau)=K(\alpha,t)+\mathcal O(\theta)$, and therefore
\begin{subequations}\label{phi-Lambda-group}
\begin{equation}\label{phi-eps-endpoint-expansion-frozen}
  \Phi_\varepsilon(\tau;\alpha,t)
  =
  -K(\alpha,t)\,\varepsilon
  +
  \Lambda(\alpha,t)\,\theta
  +
  \mathcal O\!\left(\varepsilon^2+\theta^2+\varepsilon\theta\right).
\end{equation}
where
\begin{equation}\label{Lambda-def}
  \Lambda(\alpha,t)
  \coloneqq
  1
  -
  K(\alpha,t)\,V_\NN(\alpha,t).
\end{equation}
\end{subequations}

\vspace{1em}
\noindent
\textbf{Step 2: existence of a near-endpoint root.}
The subsonic condition \eqref{subsonic-condition} implies
$K(\alpha,t)\,|V_\NN(\alpha,t)|<1$, and therefore
\begin{equation}\label{Lambda-positive}
  \Lambda(\alpha,t)>0.
\end{equation}
Evaluating \eqref{phi-eps-endpoint-expansion-frozen} at $\tau=t$ gives
\begin{equation}\label{phi-eps-at-endpoint}
  \Phi_\varepsilon(t;\alpha,t)
  =
  -\eta(x_\varepsilon,t;z(\alpha,t),t)
  =
  -K(\alpha,t)\,\varepsilon+\mathcal O(\varepsilon^2)
  <0
\end{equation}
for all sufficiently small $\varepsilon>0$.

Next choose $M>0$ such that
\begin{equation}\label{M-choice-growth}
  M\,\Lambda(\alpha,t)>K(\alpha,t)
\qquad\text{and}\qquad
  1+M\,V_\NN(\alpha,t)>0.
\end{equation}
Setting $\theta=M\varepsilon$
in \eqref{phi-eps-endpoint-expansion-frozen} yields
\begin{equation}\label{phi-eps-shifted}
  \Phi_\varepsilon(t-M\varepsilon;\alpha,t)
  =
  \varepsilon\bigl(-K(\alpha,t)+M\Lambda(\alpha,t)\bigr)
  +
  \mathcal O(\varepsilon^2).
\end{equation}
By \eqref{M-choice-growth}, the coefficient of $\varepsilon$ is positive, and hence
\begin{equation}\label{phi-eps-positive}
  \Phi_\varepsilon(t-M\varepsilon;\alpha,t)>0
\end{equation}
for all sufficiently small $\varepsilon>0$.

Combining \eqref{phi-eps-at-endpoint} and \eqref{phi-eps-positive}, we conclude
that $\Phi_\varepsilon(\tau;\alpha,t)$ changes sign on the interval
$[\,t-M\varepsilon,\;t\,]$.
Since $\Phi_\varepsilon$ is continuous in $\tau$, the intermediate value theorem
yields the existence of at least one root
$\tau_*(\varepsilon;\alpha,t)\in(t-M\varepsilon,t)$
such that
$\Phi_\varepsilon(\tau_*(\varepsilon;\alpha,t);\alpha,t)=0$.
In particular,
$\tau_*(\varepsilon;\alpha,t)\to t$ as $\varepsilon\downarrow0$.

\noindent
\textbf{Step 3: uniqueness and simplicity of the root.}
For each fixed $(\alpha,t)$ and $\varepsilon>0$, consider the function
$ \tau \mapsto \Phi_\varepsilon(\tau;\alpha,t)$. 
Since $x_\varepsilon\in\Omega^+(t)$ is fixed and the source point is
$z(\alpha,\tau)\in\Gamma(\tau)$, the non-degeneracy condition
\eqref{monotone-intersection} gives
\begin{equation}\label{Phi-eps-monotone}
  \frac{\mathrm d}{\mathrm d\tau}\Phi_\varepsilon(\tau;\alpha,t)
  =
  -1
  -
  \frac{\mathrm d}{\mathrm d\tau}\eta(x_\varepsilon,t;z(\alpha,\tau),\tau)
  < 0,
  \qquad 0\le \tau \le t.
\end{equation}
Hence $\tau\mapsto \Phi_\varepsilon(\tau;\alpha,t)$ is strictly decreasing on
$[0,t]$.
By Step 2, for all sufficiently small $\varepsilon>0$ there exists at least one
root $\tau_*(\varepsilon;\alpha,t)\in(t-M\varepsilon,t)$
such that $\Phi_\varepsilon(\tau_*(\varepsilon;\alpha,t);\alpha,t)=0$.
The strict monotonicity in \eqref{Phi-eps-monotone} implies that this root is unique.
Moreover, the root is simple since evaluation at 
$\tau=\tau_*(\varepsilon;\alpha,t)$ yields
\[
  \frac{\mathrm d}{\mathrm d\tau}
  \Phi_\varepsilon(\tau_*(\varepsilon;\alpha,t);\alpha,t)<0. 
\]
\end{proof}

The preceding lemma identifies the unique retarded source time that controls
the singular part of the kernel as the observation point approaches the moving
boundary from the exterior. We now use this fact to isolate the jump term in
the double-layer potential and obtain the exterior trace formula.

\begin{proposition}[Exterior trace of the moving-surface double-layer potential with constant speed]
\label{prop:DL-trace-const-speed}
\hfill
Let $c_0>0$ be constant, let
$z\in C^2(\mathbb{S}^2\times[0,\tmax))$, and assume that the normal velocity 
$V_\NN(\alpha,t) = \NN(\alpha,t) \cdot \partial_t z(\alpha,t)$
satisfies the subsonic condition
\[
  |V_\NN(\alpha,t)|<c_0,
  \qquad
  (\alpha,t)\in\mathbb{S}^2\times(0,\tmax).
\]
Let
\[
  \Gop_{c_0}(x,t;y,\tau)
  \coloneqq
  \frac{
    \delta\!\left(
      t-\tau-c_0^{-1} {|x-y|}
    \right)
  }{
    4\pi |x-y|
  }
\]
denote the retarded Green's function for the operator $\tfrac{1}{c_0^2}\partial_t^2-\Delta_x$.
For a causal density $\mu:\mathbb{S}^2\times(0,\tmax)\to\mathbb{R}$,
define the double-layer potential
\[
  \DopDL[\mu](x,t)
  \coloneqq
  \int_0^t\int_{\mathbb{S}^2}
    \partial_{\NN_y}
    \Gop_{c_0}(x,t;z(\beta,\tau),\tau)\,
    \mu(\beta,\tau)\,
    \dd S_y(\beta,\tau)\,\dd\tau,
  \qquad
  (x,t)\in\Omega^+(t)\times(0,\tmax).
\]
Then, for each $(\alpha,t)\in\mathbb{S}^2\times(0,\tmax)$, the exterior trace
\[
  \gamma^+\DopDL[\mu](\alpha,t)
  \coloneqq
  \lim_{\varepsilon\downarrow0}
  \DopDL[\mu]\bigl(z(\alpha,t)+\varepsilon\NN(\alpha,t),t\bigr)
\]
exists and satisfies
\begin{subequations}\label{DL-trace-const-speed}
\begin{equation}\label{DL-trace-const-speed-main}
  \gamma^+\DopDL[\mu](\alpha,t)
  =
  \tfrac{1}{2} \left( 1 + \tfrac{V_\NN(\alpha,t)^2}{c_0^2 - V_\NN(\alpha,t)^2} \right)
   \, \mu(\alpha,t)
  +
  \Dop[\mu](\alpha,t),
\end{equation}
where
\begin{equation}\label{DL-trace-const-speed-reg}
  \Dop[\mu](\alpha,t)
  \coloneqq
  \mathrm{p.v.}\!\int_0^t\int_{\mathbb{S}^2}
    \partial_{\NN_y}
    \Gop_{c_0}(z(\alpha,t)^+,t;z(\beta,\tau),\tau)\,
    \mu(\beta,\tau)\,
    \dd S_y(\beta,\tau)\,\dd\tau.
\end{equation}
\end{subequations}
\end{proposition}

\begin{proof}
Fix $(\alpha,t)\in\mathbb{S}^2\times(0,\tmax)$ and, for $\varepsilon>0$, set
\[
  x_\varepsilon \coloneqq z(\alpha,t)+\varepsilon \NN(\alpha,t).
\]
We study the limit as $\varepsilon\downarrow0$ of
\[
  \DopDL[\mu](x_\varepsilon,t)
  =
  \int_0^t\int_{\mathbb{S}^2}
    \partial_{\NN_y}
    \Gop_{c_0}(x_\varepsilon,t;z(\beta,\tau),\tau)\,
    \mu(\beta,\tau)\,
    \dd S_y(\beta,\tau)\,\dd\tau.
\]

\noindent
\textbf{Step 1: decomposition into local and regular parts.}
Choose a smooth cutoff $\chi(\beta,\tau)\in C_c^\infty(\mathbb{S}^2\times[0,t])$
supported in a sufficiently small neighborhood of $(\alpha,t)$, with
$\chi(\alpha,t)=1$.
We decompose
\begin{subequations}\label{DL-local-reg-split-const}
\begin{equation}\label{DL-local-reg-split-const-a}
  \DopDL[\mu](x_\varepsilon,t)
  =
  I_\varepsilon^{\loc}(\alpha,t)
  +
  I_\varepsilon^{\reg}(\alpha,t),
\end{equation}
where
\begin{equation}\label{DL-local-reg-split-const-b}
  I_\varepsilon^{\loc}(\alpha,t)
  \coloneqq
  \int_0^t\int_{\mathbb{S}^2}
    \chi(\beta,\tau)\,
    \partial_{\NN_y}
    \Gop_{c_0}(x_\varepsilon,t;z(\beta,\tau),\tau)\,
    \mu(\beta,\tau)\,
    \dd S_y(\beta,\tau)\,\dd\tau,
\end{equation}
and
\begin{equation}\label{DL-local-reg-split-const-c}
  I_\varepsilon^{\reg}(\alpha,t)
  \coloneqq
  \int_0^t\int_{\mathbb{S}^2}
    \bigl(1-\chi(\beta,\tau)\bigr)\,
    \partial_{\NN_y}
    \Gop_{c_0}(x_\varepsilon,t;z(\beta,\tau),\tau)\,
    \mu(\beta,\tau)\,
    \dd S_y(\beta,\tau)\,\dd\tau.
\end{equation}
\end{subequations}

The cutoff localizes all singular behavior into
$I_\varepsilon^{\loc}(\alpha,t)$.
Since $1-\chi$ vanishes in a neighborhood of $(\alpha,t)$, the kernel in
$I_\varepsilon^{\reg}(\alpha,t)$ is smooth for all sufficiently small $\varepsilon>0$, and therefore
\begin{equation}\label{Ireg-limit-const}
  \lim_{\varepsilon\downarrow0} I_\varepsilon^{\reg}(\alpha,t)
  =
  \int_0^t\int_{\mathbb{S}^2}
    \bigl(1-\chi(\beta,\tau)\bigr)\,
    \partial_{\NN_y}
    \Gop_{c_0}(z(\alpha,t)^+,t;z(\beta,\tau),\tau)\,
    \mu(\beta,\tau)\,
    \dd S_y(\beta,\tau)\,\dd\tau.
\end{equation}
Thus the trace computation reduces to the analysis of the localized term
$I_\varepsilon^{\loc}(\alpha,t)$ near $(\beta,\tau)=(\alpha,t)$.

\vspace{1em}
\noindent
\textbf{Step 2: local coordinates near the retarded root.}
Choose local coordinates $u=(u_1,u_2)\in\mathbb{R}^2$ on the tangent plane to
$\Gamma(t)$ at $z(\alpha,t)$, and choose the local parametrization so that the
surface velocity at $(\alpha,t)$ is purely normal:
\[
  \partial_t z(\alpha,t)=V_\NN(\alpha,t)\,\NN(\alpha,t).
\]
After a rigid translation and rotation, we may assume
\[
  z(\alpha,t)=0,
  \qquad
  \NN(\alpha,t)=e_3,
  \qquad
  x_\varepsilon=(0,0,\varepsilon).
\]

Writing $\theta = t-\tau$, 
the moving surface admits the local expansion
\begin{subequations}\label{const-speed-local-geometry}
\begin{equation}\label{const-speed-surface-local}
  y(u,\tau)
  =
  \bigl(u_1,u_2,-V_\NN(\alpha,t)\,\theta\bigr)
  +
  \mathcal O\!\left(|u|^2+\theta^2\right),
\end{equation}
and therefore
\begin{equation}\label{const-speed-xeps-minus-y}
  x_\varepsilon-y(u,\tau)
  =
  \bigl(-u_1,-u_2,\varepsilon+V_\NN(\alpha,t)\,\theta\bigr)
  +
  \mathcal O\!\left(|u|^2+\theta^2\right).
\end{equation}
Define
\begin{equation}\label{const-speed-Reps}
  R_\varepsilon(u,\theta)
  \coloneqq
  \sqrt{|u|^2+\bigl(\varepsilon+V_\NN(\alpha,t)\,\theta\bigr)^2}.
\end{equation}
Then
\begin{equation}\label{const-speed-distance-expansion}
  |x_\varepsilon-y(u,\tau)|
  =
  R_\varepsilon(u,\theta)
  +
  \mathcal O\!\left(|u|^2+\theta^2\right).
\end{equation}
\end{subequations}

In these local coordinates, the retarded phase for the constant-speed kernel is
\begin{equation}\label{const-speed-phase-local}
  \Phi_\varepsilon(u,\theta)
  \coloneqq
  \theta-c_0^{-1}R_\varepsilon(u,\theta).
\end{equation}
Thus the singular part of the localized kernel is governed by
\begin{equation}\label{const-speed-frozen-kernel}
  \partial_{\NN_y}
  \left[
    \frac{
      \delta\!\left(\Phi_\varepsilon(u,\theta)\right)
    }{
      4\pi\,R_\varepsilon(u,\theta)
    }
  \right].
\end{equation}

For each fixed $u$ sufficiently close to $0$, the same argument as in
\Cref{lem:near-endpoint-root} applies to the phase
\[
  \Phi_\varepsilon(u,\theta)
  \coloneqq
  \uptheta-c_0^{-1}R_\varepsilon(u,\theta). 
\]
In the constant-speed case, we have
\[
  \eta(x,t;y,\tau)=c_0^{-1}|x-y|,
  \qquad
  \kappa(y,\tau)=c_0^{-1},
\]
and for all sufficiently small
$\varepsilon>0$, the equation
\begin{equation}\label{const-speed-retarded-root-eqn}
  \Phi_\varepsilon(u,\theta)=0
\end{equation}
has a unique simple root
\[
  \theta=\theta_*(u,\varepsilon) > 0,
\]
with
\[
  \theta_*(u,\varepsilon)\to0
  \qquad\text{as}\qquad
  (u,\varepsilon)\to(0,0).
\]
Equivalently,
\begin{equation}\label{const-speed-root-relation}
  \theta_*(u,\varepsilon)
  =
  c_0^{-1}R_\varepsilon\bigl(u,\theta_*(u,\varepsilon)\bigr).
\end{equation}

Finally, the local surface Jacobian satisfies
\begin{equation}\label{const-speed-surface-jacobian}
  \dd S_y(\beta,\tau)\,\dd\tau
  =
  J(u,\theta)\,\dd u\,\dd\theta,
  \qquad
  J(u,\theta)=1+\mathcal O(|u|+\theta),
\end{equation}
after normalizing the local coordinates at $(\alpha,t)$.

\vspace{1em}
\noindent
\textbf{Step 3: reduction to the frozen integral.}
By \eqref{const-speed-local-geometry} and \eqref{const-speed-surface-jacobian},
the localized term $I_\varepsilon^{\loc}(\alpha,t)$ may be written in the form
\[
  I_\varepsilon^{\loc}(\alpha,t)
  =
  \int_0^{\theta_0}\int_{U_0}
    \chi_0(u,\theta)\,
    \partial_{\NN_y}
    \left[
      \frac{
        \delta\!\left(\Phi_\varepsilon(u,\theta)\right)
      }{
        4\pi\,R_\varepsilon(u,\theta)
      }
    \right]\,
    \upmu(u,\theta)\,
    J(u,\theta)\,
    \dd u\,\dd\theta,
\]
where $U_0\subset\mathbb{R}^2$ is a sufficiently small neighborhood of $u=0$,
$\theta_0>0$ is sufficiently small, $\chi_0(u,\theta)$ is the cutoff $\chi$
expressed in the local coordinates, and
\[
  \upmu(u,\theta)\coloneqq \mu(\beta(u),t-\theta).
\]

Since $\upmu$ is continuous and $J(u,\theta)=1+\mathcal O(|u|+\theta)$, we have
\[
  \upmu(u,\theta)\,J(u,\theta)
  =
  \mu(\alpha,t)+\mathcal O(|u|+\theta).
\]
Moreover, in the local coordinates chosen above, the source normal derivative
agrees to leading order with differentiation in the $y_3$-direction, and since
the frozen kernel depends on $y_3$ only through the combination
$\varepsilon-y_3=\varepsilon+V_\NN(\alpha,t)\theta$, we have 
$ \partial_{\NN_y} = -\partial_\varepsilon$ at leading order on the singular part of the kernel.
It follows that
\begin{equation}\label{const-speed-local-reduction}
  I_\varepsilon^{\loc}(\alpha,t)
  =
  \mu(\alpha,t)\,L_\varepsilon(\alpha,t)
  +
  o(1),
\end{equation}
where
\begin{equation}\label{const-speed-Leps}
  L_\varepsilon(\alpha,t)
  \coloneqq
  -\int_0^{\theta_0}\int_{U_0}
    \chi_0(u,\theta)\,
    \partial_\varepsilon
    \left[
      \frac{
        \delta\!\left(\Phi_\varepsilon(u,\theta)\right)
      }{
        4\pi\,R_\varepsilon(u,\theta)
      }
    \right]\,
    \dd u\,\dd\theta .
\end{equation}
The jump term is determined by the limit of the scalar model integral
$L_\varepsilon(\alpha,t)$.

\vspace{1em}
\noindent
\textbf{Step 4: reduction to the retarded root.}
Using \eqref{const-speed-root-relation}, we may write
\[
  \delta\!\left(\Phi_\varepsilon(u,\theta)\right)
  =
  \frac{
    \delta\!\left(\theta-\theta_*(u,\varepsilon)\right)
  }{
    \left|
      \partial_\theta\Phi_\varepsilon\bigl(u,\theta_*(u,\varepsilon)\bigr)
    \right|
  }.
\]
Substituting this into \eqref{const-speed-Leps} and integrating in $\uptheta$
gives
\begin{equation}\label{const-speed-Leps-root}
  L_\varepsilon(\alpha,t)
  =
  -\int_{U_0}
    \chi_0\bigl(u,\theta_*(u,\varepsilon)\bigr)\,
    \partial_\varepsilon
    \left[
      \frac{
        1
      }{
        4\pi\,
        R_\varepsilon\bigl(u,\theta_*(u,\varepsilon)\bigr)\,
        \left|
          \partial_\theta \Phi_\varepsilon\bigl(u,\theta_*(u,\varepsilon)\bigr)
        \right|
      }
    \right]\,
    \dd u .
\end{equation}

We now compute the retarded root explicitly.  Squaring
\eqref{const-speed-root-relation} yields
\[
  c_0^2\;\theta_*^2
  =
  |u|^2+\bigl(\varepsilon+V_\NN(\alpha,t)\,\theta_*\bigr)^2,
\]
that is,
\[
  \bigl(c_0^2-V_\NN(\alpha,t)^2\bigr)\theta_*^2
  -2V_\NN(\alpha,t)\varepsilon\,\theta_*
  -\bigl(|u|^2+\varepsilon^2\bigr)
  =
  0.
\]
Since $|V_\NN(\alpha,t)|<c_0$, the positive root is
\begin{equation}\label{const-speed-root-explicit}
  \theta_*(u,\varepsilon)
  =
  \frac{
    V_\NN(\alpha,t)\,\varepsilon
    +
    \sqrt{
      c_0^2\varepsilon^2+
      \bigl(c_0^2-V_\NN(\alpha,t)^2\bigr)|u|^2
    }
  }{
    c_0^2-V_\NN(\alpha,t)^2
  }.
\end{equation}
In particular,
\begin{equation}\label{const-speed-root-size}
  \theta_*(u,\varepsilon)=\mathcal O(|u|+\varepsilon). 
\end{equation}

Next, differentiating \eqref{const-speed-phase-local} with respect to
$\uptheta$ gives
\begin{equation}\label{const-speed-phase-derivative}
  \partial_\theta \Phi_\varepsilon(u,\theta)
  =
  1
  -
  \frac{
    V_\NN(\alpha,t)\bigl(\varepsilon+V_\NN(\alpha,t)\theta\bigr)
  }{
    c_0\,R_\varepsilon(u,\theta)
  }.
\end{equation}
Evaluating at $\theta=\theta_*(u,\varepsilon)$ and using
\eqref{const-speed-root-relation}, we obtain
\begin{equation}\label{const-speed-phase-derivative-root}
  \partial_\theta \Phi_\varepsilon\bigl(u,\theta_*(u,\varepsilon)\bigr)
  =
  1
  -
  \frac{
    V_\NN(\alpha,t)\bigl(\varepsilon+V_\NN(\alpha,t)\theta_*(u,\varepsilon)\bigr)
  }{
    c_0^2\,\theta_*(u,\varepsilon)
  }.
\end{equation}
Evaluating \eqref{const-speed-root-relation} at $u=0$ 
and using \eqref{const-speed-Reps}, we obtain
\[
  \varepsilon+V_\NN(\alpha,t)\theta_*(0,\varepsilon)
  =
  c_0\,\theta_*(0,\varepsilon),
\]
and therefore
\[
  \partial_\theta \Phi_\varepsilon\bigl(0,\theta_*(0,\varepsilon)\bigr)
  =
  1-\frac{V_\NN(\alpha,t)}{c_0}.
\]

\vspace{1em}
\noindent
\textbf{Step 5: evaluation of the integral.}
We now simplify the quantity inside \eqref{const-speed-Leps-root} exactly.
Write $V \coloneqq V_\NN(\alpha,t)$.
Using \eqref{const-speed-phase-derivative} and
\eqref{const-speed-root-relation}, we have
\[
  R_\varepsilon\bigl(u,\theta_*(u,\varepsilon)\bigr)\,
  \partial_\theta \Phi_\varepsilon\bigl(u,\theta_*(u,\varepsilon)\bigr)
  =
  c_0\,\theta_*(u,\varepsilon)
  -
  \frac{
    V\bigl(\varepsilon+V\,\theta_*(u,\varepsilon)\bigr)
  }{
    c_0
  }.
\]
Substituting the explicit formula \eqref{const-speed-root-explicit} for
$\theta_*(u,\varepsilon)$, a direct simplification yields
\begin{equation}\label{const-speed-exact-product}
  R_\varepsilon\bigl(u,\theta_*(u,\varepsilon)\bigr)\,
  \partial_\theta \Phi_\varepsilon\bigl(u,\theta_*(u,\varepsilon)\bigr)
  =
  \frac{1}{c_0}
  \sqrt{
    c_0^2\varepsilon^2+
    \bigl(c_0^2-V^2\bigr)|u|^2
  }.
\end{equation}

Therefore \eqref{const-speed-Leps-root} becomes
\begin{align}
  L_\varepsilon(\alpha,t)
  &=
  -\frac{c_0}{4\pi}
  \int_{U_0}
    \chi_0\bigl(u,\theta_*(u,\varepsilon)\bigr)\,
    \partial_\varepsilon
    \left[
      \frac{
        1
      }{
        \sqrt{
          c_0^2\varepsilon^2+
          \bigl(c_0^2-V^2\bigr)|u|^2
        }
      }
    \right]\,
    \dd u \notag\\
  &=
  \frac{c_0^3}{4\pi}
  \int_{U_0}
    \chi_0\bigl(u,\theta_*(u,\varepsilon)\bigr)\,
    \frac{
      \varepsilon
    }{
      \bigl(
        c_0^2\varepsilon^2+
        \bigl(c_0^2-V^2\bigr)|u|^2
      \bigr)^{3/2}
    }\,
    \dd u .
    \label{const-speed-Leps-poisson}
\end{align}
Since $\theta_*(u,\varepsilon)\to0$ as $(u,\varepsilon)\to(0,0)$ and
$\chi_0(0,0)=1$, we have
\[
  \chi_0\bigl(u,\theta_*(u,\varepsilon)\bigr)\to 1
\]
on the region where the kernel concentrates.
Employing the change of variables
\[
  u=\frac{c_0\varepsilon}{\sqrt{c_0^2-V^2}}\,w, 
\]
\eqref{const-speed-Leps-poisson} becomes
\[
  L_\varepsilon(\alpha,t)
  =
  \frac{c_0^2}{4\pi(c_0^2-V^2)}
  \int_{W_\varepsilon}
    \chi_0\!\left(
      \tfrac{c_0\varepsilon}{\sqrt{c_0^2-V^2}}\,w,
      \theta_*\!\left(\tfrac{c_0\varepsilon}{\sqrt{c_0^2-V^2}}\,w,\varepsilon\right)
    \right)
    \frac{1}{(1+|w|^2)^{3/2}}
    \,\dd w,
\]
where
\[
  W_\varepsilon
  \coloneqq
  \left\{
    w\in\mathbb{R}^2:
    \frac{c_0\varepsilon}{\sqrt{c_0^2-V^2}}\,w\in U_0
  \right\}.
\]
As $\varepsilon\downarrow0$, the dominated convergence theorem gives
\[
  L_\varepsilon(\alpha,t)
  \to
  \frac{c_0^2}{4\pi(c_0^2-V^2)}
  \int_{\mathbb{R}^2}
    \frac{1}{(1+|w|^2)^{3/2}}
    \,\dd w .
\]
We conclude that
\begin{equation}\label{const-speed-Leps-limit}
  \lim_{\varepsilon\downarrow0}L_\varepsilon(\alpha,t)
  =
  \frac{1}{2\left(1-\frac{V_\NN(\alpha,t)^2}{c_0^2}\right)}.
\end{equation}
Combining this with \eqref{const-speed-local-reduction}, we obtain
\[
  \lim_{\varepsilon\downarrow0}I_\varepsilon^{\loc}(\alpha,t)
  =
  \frac{
    1
  }{
    2\left(1-\frac{V_\NN(\alpha,t)^2}{c_0^2}\right)
  }\,
  \mu(\alpha,t).
\]
This completes the proof.

\end{proof}

\Cref{prop:DL-trace-const-speed} shows that the jump coefficient is determined
entirely by the leading local cone geometry of the kernel near the space-time
diagonal.
For the general parametrix kernel, the local expansion \eqref{eta-local-diag-expansion} for $\eta$ and
the diagonal normalization \eqref{A-normalization-st} for the amplitude
show that, at the source point $(z(\alpha,t),t)$, the parametrix has the same leading singular 
structure as the constant-speed kernel with local sound speed $C(\alpha,t) = c(z(\alpha,t),t)$. 
Since the jump term depends only on this leading local cone model,
\Cref{prop:DL-trace-const-speed} applies with $c_0$ replaced by
$C(\alpha,t)$.

\begin{proposition}[Exterior trace of the double-layer potential]
\label{prop:DL-trace-moving}

Let $(\theta,\eta)$ be travel-time coordinates in the sense of
Definition \ref{def:characteristic-coordinates}, let
$z\in C^2(\mathbb{S}^2\times[0,\tmax))$, and assume that the
subsonic condition \eqref{subsonic-condition} holds.
Assume also that $\amp(x,t;y,\tau)$ is smooth away from the diagonal $x=y$.
Define
\[
  \DopDL[\mu](x,t)
  =
  \int_0^t \int_{\mathbb{S}^2}
    \partial_{\NN_y}
    \Big[
      \amp(x,t;z(\beta,\tau),\tau)\,
      \Gopo\big(
        \eta(x,t;z(\beta,\tau),\tau),
        \theta(t;\tau)
      \big)
    \Big]\,
    \mu(\beta,\tau)\,
    \dd S_y(\beta,\tau)\,\dd\tau,
\]
where
\[
  \partial_{\NN_y}
  =
  \NN(\beta,\tau)\cdot \nabla_y,
  \qquad
  \Gopo(r,t) = \frac{\delta(t-r)}{4 \pi r}.  
\]
Then, for each $(\alpha,t)\in\mathbb{S}^2\times(0,\tmax)$, the exterior trace
\[
  \gamma^+\DopDL[\mu](\alpha,t)
  \coloneqq
  \lim_{\varepsilon\downarrow0}
  \DopDL[\mu]\bigl(z(\alpha,t)+\varepsilon \NN(\alpha,t),t\bigr)
\]
exists and satisfies
\begin{subequations}\label{DL-trace}
\begin{equation}\label{DL-trace-general-form}
  \gamma^+\DopDL[\mu](\alpha,t)
  =
  \tfrac12 \lambda(\alpha,t) \,
  \mu(\alpha,t)
  +
  \Dop[\mu](\alpha,t),
\end{equation}
where
\begin{equation}\label{lambda-def}
  \lambda(\alpha,t)
  \coloneqq
  1 
  + 
  \frac{V_\NN(\alpha,t)^2}
  {
  C(\alpha,t)^2-V_\NN(\alpha,t)^2
  },
  \qquad
  C(\alpha,t)\coloneqq c(z(\alpha,t),t),
\end{equation}
and the regular boundary operator is
\begin{equation}\label{Dop-regular-trace}
  \Dop[\mu](\alpha,t)
  \coloneqq
  \mathrm{p.v.}\!\int_0^t\int_{\mathbb{S}^2}
    \partial_{\NN_y}
    \Big[
      A(\alpha,t;\beta,\tau)\,
      \Gopo\big(
        \upeta(\alpha,t;\beta,\tau),
        \theta(t;\tau)
      \big)
    \Big]\,
    \mu(\beta,\tau)\,
    \dd S_y(\beta,\tau)\,\dd\tau .
\end{equation}
\end{subequations}
\end{proposition}

\section{Newton method for geometric ray computation}
\label{subsec:newton-ray-continuous}

In this section, we describe a Newton iteration scheme for computing an
approximate stationary ray of the travel-time functional. The resulting
iteration provides a practical correction to the simple chord approximation \eqref{eta-chord} in
settings where refraction cannot be neglected.

We seek a stationary curve of the travel-time functional
\begin{equation}\label{F-functional}
  \mathcal{F}[\gamma]
  =
  \int_0^1
  \kappa(\gamma(s))\,|\dot\gamma(s)|
  \,\dd s,
  \qquad
  \kappa=\frac{1}{c},
\end{equation}
subject to the endpoint constraints \eqref{rays-bc}. The corresponding
Euler--Lagrange equation is
\begin{equation}\label{EL-cont}
  \frac{\mathrm{d}}{\mathrm{d}s}
  \bigl(\kappa(\gamma)\,\tau_\gamma\bigr)
  =
  |\dot\gamma|\,\nabla\kappa(\gamma),
  \qquad
  \tau_\gamma=\frac{\dot\gamma}{|\dot\gamma|}.
\end{equation}

Since the functional $\mathcal F$ is invariant under reparameterization,
the Euler--Lagrange equation \eqref{EL-cont} is degenerate in the tangential
direction. We therefore restrict the Newton correction to normal variations.
Let
\[
  n_\gamma(s)=\frac{\dot\gamma(s)^\perp}{|\dot\gamma(s)|}
\]
be the unit normal along $\gamma$.  Define the vector residual
\begin{equation}\label{vector-residual}
  \mathcal R(\gamma)
  \coloneqq
  \frac{\mathrm d}{\mathrm ds}\bigl(\kappa(\gamma)\tau_\gamma\bigr)
  -
  |\dot\gamma|\,\nabla\kappa(\gamma),
  \qquad
  \tau_\gamma=\frac{\dot\gamma}{|\dot\gamma|}, 
\end{equation}
and its normal projection
\begin{equation}\label{scalar-residual-cont}
  \mathcal R_\perp(\gamma)\coloneqq n_\gamma\cdot \mathcal R(\gamma). 
\end{equation}
A minimizing ray is characterized by  $\mathcal{R}_\perp(\gamma)=0$. 

We seek a Newton correction of the form
\begin{equation}\label{normal-ansatz-cont}
  \gamma^{(k+1)}
  =
  \gamma^{(k)}+a^{(k)}(s)\,n_{\gamma^{(k)}}(s),
  \qquad
  a^{(k)}(0)=a^{(k)}(1)=0,
\end{equation}
where $a^{(k)}$ is a scalar function.
In the frozen-normal approximation, the normal field
$n_{\gamma^{(k)}}$ is held fixed during a single 
Newton step, i.e., $\delta \mathcal{R}_\perp(\gamma) \approx  n_{\gamma^{(k)}} \cdot \delta \mathcal{R}(\gamma)$. 

A direct linearization of \eqref{scalar-residual-cont} under the variation
$\delta\gamma=a\,n_\gamma$ yields
\begin{equation}\label{DRperp-explicit}
  D\mathcal R_\perp(\gamma)[a\,n_\gamma]
  =
  \frac{\mathrm d}{\mathrm ds}
  \left(
    \frac{\kappa(\gamma)}{|\dot\gamma|}\,a'
  \right)
  +
  \chi\,
  (n_\gamma\cdot\nabla\kappa(\gamma))\,a
  -
  |\dot\gamma|\,q_\gamma\,a,
\end{equation}
where
\[
  \chi(s)\coloneqq \tau_\gamma'(s)\cdot n_\gamma(s)
\]
is the signed curvature of the current curve, and
\[
  q_\gamma(s)\coloneqq
  n_\gamma(s)^\top \nabla^2\kappa(\gamma(s))\,n_\gamma(s).
\]
Here the frozen-normal approximation amounts to discarding additional
terms involving derivatives of $n_\gamma$, thereby retaining a scalar
second-order elliptic operator.

Hence the Newton equation
\[
  D\mathcal R_\perp(\gamma^{(k)})[a^{(k)}n_{\gamma^{(k)}}]
  =
  -\mathcal R_\perp(\gamma^{(k)})
\]
takes the form
\begin{equation}\label{newton-bvp-final}
  \frac{\mathrm d}{\mathrm ds}
  \left(
    A_\gamma(s)\,a'
  \right)
  -
  B_\gamma(s)\,a
  =
  -\mathcal R_\perp(\gamma),
  \qquad
  a(0)=a(1)=0,
\end{equation}
with
\[
  A_\gamma(s)=\frac{\kappa(\gamma(s))}{|\dot\gamma(s)|},
  \qquad
  B_\gamma(s)=
  |\dot\gamma(s)|\,q_\gamma(s)
  -
  \chi(s)\,(n_\gamma(s)\cdot\nabla\kappa(\gamma(s))).
\]
In practice, the curvature term proportional to $\chi$ may be dropped to
improve conditioning, giving the approximation
\[
  B_\gamma(s)\approx |\dot\gamma(s)|\,q_\gamma(s).
\]

Given the solution $a^{(k)}$ of \eqref{newton-bvp-final}, we update
\begin{equation}
  \gamma^{(k+1)}
  =
  \gamma^{(k)}
  +
  \omega^{(k)} a^{(k)} n_{\gamma^{(k)}},
\end{equation}
where $\omega^{(k)}\in(0,1]$ is chosen (e.g., by backtracking line search) so that
$\|\mathcal R_\perp(\gamma^{(k+1)})\|
  <
  \|\mathcal R_\perp(\gamma^{(k)})\|$.
The iteration terminates once the residual norm falls below a prescribed
absolute/relative tolerance.


\bibliographystyle{abbrvnatmod}
\bibliography{references}

\end{document}